\documentclass{amsart}
\usepackage{graphicx}
\usepackage{amssymb}
\usepackage{array}
\usepackage{enumerate}

\newtheorem{thm}{Theorem}[section]
\newtheorem{lem}[thm]{Lemma}

\theoremstyle{definition}

\theoremstyle{remark}

\numberwithin{equation}{section}

\begin{document}

\title[Fast methods to compute the Riemann zeta function]{Fast methods to compute the Riemann zeta function}
\author[G.A. Hiary]{Ghaith Ayesh Hiary}
\thanks{Preparation of this material is partially supported by the National Science Foundation under agreements No.  DMS-0757627 (FRG grant) and DMS-0635607 (while at the Institute for Advanced Study). This material is based on the author's PhD thesis.}
\address{Pure Mathematics, University of Waterloo, 200 University Ave West, Waterloo, Ontario, Canada, N2L 3G1.}
\email{hiaryg@gmail.com}

\begin{abstract}
The Riemann zeta function on the critical line can be computed  using a straightforward application of the Riemann-Siegel formula, Sch\"onhage's method, or Heath-Brown's method. The complexities of these methods have exponents 1/2, 3/8, and 1/3 respectively. In this article, three new fast and potentially practical methods to compute zeta are presented. One method is very simple. Its complexity has exponent 2/5. A second method relies on this author's algorithm to compute quadratic exponential sums. Its complexity has exponent 1/3. The third method, which is our main result, employs an algorithm developed here to compute cubic exponential sums with a small cubic coefficient. Its complexity has exponent 4/13 (approximately, 0.307). 
\end{abstract}

\maketitle

\section{Introduction}
The Riemann zeta function is defined by:

\begin{equation}
\zeta(s):=\sum_{n=1}^{\infty} \frac{1}{n^s}\,, \qquad \Re(s)>1\,.
\end{equation}

\noindent
It can be continued analytically to the entire complex plane except for a simple pole at $s=1$. The values of $\zeta(1/2+it)$ on finite intervals are of great interest to number theorists. For example, they are used in the numerical verification of the Riemann Hypothesis, and more recently, as numerical evidence for apparent connections between the zeta function and certain random matrix theory models. 

Such considerations have motivated searches for methods to numerically evaluate $\zeta(1/2+it)$ to within $\pm \,t^{-\lambda}$ for any $t>1$ and any fixed $\lambda$.  Searches for such methods can also be motivated from a computational complexity perspective, for the zeta function is of fundamental importance in number theory, so one may simply ask: how fast can it be computed?

In this article, new fast methods to numerically evaluate $\zeta(1/2+it)$ to within $\pm\, t^{-\lambda}$ for any $t>1$ and any fixed $\lambda$ are presented. Our fastest method has complexity $t^{4/13+o_\lambda(1)}$ (notice $4/13\approx 0.307$). This improves by a noticeable margin on the ``complexity bound'' of $t^{1/3+o_{\lambda}(1)}$. (The notations $O_{\lambda}(t)$ and $t^{o_{\lambda}(1)}$ indicate asymptotic constants depend only on $\lambda$, and are taken as $t\to \infty$.) 

Our main result is the following upper bound on the number of arithmetic operations (additions, multiplications, evaluations of the logarithm of a positive number, and evaluations of the complex exponential) on numbers of $O_{\lambda}((\log t)^2)$ bits that our fastest algorithm uses. 

\begin{thm} \label{thm:zeta1}
Given any constant $\lambda$, there are effectively computable constants $A_1:=A_1(\lambda)$, $A_2:=A_2(\lambda)$, $A_3:=A_3(\lambda)$, and $A_4:=A_4(\lambda)$, and absolute constants $\kappa_1$, $\kappa_2$, and $\kappa_3$, such that for any $t>1$, the value of $\zeta(1/2+it)$ can be computed to within $\pm\, t^{-\lambda}$ using $\le A_1\,(\log t)^{\kappa_1}\,t^{4/13}$ operations on numbers of $\le A_2\, (\log t)^2$ bits, provided a precomputation costing $\le A_3\,(\log t)^{\kappa_2}\,t^{4/13}$ operations, and requiring $\le A_4\,(\log t)^{\kappa_3}\,t^{4/13}$ bits of storage, is performed.
\end{thm}

We did not try to obtain numerical values for the constants $\kappa_1$, $\kappa_2$, and $\kappa_3$, in the statement of the theorem. With some optimization, it is likely each can be taken around $4$. We remark that a bit-complexity bound follows routinely from the arithmetic operations bound because all the numbers occurring in our algorithm have $O_{\lambda}((\log t)^2)$ bits. Also, our algorithm can be modified easily so that arithmetic is performed using $O_{\lambda}(\log t)$ bits, which is what one should do in a practical version. 

All of our methods immediately generalize off the critical line  since the Riemann-Siegel formula (\ref{eq:rsform}) can be generalized there. The methods exclusively tackle the issue of accurately computing the \textit{main sum} (\ref{eq:rsform1}) in the Riemann-Siegel formula, so they are quite independent of the precise choice of the remainder function there.

There are several known methods to compute $\zeta(1/2+it)$ to within $\pm\, t^{-\lambda}$ for any $t>1$ and any fixed $\lambda$. An elementary such method is usually derived from the Euler-Maclaurin summation formula. The majority of the computational effort in that method is in computing a main sum of length $O_{\lambda}(t)$ terms, where each term is of the form $n^{-1/2}\exp(it\log n)$; see~\cite{Ed} and~\cite{Ru} for a detailed description. 

Another method to compute $\zeta(1/2+it)$ relies on a straightforward application of the Riemann-Siegel formula, which has a main sum of length $\lfloor \sqrt{t/(2\pi)}\rfloor$ terms. A simplified version of that formula on the critical line is:

\begin{equation} \label{eq:rsform}
\zeta(1/2+it)= e^{-i\theta(t)}\,\Re \left(  2\,e^{-i\theta(t)}\sum_{n=1}^{n_1} n^{-1/2} \exp(it \log n)\right) +\Phi_\lambda(t)+O(t^{-\lambda})\,,
\end{equation}

\noindent
where $n_1:=\lfloor \sqrt{t/(2\pi)}\rfloor$, and $\theta(t)$ and $\Phi_{\lambda}(t)$ are certain generally understood functions. Odlyzko and Sch\"onhage~\cite{OS} showed that the rotation factor $\theta(t)$ and the remainder term $\Phi_{\lambda}(t)$ can both be evaluated to within $\pm\, t^{-\lambda}$ for any $\lambda$ using $t^{o_{\lambda}(1)}$ operations on numbers of $O_{\lambda}(\log t)$ bits. So to calculate $\zeta(1/2+it)$ using (\ref{eq:rsform}) directly, the bulk of the computational effort is exerted on the main sum, which is:

\begin{equation} \label{eq:rsform1}
\sum_{n=1}^{n_1} n^{-1/2} \exp(it \log n)\,,\qquad n_1:=\lfloor \sqrt{t/(2\pi)}\rfloor\,.
\end{equation}

Odlyzko and Sch\"onhage~\cite{OS} derived a practical algorithm to \textit{simultaneously} compute any $\lfloor T^{1/2}\rfloor$ values of \mbox{$\zeta(1/2+it)$} in the interval $t\in [T,T+T^{1/2}]$, to within $\pm\, T^{-\lambda}$ each, using $T^{1/2+o_{\lambda}(1)}$ arithmetic operations on numbers of $O_{\lambda}(\log T)$ bits, and requiring $T^{1/2+o_{\lambda}(1)}$ bits of storage. The Odlyzko-Sch\"onhage algorithm does not reduce the cost of a \textit{single} evaluation of zeta because it requires numerically evaluating a certain sum of length about $\sqrt{t/(2\pi)}$ terms, which is the same length as the Riemann-Siegel main sum (\ref{eq:rsform1}).

Sch\"onhage~\cite{Sc} improved the complexity of a single evaluation of zeta to $t^{3/8+o_{\lambda}(1)}$ operations on numbers of $O_{\lambda}(\log t)$ bits, and requiring $t^{3/8+o_{\lambda}(1)}$ bits of storage. Sch\"onhage~\cite{Sc} employed the Fast Fourier Transform (FFT) and subdivisions of the main sum in the Riemann-Siegel formula to derive his algorithm. 

Heath-Brown~\cite{HB} later presented a method that further lowered the cost of a single evaluation to about $t^{1/3}$ operations. He described the approach in the following way:\\

\begin{quote}
The underlying idea was to break the zeta-sum into $t^{1/3}$ subsums of  length $t^{1/6}$, on each of which $\exp({it\log(n+h)})$ could be approximated by a quadratic exponential $e(Ah+Bh^2+f(h))$ with $f(h)=O(1)$.  One would then pick a rational approximation $a/q$ to $B$ and write the sum in terms of  complete Gauss sums to \mbox{modulus $q$}.

This is motivated by section 5.2 in Titchmarsh, but using explicit formulae with Gauss sums in place of Weyl squaring.\\

\end{quote}

The problem of numerically evaluating $\zeta(1/2+it)$ was also considered by Turing~\cite{Tu}, Berry and Keating~\cite{BK}, and Rubinstein~\cite{Ru}, among others. 

The work of~\cite{Sc} and~\cite{HB} (see also~\cite{Ti} p.99, and \textsection{2}) makes it quite apparent that a possible approach to improving the complexity of computing zeta is to find efficient methods to numerically evaluate exponential sums of the form

\begin{equation} \label{eq:gensum}
\frac{1}{K^j}\sum_{k=1}^{K} k^j \exp(2\pi i f(k)), \qquad f(x) \in \mathbb{R}[x]\,.
\end{equation}

\noindent
This is our approach to improving the complexity of computing zeta. We derive algorithms that enable faster evaluations of the sum (\ref{eq:gensum}) when $f(x)$ is a quadratic polynomial or a cubic polynomial, with additional restrictions on the size of the cubic coefficient in latter. The basic idea is to apply Poisson summation to (\ref{eq:gensum}) to obtain a shorter exponential sum of a similar type. This is followed by an intervention that suitably normalizes the arguments of the new sum, then another application of Poisson summation, which yields yet a shorter sum, and so on. Notice the reason that repeated applications of Poisson summation do not trivialize (i.e. bring us back to where we started) is precisely because we intervene in between consecutive applications. 

We explain the relation between the exponential sums (\ref{eq:gensum}) and the zeta function in \textsection{2}. In the same section, we outline three methods to compute $\zeta(1/2+it)$ to within $\pm\,t^{-\lambda}$ in asymptotic complexities $t^{2/5+o_{\lambda}(1)}$, $t^{1/3+o_{\lambda}(1)}$, and $t^{4/13+o_{\lambda}(1)}$. The first method is rather simple, while the second and third methods are substantially more complicated. The second method, which has complexity $t^{1/3+o_{\lambda}(1)}$, relies on an efficient and quite involved algorithm to compute \textit{quadratic exponential sums}. These are sums of the form (\ref{eq:gensum}) with $f(x)$ a (real) quadratic polynomial. A nearly-optimal algorithm to compute such sums has already been derived in~\cite{Hi}. The third method, which has complexity $t^{4/13+o_{\lambda}(1)}$, relies on an efficient algorithm to compute \textit{cubic exponential sums} with a small cubic coefficient (see \textsection{2} for precise details). This algorithm is developed in \textsection{3}, \textsection{4}, and \textsection{5} here. 

We wish to make two remarks about the structure and the presentation of the cubic sums algorithm, which, as mentioned earlier, is the essential component of our $t^{4/13+o_{\lambda}(1)}$ method. First, as discussed in \textsection{3}, the algorithm generalizes that of~\cite{Hi} to compute quadratic exponential sums, except it incorporates the FFT and a few additional ideas. Second, our motivation for deriving the algorithm is to enable efficient enough evaluations of cubic sums rather than obtain elegant asymptotic expressions for them. So, for example, we describe in \textsection{4} how the algorithm involves evaluating certain exponential integrals, then we show in \textsection{5} how to compute each of these integrals separately. By doing so, however, we obscure that some of these integrals might combine in a manner that collapses them. We do not concern ourselves with combining or simplifying such integrals since this does not improve the complexity exponent of our method to compute zeta.

To summarize, \textsection{4} and \textsection{5} communicate the details of the cubic sums algorithm. If one merely wishes to overview the general structure without getting too involved in the details, then \textsection{3} might suffice.

\section{Outline of fast methods to compute $\zeta(1/2+it)$}

We first consider the following simplified, but prototypical, situation. Suppose we wish to numerically evaluate the sum

\begin{equation} \label{eq:tailbook}
\sum_{n=P}^{2P-1} \exp(it \log n)\,,\qquad P:=P_t= \lceil 0.5 \sqrt{t/(2\pi)} \rceil\,.
\end{equation}

\noindent
The sum (\ref{eq:tailbook}) is basically the last half of the Riemann-Siegel main sum for $\zeta(\sigma+it)$ on the line $\sigma=0$. We initially restrict our discussion to this line, and for the last half of the main sum, because the treatment of other values of $\sigma$, as well as of the remainder of the main sum, is completely similar, albeit more tedious. (See the discussion following Theorem~\ref{cubicthm} for a detailed presentation on the critical line for the full main sum.) 
 
A direct evaluation of the sum (\ref{eq:tailbook}) to within $\pm\, t^{-\lambda}$ requires $P^{1+o_{\lambda}(1)}$ operations on numbers of $O_{\lambda}(\log t)$ bits. However, one observes that individual ``blocks'' in that sum have a common structure, and that they all can be expressed in terms of exponential sums of the form (\ref{eq:gensum}). The new algorithms take advantage of this common structure to obtain substantially lower running times.  

Specifically, we divide the sum (\ref{eq:tailbook}) into consecutive blocks of length $K:=K_t$ each, where, for simplicity, we assume $P$ is multiple of $K$. So the sum (\ref{eq:tailbook}) is equal to a sum of $P/K$ blocks:

\begin{equation} \label{eq:blockbook}
\sum_{n=P}^{P+K-1} \exp(it \log n)+\sum_{n=P+K}^{P+2K-1}\exp(it \log n)+\cdots+\sum_{n=2P-K}^{2P-1} \exp(it \log n)\,.
\end{equation}

Let $v:=v_{P,K,r}=P+(r-1)K-1$. Then the $v^{th}$ block (the one starting at $v$) can be written in the form

\begin{equation} \label{eq:taybooktemp}
\sum_{k=0}^{K-1} \exp(it\log(v+k)) = \exp(it \log v) \sum_{k=0}^{K-1} \exp(it \log (1+k/v))\,.
\end{equation}

\noindent
Since $K-1<v$, we may apply Taylor expansions to $\log(1+k/v)$ to obtain

\begin{equation} \label{eq:taybook}
\begin{split}
\sum_{k=0}^{K-1} \exp(it \log (v+k))= \exp(it\log v)\sum_{k=0}^{K-1} \exp\left(\frac{it k}{v}- \frac{it k^2}{2v^2} +\frac{it k^3}{3v^3}-\frac{it k^4}{4v^4}+\ldots\right)\,.
\end{split} 
\end{equation}

There is much flexibility in which block sizes $K$ can be used. For example, if we choose $P t^{-1/3}< K\le P t^{-1/3}+1$, then the sum (\ref{eq:taybook}) can be reduced to a linear combination of quadratic exponential sums. Because with this choice of $K$, the ratio $K/P$ is very small and in particular $t K^3/P^3 = O(1)$, so the terms in the exponent on the r.h.s. of (\ref{eq:taybook}) become of size $O(1)$ starting at  the cubic term $itk^3/(3v^3)$ (more precisely, the $r^{th}$ term, which is  $(-1)^{r+1}it k^r/(rv^r)$, is of size $\le 2^r t^{1-r/3}$). This suggests the cubic and higher terms in (\ref{eq:taybook}) should be expanded away as polynomials in $k$ of low degree (say degree $J$). This in turn allows us to express the $v^{th}$ block (\ref{eq:taybook}) as a linear combination of quadratic exponential sums

\begin{equation} \label{eq:prelimbook}
\frac{1}{K^j}\sum_{k=0}^{K-1} k^j \exp\left(2\pi i a_{t,v} k+2\pi i b_{t,v} k^2\right)\,,\qquad j=0,1,\ldots,J\,,
\end{equation}

\noindent
plus a small error that we can easily control via our choice of $J$. It is easy to see that the coefficients of said linear combination are quickly computable, and are of size $O(1)$ each.

If the condition $P t^{-1/3}< K\le P t^{-1/3}+1$ is replaced by $P t^{-1/4}< K\le P t^{-1/4}+1$ say, then the cubic term $itk^3/(3v^3)$ is no longer of size $O(1)$, but the quartic term $-i t k^4/(4v^4)$ still satisfies the bound $O(1)$. Thus, on following a similar procedure as before, each block (\ref{eq:taybook}) can be reduced to a linear combination of cubic sums (instead of quadratic sums):

\begin{equation} \label{eq:cubicbook}
\frac{1}{K^j}\sum_{k=0}^{K-1}k^j \exp\left(2\pi i a_{t,v} k+2\pi i b_{t,v} k^2+2\pi i c_{t,v} k^3\right)\,,\qquad j=0,1,\ldots,J\,, 
\end{equation}

\noindent
where the cubic coefficient $c_{t,v}:=t/(6\pi v^3)$, and $K\approx t^{1/4}$. Notice by a straightforward calculation, we have $0\le c_{t,v}\le K^{-2}$, so the range where $c_{t,v}$ can assume values is quite restricted. By comparison, $a_{t,v}$ and $b_{t,v}$ can fall anywhere in $[0,1)$. 


\begin{table}[ht]
\caption{\footnotesize Choosing $K\approx t^{\beta}$ in (\ref{eq:blockbook}), where $P\approx t^{1/2}$, yields a total of $\approx t^{1/2-\beta}$ blocks, each of which can be expressed as a linear combination of the exponential sums (\ref{eq:genericbook}). Below are examples of  the polynomial $f_{\beta,t,v}(x)$ in (\ref{eq:genericbook}) for various choices of $\beta$.} \label{zetamethod1}
\begin{tabular}{|l|l||l|}
\hline
$\beta$ & $1/2-\beta$ & $f_{\beta,t,v}(x)$ \\
\hline
$1/10$ & $2/5$ & $a_{t,v} x+ b_{t,v} x^2$\\
$1/8$ & $3/8$ & $a_{t,v} x+ b_{t,v} x^2$\\
$1/6$ & $1/3$ & $a_{t,v} x+b_{t,v} x^2$ \\
$5/26$ & $4/13$ & $a_{t,v} x+b_{t,v} x^2+c_{t,v}x^3$ \\
$1/5$ & $3/10$ & $a_{t,v} x +b_{t,v} x^2 +c_{t,v}x^3$\\
$1/4$ & $1/4$ & $a_{t,v} x+b_{t,v} x^2+c_{t,v}x^3$ \\
$3/10$ & $1/5$ & $a_{t,v} x+b_{t,v} x^2+c_{t,v}x^3+d_{t,v}x^4$ \\
$1/3$ & $1/6$ & $a_{t,v} x+b_{t,v} x^2+c_{t,v}x^3+d_{t,v}x^4 +e_{t,v} x^5$ \\
\hline
\end{tabular}
\end{table}

\begin{table}[ht]
\caption{\footnotesize Bounds, in terms of $K$, on the absolute values of the coefficients of the polynomial $f_{\beta,t,v}(x)$ in (\ref{eq:genericbook}).}\label{zetamethod2}
\begin{tabular}{|l|l||lllllll|}
\hline
$\beta$ & $1/2-\beta$  &$a_{t,v}$ & $\,$ & $b_{t,v}$ & $\,$ & $c_{t,v}$ & $d_{t,v}$ & $e_{t,v}$ \\
\hline
$1/10$ &$2/5$ &$1$ &$\,\,$ & $1$&$\,\,$& \, & \,&\,\\
$1/8$ &$3/8$ &$1$ &$\,\,$ & $1$&$\,\,$& \, & \,&\,\\
$1/6$&$1/3$ &$1$ &$\,\,$ & $1$ &$\,\,$& \, & \,&\,\\
$5/26$&$4/13$ &$1$ &$\,\,$& $1$ &$\,\,$& $K^{-13/5}$ &\,&\,\\
$1/5$ &$3/10$ &$1$ &$\,\,$& $1$ &$\,\,$& $K^{-5/2}$ &\,&\,\\
$1/4$ &$1/4$ &$1$ &$\,\,$& $1$ &$\,\,$& $K^{-2}$ &\,&\,\\
$3/10$ &$1/5$ &$1$&$\,\,$ & $1$ &$\,\,$& $K^{-5/3}$ & $K^{-10/3}$ &\,\\
$1/3$ & $1/6$&$1$ &$\,\,$& $1$ &$\,\,$& $K^{-3/2}$ & $K^{-6/2}$ & $K^{-9/2}$\\
\hline
\end{tabular}
\end{table}

It is plain the procedure described so far can be continued further under appropriate hypotheses. In general, given $\beta\in (0,1/2)$, if the block size $K:=K_{\beta,t}$ in (\ref{eq:blockbook}) is chosen according to $P t^{\beta-1/2} < K \le P t^{\beta-1/2}+1$, we obtain a total of $\approx t^{1/2-\beta}$ blocks, each of which can be expressed as a linear combination of exponential sums of degree $d:=d_{\beta}=\lceil 1/(1/2-\beta)\rceil-1$: 

\begin{equation} \label{eq:genericbook}
\frac{1}{K^j}\sum_{k=0}^{K-1} k^j \exp(2\pi i f_{\beta,t,v}(k))\,,\qquad j=0,1,\ldots,J\,,
\end{equation}

\noindent
plus a small error of size $O(t^{\beta+J(1-(d+1)(1/2-\beta))}/\lfloor J/(d+1)\rfloor !)$. Tables~\ref{zetamethod1} and~\ref{zetamethod2} provide examples of the (degree-$d_{\beta}$) polynomial $f_{\beta,t,v}(x)$ for various choices of $\beta$. Notice that the error in approximating each block by a linear combination of the sums (\ref{eq:genericbook}) declines extremely rapidly with $J$. For example, taking $J = J_{\beta,t,\lambda}:= \lceil (d+1)(\lambda+3)\log t \rceil$ enables the computations of the $v^{th}$ block to within $\pm\, t^{-\lambda-2}$ for any fixed $\lambda$. 

Furthermore, it is straightforward to show, under mild hypotheses, that if one is capable of evaluating the exponential sums (\ref{eq:genericbook}), to within $\pm\, t^{-\lambda-2}$ each, for all $v\in \{P,P+K,\ldots,2P-K\}$ (which is a total of $\approx (J+1) t^{1/2-\beta}$ such sums), using $t^{1/2-\beta+o_{\lambda}(1)}$ time, then the entire Riemann-Siegel main sum, rather than its last half (\ref{eq:blockbook}) only, can be computed to within $\pm\, t^{-\lambda}$ in $t^{1/2-\beta+o_{\lambda}(1)}$ time. In particular, the column $1/2-\beta$ in tables~\ref{zetamethod1} and~\ref{zetamethod2} is the complexity exponent with which $\zeta(\sigma+it)$ can be computed on the line $\sigma=0$ should the sums  (\ref{eq:genericbook}) lend themselves to an efficient enough computation (in the sense just described). It is clear the restriction to the line $\sigma=0$ is not important, and similar conclusions can be drawn for other values of $\sigma$. 

Reformulating the main sum of the Riemann-Siegel formula in terms of quadratic exponential sums was carried out by Titchmarsh~\cite{Ti} (page 99), and later by Sch\"onhage ~\cite{Sc} and Heath-Brown~\cite{HB}. But higher degree exponential sums were not considered in these approaches.  Sch\"onhage sought an efficient method to evaluate quadratic exponential sums in order to improve the complexity of computing the zeta function. He observed  that if the values of the quadratic sums 

\begin{equation}\label{eq:schquad} 
F(K,j;a,b):=\frac{1}{K^j}\sum_{k=0}^K k^j \exp\left(2\pi i a k+2\pi i b k^2\right)\,,\qquad j=0,1,\ldots,J\,,
\end{equation}

\noindent
and several of their partial derivatives (which are of the same form), were known at all points $(a,b)$ of the lattice

\begin{equation}
\mathcal{L}=\{(p/K,q/K^2): 0\le p<K, 0\le q<K^2\}\,, 
\end{equation}

\noindent
then values of $F(K,j;a,b)$ elsewhere can be calculated quickly via a Taylor expansion like:

\begin{equation} \label{eq:tay1}
F(K,j;a,b)= \sum_{r=0}^{\infty} \frac{1}{r!} \sum_{l=0}^{r} \binom{r}{l} F(K,j+2r-l;p_0/K,q_0/K) (\Delta a)^l (\Delta b)^{r-l}\,,
\end{equation}

\noindent
where $\Delta a:=2\pi K(a-p/K)$, $\Delta b := 2\pi K^2(b-q/K^2)$, $p_0/K$ is a member of $\mathcal{L}$ closest to $a$, and $q_0/K^2$ is a member of $\mathcal{L}$ closest to $b$. This is because $|\Delta a|$ and $|\Delta b|$ are both bounded by $\pi$, so $|\binom{r}{l}(\Delta a)^l (\Delta b)^{r-l}|\le (2\pi)^r$, which implies the $r^{th}$ term in (\ref{eq:tay1}) is of size $\le K (2\pi)^r/r!$. Therefore, expansion (\ref{eq:tay1}) can be truncated early, say after $J':=J'_{t,\lambda}=\lceil 100(\lambda+1)\log t\rceil=t^{o_{\lambda}(1)}$ terms, which yields a truncation error of size $\le K t^{-\lambda-2}$. 

In particular, Sch\"onhage observed, if for each integer  $0\le m \le J+2J'$ the values  $F(K,m;p/K,q/K)$ are precomputed to within $\pm\,t^{-\lambda-2}$ on all of $\mathcal{L}$, then $F(K,j;a,b)$ can be computed for any $(a,b) \in [0,1)\times [0,1)$ to within $\pm\,t^{-\lambda-1}$ using expansion (\ref{eq:tay1}) in about $(J'+1)^2=t^{o_{\lambda}(1)}$ steps. Consequently, provided $K \le P t^{-1/3} +1$, which ensures that each block in (\ref{eq:blockbook}) can be approximated accurately by a linear combination of the quadratic sums (\ref{eq:schquad}), then (\ref{eq:blockbook}) can be computed to within $\pm t^{-\lambda}$ in about $(J+1)(J'+1)^2P/K=t^{o_{\lambda}(1)} P/K$ steps. 

Letting $\mathcal{C}:=\mathcal{C}_{K,J,J'}$ denote the presumed cost of computing the values of $F(K,m;a,b)$ on all of $\mathcal{L}$ and for $0\le m \le J+2J'$, we deduce that the choice of $K$ minimizing the complexity exponent for computing zeta is essentially specified by the condition $\mathcal{C}= P/K$, since this is when the precomputation cost $\mathcal{C}$ is balanced against the cost of computing the sum (\ref{eq:blockbook}) in blocks. Notice if $K$ is small,  then the number of blocks $P/K$ will be much larger than the size of the lattice $\mathcal{L}$. So there will be significant overlaps among the quadratic sums arising from the blocks in the sense many of them can be expanded about the same point $(p/K,q/K^2)\in \mathcal{L}$. And this should lead to savings in the running time.

To reduce the precomputation cost $\mathcal{C}$, Sch\"onhage observed that the value of $F(K,j;a,b)$ at $(p/K,q/K^2)$ is the discrete Fourier transform, evaluated at $-p/K$, of the sequence of points 

\begin{equation}
\left\{(k/K)^j \exp(2\pi i q k^2/K^2): 0\le k<K\right\}\,.
\end{equation}

\noindent
So for each $0\le m\le J+2J'$, and each $0\le q<K^2$, one can utilize the FFT to compute $F(K,m;p/K,q/K^2)$ for all $0\le p<K$ in $K^{1+o_{\lambda}(1)}$ steps. Since there are $J+2J'+1=t^{o_{\lambda}(1)}$ relevant values of $m$, and $K^2$ relevant values of $q$, then the total cost of the precomputation is about $K^{3+o_{\lambda}(1)}$ steps. The condition $\mathcal{C}=P/K$ thus reads $K^{3+o_{\lambda}(1)}=P/K$, and so one chooses $K=P^{1/4}$. This implies $P/K= P^{3/4+o_{\lambda}(1)}=t^{3/8+o_{\lambda}(1)}$, yielding Sch\"onhage's $t^{3/8+o_{\lambda}(1)}$ method to compute the sum (\ref{eq:blockbook}), hence, by a slight extension, the zeta function itself. 

It is possible to improve the complexity of computing the zeta function via this approach while avoiding the FFT, which can be advantageous in practice. Indeed, if one simply evaluates (\ref{eq:schquad}) at all points of $\mathcal{L}$ in a direct way, then the total cost of the precomputation is about $(J+2J'+1)K^{4+o_{\lambda}(1)}$ steps. The condition $\mathcal{C}=P/K$ hence reads $K= P^{1/5}$, which gives a $t^{2/5+o_{\lambda}(1)}$ method to compute the zeta function.

In order to achieve a $t^{1/3+o_{\lambda}(1)}$ complexity via this approach, the precomputation cost $\mathcal{C}$ must be lowered to about $K^{2+o_{\lambda}(1)}$ operations. However, it is not possible to lower the precomputation cost to $K^2$ operations (or to anything $\le K^3$ operations) if one insists on precomputing the quadratic sum on all of $\mathcal{L}$. This difficulty motivated our search in~\cite{Hi} for efficient methods to compute quadratic exponential sum, which led to:

\begin{thm}[Theorem~1.1 in~\cite{Hi}] \label{quadraticthm}
There are absolute constants $\kappa_4$, $\kappa_5$, $A_5$, $A_6$, and $A_7$, such that for any integer $K>0$, any integer $j \ge 0$, any positive $\epsilon<e^{-1}$, any $a,b \in [0,1)$, and with $\nu:=\nu(K,j,\epsilon)=(j+1)\log (K/\epsilon)$,  the value of the function $F(K,j;a,b)$ can be computed to within $\pm \, A_5\, \nu^{\kappa_4} \epsilon$ using $\le A_6\, \nu^{\kappa_5}$ arithmetic operations on numbers of $\le A_7\, \nu^2$ bits.
\end{thm}

\noindent
Theorem~\ref{quadraticthm} yields a $t^{1/3+o_{\lambda}(1)}$ method to compute zeta, as we describe, in detail, later in this section.

In \textsection{3}, \textsection{4}, \& \textsection{5}, we generalize Theorem~\ref{quadraticthm} to cubic exponential sums 

\begin{equation}\label{eq:hkjabc}
H(K,j;a,b,c):=\frac{1}{K^j}\sum_{k=0}^K k^j \exp(2\pi i a k +2\pi i b k^2+2\pi i c k^3)\,,
\end{equation}

\noindent
with a small cubic coefficient $c$. Unlike the the algorithm for quadratic sums though, where neither a precomputation nor an application of the FFT is necessary, our algorithm for cubic sums does require a precomputation and, in doing so, relies on the FFT in a critical way. We prove:

\begin{thm} \label{cubicthm}
There are absolute constants $\kappa_6$, $\kappa_7$, $\kappa_8$, $\kappa_9$, $A_8$, $A_9$, $A_{10}$, $A_{11}$, and $A_{12}$, such that for any $\mu \le 1$, any integer $K>0$, any integer $j\ge 0$, any positive $\epsilon<e^{-1}$, any $a,b \in [0,1)$, any $c\in [0,K^{\mu-3}]$, and with $\nu:=\nu(K,j,\epsilon)=(j+1)\log(K/\epsilon)$, the value of the function $H(K,j;a,b,c)$ can be computed to within  $\pm \, A_8 \, \nu^{\kappa_6}\epsilon$ using $\le A_9\,\nu^{\kappa_7}$ arithmetic operations on numbers of $\le A_{10}\,\nu^2$ bits, provided a precomputation costing $\le A_{11}\, \nu^{\kappa_8}\,K^{4\mu}$ arithmetic operations, and requiring $\le A_{12}\,\nu^{\kappa_9} \,K^{4\mu}$ bits of storage, is performed. 
\end{thm}

\noindent
Theorem~\ref{cubicthm} yields a $t^{4/13+o_{\lambda}(1)}$ method to compute zeta, as we explain later in this section. (We remark that the restriction $\mu \le 1$ in the statement of the Theorem is first needed during the first phase of the algorithm, in \textsection{4.1}, to ensure that a certain cubic sum can be evaluated accurately using the Euler-Maclaurin summation formula with only a few correction terms.) 

Let us show precisely  how theorems~\ref{quadraticthm} and~\ref{cubicthm} lead to faster method to compute the main sum in the Riemann-Siegel formula, hence $\zeta(1/2+it)$ itself. To this end, assume $t$ is large (say $t>10^{6}$), fix parameters $\lambda>0$ and $0<\beta<1/4$ say, and suppose our goal is to enable the computation of $\zeta(1/2+it)$ to within $\pm\,t^{-\lambda}$ using $t^{1/2-\beta+o_{\lambda}(1)}$ operations on numbers of $O_{\lambda}((\log t)^{\kappa_0})$ bits for some absolute constant $\kappa_0$, where we allow for a one-time precomputation costing $t^{1/2-\beta+o_{\lambda}(1)}$ operations (e.g. in order to obtain the $t^{1/3+o_{\lambda}(1)}$ and $t^{4/13+o_{\lambda}(1)}$ complexities, $\beta$ will have to be specialized to $1/6$ and $5/26$, respectively, while to obtain the $t^{2/5+o_{\lambda}(1)}$ and $t^{3/8+o_{\lambda}(1)}$ complexities, one specializes $\beta$ to $1/10$ and $1/8$). 

We digress briefly to remark that as $\beta$ increases, the target complexity $t^{1/2-\beta+o_{\lambda}(1)}$ improves, but also, the arguments of the resulting exponential sums will be allowed to assume values in larger and larger intervals, and sometimes the degree of these sums will increase (see tables~\ref{zetamethod1} and \ref{zetamethod2} for examples). In fact, the latter observations are the main difficulties in deriving yet faster zeta methods. Because, as explained in \textsection{3}, computing exponential sums via our approach becomes harder as the intervals where the arguments are allowed to assume values expand, or as the degree of the sum increases. 

Let us subdivide the Riemann-Siegel main sum (\ref{eq:rsform1}) into $O(\log t)$ subsums of the form 
\begin{equation} \label{eq:intermed}
\sum_{n=n_g}^{2n_g-1} n^{-1/2}\exp(it \log n)\,,
\end{equation}

\noindent
where $t^{1/2-\beta}\le n_g\le 0.5 \sqrt{t/(2\pi)}$, plus a remainder series of length $O(t^{1/2-\beta})$. The remainder series is not problematic since it can be evaluated directly in $t^{1/2-\beta+o_{\lambda}(1)}$ operations on numbers of $O_{\lambda}(\log t)$ bits, which falls within our target complexity. 

By construction, $n_g$ assumes only $O(\log t )$ distinct values. For each $n_g$, we choose a block size $K_{n_g}:=K_{\beta,n_g}$ according to $n_g t^{\beta-1/2}< K_{n_g}\le n_g t^{\beta-1/2}+1$, so $K_{n_g}\approx n_gt^{\beta-1/2}$. Then we rewrite the sum (\ref{eq:intermed}) as

\begin{equation} \label{eq:intermed1}
\sum_{v\in \mathcal{V}_{n_g}}   \sum_{k=0}^{K_{n_g}} \frac{\exp(i t \log (v+k))}{\sqrt{v+k}} +\mathcal{R}_{t,n_g,K_{n_g}}\,, 
\end{equation}

\noindent
where $\mathcal{V}_{n_g}:=\mathcal{V}_{n_g,K_{n_g}}$ is any set of $\lfloor n_g/K_{n_g} \rfloor$ equally spaced points in the interval $[n_g,2n_g)$, and $\mathcal{R}_{t,n_g}$ is a Dirichlet series of length at most $2K=O(t^{\beta})$ terms. Since $t^{\beta}\le t^{1/2-\beta}$ for $\beta\in[0,1/4]$, then $\mathcal{R}_{t,n_g}$ can be evaluated directly in $t^{1/2 -\beta+o_{\lambda}(1)}$ operations on numbers of $O_{\lambda}(\log t)$ bits, which is falls within our target complexity. So we may focus our attention on the double sum in (\ref{eq:intermed1}). 

Like before, one notes $\exp(it\log(v+k))=\exp(it\log v)\exp(it\log (1+k/v))$. Also, $\sqrt{v+k}=\sqrt{v}\sqrt{1+k/v}$. So by a routine application of Taylor expansions to $\log(1+k/v)$ and $1/\sqrt{1+k/v}$, each inner sum in (\ref{eq:intermed1}) is expressed in the form

\begin{equation} \label{eq:intermed3}
\exp(it\log v)\,\sum_{l=0}^{\infty}\frac{(-1)^l\, \prod_{r=1}^{l} (2r-1) }{l!\, 2^{l} v^{l+1/2}} \sum_{k=0}^{K_{n_g}} k^l \exp \left(it \sum_{m=1}^{\infty} \frac{(-1)^{m+1} k^m}{m\, v^m} \right)\,.
\end{equation}

\noindent
Since $k/v \le K_{n_g}/n_g \le 2t^{\beta-1/2}$, the Taylor series over $l$ and $m$ in (\ref{eq:intermed3}) converge quite fast. Truncating them at $L$ and $U$ respectively yields

\begin{equation} \label{eq:intermed4}
\exp(it\log v)\,\sum_{l=0}^{L}\frac{(-1)^l\, \prod_{r=1}^{l} (2r-1)}{l!\, 2^{l} v^{l+1/2}}\sum_{k=0}^{K_{n_g}} k^l \exp \left(it \sum_{m=1}^U \frac{(-1)^{m+1} k^m}{m\, v^m}+\epsilon_{k,U}\right) +\delta_{L}\,, 
\end{equation}

\noindent
where
\begin{equation}
\begin{split}
&\epsilon_{k,U}=O\left( t K_{n_g}^U/n_g^U\right)=O(t^{1-(1/2-\beta)U})\,,  \\
&\delta_{L}= O\left(K_{n_g}^{L+1}/n_g^L\right)=O(t^{\beta-(1/2-\beta)L})\,. 
\end{split}
\end{equation}

\noindent
We choose $L=U:=U_{\beta,\lambda}=\lceil (\lambda+10)/(1/2-\beta) \rceil$ say. So $L=U=O_{\beta,\lambda}(1)$.  By straightforward calculations, the total error (from the $\epsilon_{k,U}$'s and $\delta_L$) is of size $O_{\lambda}(t^{-\lambda-1})$. Therefore, each of the inner sums in (\ref{eq:intermed1}) is equal to

\begin{equation} \label{eq:intermed6}
\sum_{l=0}^L \frac{w_{l,t,v,K_{n_g}}}{K_{n_g}^l}\sum_{k=0}^{K_{n_g}} k^l \exp \left(it\sum_{m=0}^U \frac{(-1)^{m+1} k^m}{m\, v^m}\right)+O(t^{-\lambda-1})\,,
\end{equation}

\noindent
where
\begin{equation}
w_{l,t,v,K_{n_g}}:=\frac{(-1)^l\, \prod_{r=1}^{l} (2r-1)\,K_{n_g}^l}{l!\, 2^{l} v^{l+1/2}}\,\exp(it\log v)\,.
\end{equation}

\noindent
It is easy to see that each $w_{l,t,v,K_{n_g}}$ can be computed to within $\pm\, t^{-\lambda-1}$ using $t^{o_{\lambda}(1)}$ operations on numbers of $O_{\lambda}(\log t)$ bits. Also, by our choices of $v$ and $K_{n_g}$, $w_{l,v,K_{n_g}}=O(1)$ (in fact $|w_{l,v,K_{n_g}}|\le 2^l t^{-l(1/2-\beta)}$, but that is not important). Now let 
\begin{equation}
d:=d_{\beta}=\lceil 1/(1/2-\beta)\rceil-1\,.
\end{equation}

\noindent
By truncating the series over $m$ in (\ref{eq:intermed6}) at $d$, each inner sum there is equal to:

\begin{equation} \label{eq:intermed7}
\sum_{l=0}^L \frac{w_{l,t,v,K_{n_g}}}{K_{n_g}^l}\sum_{k=0}^{K_{n_g}} k^l \exp\left(2\pi i f_{\beta,t,v}(k)\right) \left[ \sum_{h=0}^{\infty} \frac{(-1)^{hd} i^h}{h!} \left(\frac{t k^{d+1}}{v^{d+1}}\right)^h\,(\alpha_{v,U,d,k})^h  \right]\,,
\end{equation}

\noindent
where 
\begin{equation}
f_{\beta,t,v}(x):=\sum_{m=0}^d a_{t,v,m} x^m\,, \qquad a_{t,v,m}:=\frac{(-1)^{m+1} t}{2\pi m\, v^m}\,.
\end{equation}

\begin{equation}
\alpha_{v,U,d,k}=\sum_{m=0}^{U-d-1} \frac{(-1)^{m} k^{m}}{(m+d+1)\, v^{m}}\,,
\end{equation}

\noindent
Notice that in the notation of tables~\ref{zetamethod1} and~\ref{zetamethod2}, we have $a_{t,v,1}=a_{t,v} \bmod{1}$, $a_{t,v,2}=b_{t,v}\bmod{1}$, $a_{t,v,3}=c_{t,v}\bmod{1}$,  and so on. 

By  routine calculations,  $a_{t,v,m}\in [-t/(2\pi m n_g^m),t/(2\pi m n_g^m)]$, $|\alpha_{v,U,d,k}|<2$, and, by our choice of $d$, $t (k/v)^{d+1}\le 2^{d+1}$. So if the Taylor series over $h$ in (\ref{eq:intermed7}) is truncated after $J'':=J''_{t,\lambda}=\lceil (\lambda+100) \log t \rceil$ terms, the resulting error is of size $O(t^{-\lambda-1})$. Each inner sum in (\ref{eq:intermed1}) is thus equal to

\begin{equation} \label{eq:intermed99}
\sum_{j=0}^{J'''} \frac{z_{l,t,v,K_{n_g},U}}{K_{n_g}^j}\sum_{k=0}^{K_{n_g}} k^j \exp \left(2\pi i f_{\beta,t,v}(k)\right)+O(t^{-\lambda-1})
\end{equation}

\noindent
where $J''':=L+(d+1)J''+1=O_{\lambda}(\log t)$, and $f_{\beta,t,v}(x)$ is a real polynomial of degree $d=\lceil 1/(1/2-\beta)\rceil -1$. The coefficients $z_{l,t,v,K_{n_g},U}$ in (\ref{eq:intermed99}) are of size $O(1)$ each, and are computable to within $\pm\, t^{-\lambda -1}$ using $t^{o_{\lambda}(1)}$ operations on numbers of $O_{\lambda}(\log t)$ bits. Last, suppose each of the sums: (this is a critical assumption, and proving it is the main difficulty in deriving faster zeta methods)

\begin{equation} \label{eq:intermed9}
 \frac{1}{K_{n_g}^j}\sum_{k=0}^{K_{n_g}} k^j \exp \left(2\pi i f_{\beta,t,v}(k)\right)\,,\qquad j=0,1,\ldots,J'''\,,
\end{equation}

\noindent
can be evaluated to within $\pm\, t^{-\lambda-1}$ in $t^{o_{\lambda}(1)}$ operations on numbers of $O_{\lambda}((\log t)^{\kappa_0})$ bits, for some absolute constant $\kappa_0$, where we allow for a one-time precomputation costing $t^{1/2-\beta+o_{\lambda}(1)}$ operations. Then $\zeta(1/2+it)$ can be computed to within $\pm\, t^{-\lambda}$ using $t^{1/2-\beta+o_{\lambda}(1)}$ operations on numbers of $O((\log t)^{\kappa_0})$ bits. 

Using this setup, Theorem~\ref{quadraticthm}, which handles quadratic exponential sums, yields a $t^{1/3+o_{\lambda}(1)}$ method to compute zeta. First, we let $\beta=1/6$, which implies the polynomial $f_{\beta,t,v}(x)$ has degree $d=\lceil 1/(1/2-1/6)\rceil -1 =2$, and the sums (\ref{eq:intermed9}) are quadratic exponential sums $F(K_{n_g}, j; a_{t,v}, b_{t,v})$.  So applying the algorithm of Theorem~\ref{quadraticthm} with $j\le \lceil (10\lambda+100) \log t\rceil$, $\epsilon=t^{-\lambda-10}$, and $K=K_{n_g}$, permits the evaluation of each quadratic sum to within $\pm\,t^{-\lambda-1}$ using $t^{o_{\lambda}(1)}$ operations on numbers of $O_{\lambda}(\nu(K,j,\epsilon)^2)$ bits. In addition, given our choices of $K_{n_g}$, $j$, and $\epsilon$, we have $\nu(K_{n_g},j,\epsilon)=O_{\lambda}((\log t)^2)$. Together, this yields a $t^{1/3+o_{\lambda}(1)}$ method to compute zeta:

\begin{thm} \label{thm:zeta0}
Given any constant $\lambda$, there are effectively computable constants $A_{13}:=A_{13}(\lambda)$ and $A_{14}:=A_{14}(\lambda)$, and an absolute constant $\kappa_{10}$, such that for any $t>1$, the  value of the function $\zeta(1/2+it)$ can be computed to within $\pm\, t^{-\lambda}$ using $\le A_{13}\,(\log t)^{\kappa_{10}}\,t^{1/3}$ operations on numbers of $\le A_{14}\, (\log t)^4$ bits.
\end{thm}

Quite similarly, Theorem~\ref{cubicthm}, which handles cubic sums, yields the $t^{4/13+o_{\lambda}(1)}$ asymptotic complexity. First, we let $\beta=5/26$, which implies the degree $d$ of $f_{\beta,t,v}(x)$ in (\ref{eq:intermed9}) is $\lceil 1/(1/2-5/26)\rceil -1 =3$. So the sums (\ref{eq:intermed9}) are now cubic exponential sums $H(K_{n_g},j;a_{t,v},b_{t,v},c_{t,v})$. Also, by construction, we have $K_{n_g}\le n_g t^{-4/13}+1$ and $|c_{t,v}|\le t/(6\pi n_g^3)$, and so $|c_{t,v}| \le  t^{1/13} K_{n_g}^{-3}$. In particular, the range where the cubic coefficient $c_{t,v}$ can assume values is restricted by the sizes of $K_{n_g}$ and $t$. With this in mind, we distinguish the following two cases: $K_{n_g} \le t^{1/13}$, and $K_{n_g} > t^{1/13}$. 

If $K_{n_g} \le t^{1/13}$, we use the FFT to precompute $H(K_{n_g},j;a,b,c)$, for each $0\le j\le J'''$, and at all points $(a,b,c)$ of the lattice

\begin{equation}
\{(p/K_{n_g},q/K_{n_g}^2,r/K_{n_g}^3)\, | \, 0\le p <K_{n_g}, 0\le q <K_{n_g}^2, 0\le r \le t^{1/13}\}\,.
\end{equation}

\noindent
Once the precomputation is carried out, then by a similar arguments to Sch\"onhage's method, Taylor expansions can be used to evaluate $H(K_{n_g},j;a_{t,v},b_{t,v},c_{t,v})$ for any $v\in \mathcal{V}_{n_g}$  to within $\pm\, t^{-\lambda-1}$ in $t^{o_{\lambda}(1)}$ operations on numbers of $O_{\lambda}(\log t)$ bits using the precomputed data. It remains to calculate the cost of the precomputation. Since $|c_{t,v}| \le  t^{1/13} K_{n_g}^{-3}$, then for each $K_{n_g}$, the cost of the FFT precomputation is $K_{n_g}^3 t^{1/13+o_{\lambda}(1)}= t^{4/13+o_{\lambda}(1)}$ operations on numbers of $O_{\lambda}(\log t)$ bits, and requiring $t^{4/13+o_{\lambda}(1)}$ bits of storage. As there are only $O(\log t)$ different values of $K_{n_g}$ that arise (one for each $n_g$), then the total cost of the precomputation is still $t^{4/13+o_{\lambda}(1)}$ operations.

On the other hand, if $K_{n_g}>t^{1/13}$, then the previous FFT computation becomes too costly. So we invoke the algorithm of Theorem~\ref{cubicthm} for cubic sums. We first observe $|c_{t,v}|\le t^{1/13}K_{n_g}^{-3}=K_{n_g}^{\tilde{\mu}-3}$, where $\tilde{\mu}:=\tilde{\mu}_{t,K_{n_g}}=(\log t^{1/13})/(\log K_{n_g})$ (notice since $t^{1/13}<K_{n_g}<t^{5/26}$, then $2/5\le\tilde{\mu}\le 1$, but this is not important to what follows). Applying Theorem~\ref{cubicthm} with $\mu=\tilde{\mu}$,  we merely need to perform an FFT precomputation costing $K_{n_g}^{4\tilde{\mu}+o_{\lambda}(1)}=t^{4/13+o_{\lambda}(1)}$ operations and requiring $K_{n_g}^{4\tilde{\mu}+o_{\lambda}(1)}=t^{4/13+o_{\lambda}(1)}$ bits of storage, after which $H(K_{n_g},j;a_{t,v},b_{t,v},c_{t,v})$ can be computed for any $v\in \mathcal{V}_{n_g}$ to within $\pm\, t^{-\lambda-10}$ say using $t^{o_{\lambda}(1)}$ operations on numbers of $O_{\lambda}(\nu(K,j,\epsilon)^2)$ bits, where $\epsilon= t^{-\lambda-10}$, and $j\le J'''$. Last, since $\beta=5/26$ we have the finer estimate $J'''=O_{\lambda}(1)$, and also $\log K_{n_g}=O(\log t)$, which together imply $\nu(K_{n_g},j,\epsilon)=O_{\lambda}(\log t)$. Collecting the various pieces together yields Theorem~\ref{thm:zeta1} immediately.

We make several comments. First, in order to improve on the complexity exponent $4/13$, one must lower the precomputation cost in Theorem~\ref{cubicthm} (equivalently, one needs a better handle on sums with larger cubic coefficients). In this regard, it appears that if certain quite substantial modifications to our cubic sums algorithm are carried out, then the $t^{4/13+o_{\lambda}(1)}$ method is capable of improvement to $t^{3/10+o_{\lambda}(1)}$ complexity (see \textsection{3} for further comments on this). 

Second, both of the $t^{1/3+o_{\lambda}(1)}$ and $t^{4/13+o_{\lambda}(1)}$ methods appear to be compatible with the amortized complexity techniques of Odlyzko-Sch\"{o}nhage~\cite{OS}. In the case of the $t^{1/3+o_{\lambda}(1)}$ method, for instance, this means that it can be modified to permit the computation of about $T^{4/13}$ values of $\zeta(1/2+i(T+t))$ in the interval $t\in [0,T^{4/13}]$ using $T^{4/13+o_{\lambda}(1)}$ operations.  

Last, in the specific case of the $t^{1/3+o_{\lambda}(1)}$ method, the use of the Riemann-Siegel formula is actually not important. One can  still achieve the $t^{1/3+o_{\lambda}(1)}$ complexity using Theorem~\ref{quadraticthm} even if one starts with the main sum in the Euler-Maclaurin formula, which involves $\approx t$ terms. This flexibility is particularly useful if one is interested in very accurate evaluations of $\zeta(s)$ in regions where explicit asymptotics for the error term in the Riemann-Siegel formula have not been worked out.

\section{The idea of the algorithm to compute cubic exponential sums with a small cubic coefficient}

We first recall that by a direct application of Poisson summation we have 

\begin{equation}\label{eq:poisson}
\sum_{k=0}^K e^{2\pi i f(k)} = \frac{1+e^{2\pi i f(K)}}{2}+ PV \sum_{m=-\infty}^{\infty} \int_0^K e^{2\pi i f(x)-2\pi i m x}\,dx\,,
\end{equation}

\noindent
where {\em{PV}} means the terms of the infinite sum are taken in conjugate pairs. Theorems \ref{quadraticthm} and \ref{cubicthm}, which handle quadratic and cubic exponential sums respectively, were inspired by the following application of formula (\ref{eq:poisson}) due to van der Corput (see~\cite{Ti}, page 75, for a slightly different version; also see the discussion following Theorem~1.2 in~\cite{Hi}):

\begin{thm}[van der Corput iteration]
Let $f(x)$ be a real function with a continuous and strictly increasing derivative in $s\le x\le t$. For each integer \mbox{$f'(s)\le m\le f'(t)$}, let $x_m$ be the (unique) solution of $f'(x)=m$ in $[s,t]$. Assume $f(x)$ has continuous derivatives up to the third order, that $\lambda_2 <|f''(x)|<A\lambda_2$ and $|f'''(x)|<A\lambda_3$ for all $s\le x \le t$, where $A$ is an absolute constant. Then

\begin{equation}\label{eq:vanbook}
\sum_{s\le k \le t} e^{2\pi i f(k)}=e^{\pi i /4} \sum_{f'(s) \le m \le f'(t)}\frac{e^{2\pi i (f(x_m)-mx_m)}}{\sqrt{|f''(x_m)|}}+\mathcal{R}_{s,t,f} 
\end{equation}

\noindent
where $\mathcal{R}_{s,t,f} =  O\left(\lambda_2^{-1/2}+\log(2+(t-s)\lambda_2)+(t-s)\lambda_2^{1/5}\lambda_3^{1/5}\right)$.
\end{thm}

In the case of quadratic exponential sums the polynomial $f(x)$ in (\ref{eq:vanbook}) is \mbox{$ax+bx^2$}, where, by the periodicity of $e^{2\pi i x}$, we may assume $a$ and $b$ are in $[0,1)$.  Taking $s=0$ and $t=K$ in (\ref{eq:vanbook}), and assuming $\lceil a\rceil <\lfloor a+2bK\rfloor$, which is frequently the case (this assumption is for a technical reason, and is to ensure $bK$ is bounded away from 0), then on substituting  $x_m=(m-a)/(2b)$ and \mbox{$f(x_m)=ax_m+bx_m^2$} in (\ref{eq:vanbook}), the van der Corput iteration points to a relation like 

\begin{equation}  \label{eq:vanbook1}
\begin{split}
\sum_{k=0}^K & \exp (2\pi i a k +2\pi i b k^2) =\\
&\quad \frac{e^{\pi i /4-\pi i a^2/(2b)}}{\sqrt{2b}}\sum_{m=\lceil a\rceil}^{\lfloor a+2bK\rfloor} \exp\left(2\pi i \frac{a}{2b} m-2\pi i\frac{1}{4b} m^2\right)+\mathcal{R}_1(a,b,K) \,.
\end{split}
\end{equation}

\noindent
As explained in~\cite{Hi}, the sum on the r.h.s. of (\ref{eq:vanbook1}) arises from the integrals in  (\ref{eq:poisson}) that contain a saddle-point, where an integral is said to contain a saddle-point if the exponent $f'(x)-m$ vanishes for some $0\le x\le K$ (intuitively, the integral accumulates mass in a neighborhood of the saddle-point). 

As discussed in~\cite{Hi}, the relation (\ref{eq:vanbook1}) has several interesting features. For example, the new sum there (on the r.h.s.) is still a quadratic sum. This is essentially a consequence of the self-similarity of the Gaussian $e^{-x^2}$. Also, although we started with a sum of length $K+1$ terms, the new sum has $\le 2bK$ terms. In particular, since we can always normalize $b$ so that $b\in [0,1/4]$ (see \cite{Hi}), then we can ensure the new sum has length $\le K/2$ terms, which is at most half the length of the original sum. 

Moreover, the remainder term $\mathcal{R}_1:=\mathcal{R}_1(a,b,K)$, which corresponds to integrals in (\ref{eq:poisson}) with no saddle-point, has a fairly elementary structure, and, as proved in~\cite{Hi}, it can be computed to within $\pm\, \epsilon$ using $O(\log^{\tilde{\kappa}_1}(K/\epsilon))$ operations on numbers of $O(\log^2(K/\epsilon))$ bits, where $\tilde{\kappa}_1$ is some absolute constant. Indeed, the bulk of the effort in computing $\mathcal{R}_1$ is exerted on a particularly simple type of an incomplete Gamma function: 

\begin{equation} \label{eq:igf}
h(z,w):=\int_0^1 t^z \exp(w t)\,dt\,,\qquad 0\le z\,,\,\,\,\, z\in \mathbb{Z}\,,\,\,\,\,\Re(w)\le 0\,.
\end{equation}

\noindent
For purposes of our algorithms, the non-negative integer $z$ in (\ref{eq:igf}) will be of size $O(\log(K/\epsilon)^{\tilde{\kappa}})$, where $\tilde{\kappa}$ is some absolute constant, and we wish to compute the quadratic sum to within $\pm\,\epsilon$ say. So the relevant range of $z$ is rather restricted, which is important as it enables fast evaluations of the integrals (\ref{eq:igf}).  

Given these features, the important observation is that by iterating relation (\ref{eq:vanbook1}) at most $\log_2 K$ times ($\log_2 x$ stands for the logarithm to base 2), the quadratic sum can be computed to within $\pm\, \epsilon \sqrt{K} \log_2 K$ in poly-log time in $K/\epsilon$. It should be emphasized that the main technical hurdle in proving this is in showing the remainder term $\mathcal{R}_1$ can in fact be computed sufficiently accurately and quickly.

It is natural to ask whether this procedure, whereby one attempts to apply the van der Corput iteration repeatedly, should generalize to  cubic and higher degree exponential sums. If it does, then in view of tables~\ref{zetamethod1} and~\ref{zetamethod2} from \textsection{2}, we could obtain faster methods to compute the zeta function.  One difficulty towards such generalizations is the lack of self-similarity in the case of higher degree exponential sums. As we explain next, however, it is possible to overcome this difficulty if the coefficients of  higher degree terms are small enough. 

To this end,  consider the van der Corput iteration in the case of cubic exponential sums. So in (\ref{eq:vanbook}), we take $s=0$, $t=K$, and 

\begin{equation}
f(x)=ax+bx^2+cx^3\,. 
\end{equation}

\noindent
As explained in \cite{Hi}, it follows from the periodicity of $e^{2\pi i x}$, and conjugation if necessary, that we may assume $a \in [0,1)$ and $b \in [0,1/4]$. We may also assume $K>1000$ say, and $c\ne 0$, because if $c=0$ we obtain a quadratic sum. We require the \textit{starting} cubic coefficient $c$ and the \textit{starting} sum length $K$ to satisfy $|cK^2|\le 0.01$ (the constant $0.01$ is not significant but it is convenient to use). This condition essentially corresponds to restricting $\mu\le 1$ in the statement of Theorem~\ref{cubicthm}. 

Our plan is to cut the length of the cubic exponential sums by repeatedly applying the van der Corput iteration until one, or both, of the following two conditions fails (possibly before entering the first iteration): $|cK^2|\le 0.01 b$, or $bK\ge 1$. Notice these conditions necessitate $\lceil a\rceil \le \lfloor a+2bK+3cK^2\rfloor$, among other consequences. The significance of these conditions will become clear a posteriori.

First, we locate the saddle-points by solving the equation $f'(x)=m$ for $f(0)\le m \le f(K)$, or equivalently for $\lceil a\rceil \le m \le \lfloor a+2bK+3cK^2\rfloor$. We obtain

\begin{equation} \label{eq:quadbook}
x_m=\frac{\sqrt{b^2+3c(m-a)}-b}{3c}\,.
\end{equation}

\noindent
On substituting $x_m$ in $f(x)-mx$, one finds

\begin{equation}\label{eq:quadbook0}
f(x_m)-mx_m=\frac{2b^3+9bc(m-a)-2\left(b^2+3c(m-a) \right)^{3/2}}{27c^2}\,.
\end{equation}

The function $f(x_m)-mx_m$ is not a polynomial in $m$ because it involves a square-root term $\sqrt{b^2+3c(m-a)}$ (informally, this is a manifestation of the lack of self-similarity in the cubic case). So the behavior of $f(x_m)-mx_m$ under further applications of the van der Corput iteration is likely to be complicated and hard to control. It therefore appears we cannot apply relation (\ref{eq:vanbook}) repeatedly, like we did in the quadratic case, because we do not obtain sums of the same type with each application. However, if $c$ is sufficiently small, it is reasonable to apply Taylor expansions to the problematic term $\sqrt{b^2+3c(m-a)}$ in order to express it as a rapidly convergent power series in $3c(m-a)/b^2$. And if $c$ is sufficiently small, then the contribution of the quartic and higher terms in this series will be of size $O(1)$, which allows us to expand them away as polynomials of low degree in $m-a$. 

Indeed, by the conditions $|cK^2|\le 0.01b$ and $1\le bK$ that we imposed on $c$ and $b$ earlier, together with the restriction $\lceil a\rceil \le m \le \lfloor a+2bK+3cK^2\rfloor$, we see that $3c (m-a)/b^2\le 9cK/b \le 0.09b<1$. So it is permissible to apply Taylor expansions to $\sqrt{b^2+3c(m-a)}=(1/b)\sqrt{1+3c(m-a)/b^2}$ to obtain 

\begin{equation} \label{eq:vancubbook1}
\begin{split}
f(x_m)-mx_m=&-\frac{(m-a)^{2}}{4 \, b} + \frac{c (m-a)^{3}}{8 \, b^{3}} - \frac{9 \, c^{2}(m-a)^{4}}{64 \, b^{5}}\\ 
&\qquad + \frac{27 \, c^{3} (m-a)^{5}}{128 \, b^{7}}- \frac{189\,c^4 (m-a)^6}{512\,b^9}+ \cdots
\end{split}
\end{equation}

\noindent
Notice that if $bK \ge 1$, then expansion (\ref{eq:vancubbook1}) is still valid under the looser condition $|cK^2|\le 0.01$, so, in particular, the full force of the condition $|cK^2|\le 0.01b$ has not been used yet.

Now the quartic term in (\ref{eq:vancubbook1}) satisfies $9c^2(m-a)^4/(64b^5) \le 100 c^2K^4/b< b$, which is small. And the quintic term satisfies $27 c^3(m-a)^5/(128 b^7)< b/K$, which is also small. In general, the $r^{th}$ term is of the form \mbox{$\eta_r c^{r-2}(m-a)^r/b^{2r-3}$}, where $|\eta_r|< 3^r$, and so it has size \mbox{$\le 9^r c^{r-2} K^r/b^{r-3} < bK^{4-r}$}.  The rapid decline in (\ref{eq:vancubbook1}) is useful because, as indicated earlier, the contributions of quartic and higher terms can now be eliminated from the exponent on the right side of (\ref{eq:vanbook}) by expanding them away as a polynomial in $m-a$ of relatively low degree, plus a small error term. This ensures the new sum produced by the van der Corput iteration is still cubic. Specifically, we obtain 

\begin{equation}\label{eq:sumboth}
\begin{split}
&\frac{\exp(2\pi i (f(x_m)-mx_m))}{\sqrt{|f''(x_m)|}}=\\
&\quad \exp\left(-\frac{2\pi i\, (m-a)^{2}}{4 \, b} + \frac{2\pi i\, c (m-a)^{3}}{8 \, b^{3}}\right)\, \sum_{j=0}^{\tilde{J}} \eta_{a,b,c,j} (m-a)^j+\mathcal{E}_{\tilde{J},a,b,c,m}\,,
\end{split}
\end{equation}

\noindent
where we have tactfully expanded the term $\sqrt{|f''(x_m)|}=\sqrt{2b}(1+3c(m-a)/b^2)^{1/4}$ in the denominator on the l.h.s. as a power series in $m-a$ as well. Notice that the full force of the condition $|cK^2|<0.01b$ still has not been used, and all is needed to ensure the new sum is still cubic is a looser bound like $|cK^2|<0.01\sqrt{b}$. 

By straightforward estimates, we have

\begin{equation}\label{eq:twoest22}
|\eta_{a,b,c,j}m^j| \le \frac{A}{\sqrt{b}}\,,\qquad |\mathcal{E}_{\tilde{J},a,b,c,m}|\le \frac{A}{\sqrt{b}}\left(\frac{1}{\lfloor \tilde{J}/4\rfloor !}+e^{-\tilde{J}}\right)\,,
\end{equation}

\noindent
where $A$ is some absolute constant, $4\le \tilde{J}$, and $\lceil a\rceil \le m\le \lfloor a+2bK+3cK^2\rfloor$. Therefore, given $a\in [0,1)$ and $b\in [0,1/4]$, and assuming $|cK^2| \le 0.01 b$ ($c$ is real), and $1\le bK$, as we have done so far, then summing both sides of (\ref{eq:sumboth}) over $\lceil a\rceil \le m\le \lfloor a+2bK+3cK^2\rfloor$,  the van der Corput iteration (\ref{eq:vanbook}) suggests a relation of the form 

\begin{equation}  \label{eq:vanbook2}
H(K,0;a,b,c)=\sum_{j=0}^{\tilde{J}} \tilde{z}_j \, H(\tilde{K},j; \tilde{a},\tilde{b},\tilde{c})+\mathcal{R}_2(a,b,c,K)+\mathcal{E}_{\tilde{J},a,b,c,K}\,.
\end{equation}

\noindent
where $\tilde{K}:=\tilde{K}(a,b,c,K)=\lfloor a+2bK+3cK^2\rfloor$, and

\begin{equation}
\tilde{z}_j:=\tilde{z}_{j,a,b,c,K}= \exp\left(\displaystyle \frac{\pi i}{4}-\frac{\pi i a^2(ac+2b^2)}{4b^3}\right) \, \tilde{K}^j\,\eta_{a,b,c,j}\,.
\end{equation}

\noindent
The remainder $\mathcal{R}_2:=\mathcal{R}_2(a,b,c,K)$ corresponds to the remainder term in iteration (\ref{eq:vanbook}), except possibly for an extra term of $-z_0$ in case $\lceil a\rceil =1$. Also, simple algebraic manipulations yield

\begin{equation} \label{eq:neww2}
\tilde{a}:=\tilde{a}_{a,b,c}= \frac{a}{2b}+\frac{3a^2c}{8b^3}\,,\quad \tilde{b}:=\tilde{b}_{a,b,c}=\frac{1}{4b}+\frac{3ac}{8b^3}\,,\quad \tilde{c}:=\tilde{c}_{a,b,c}=\frac{c}{8b^3}\,.
\end{equation} 

\noindent
And as a consequence of the estimates in (\ref{eq:twoest22}), we have 

\begin{equation}\label{eq:twoests}
|\tilde{z}_j|\le \frac{A}{\sqrt{b}}\,,\qquad |\mathcal{E}_{\tilde{J},a,b,c,K}|\le 3A\sqrt{b}K \left(\frac{1}{\lfloor \tilde{J}/4\rfloor !}+e^{-\tilde{J}}\right)\,,
\end{equation}

\noindent
By the second estimate in (\ref{eq:twoests}), if we choose $\tilde{J} =\lceil 4\log (K/\epsilon)\rceil$ say, we ensure $|\mathcal{E}_{\tilde{J},a,b,c,K}|\le A\epsilon$, which is small enough for purposes of the algorithm. So $\tilde{J}$ need not be taken large at all. 

The remainder term $\mathcal{R}_2$ in the van der Corput iteration (\ref{eq:vanbook2}) essentially corresponds to terms in the Poisson summation formula with no saddle-point contribution. This is because such contributions have already been extracted as the new cubic sums. So, guided in part by the case of quadratic sums, we expect $\mathcal{R}_2$ to involve relatively little cancellation among its terms, and that its computation is much less expensive than a direct evaluation of the original cubic sum. 

It will transpire that the bulk of the effort in computing $\mathcal{R}_2$ is essentially spent on dealing with a particularly simple type of an incomplete Gamma function like (\ref{eq:igf}) (which, as mentioned earlier, is where the bulk of the effort is spent in the quadratic case as well). It will also transpire there are many routes with great flexibility to evaluating the expressions and integrals occurring in $\mathcal{R}_2$. One finds several methods, applicable in overlapping regions, that can be used. This flexibility is somewhat analogous to that encountered in evaluating the incomplete Gamma function for general values of its arguments, where also many methods applicable in overlapping regions are available; see~\cite{Ru}. But in our case, the task is significantly simpler because the relevant ranges of the arguments will be fairly restricted. 

With this in mind, our goal for this article is not to develop especially practical techniques to compute the integrals occurring in $\mathcal{R}_2$, which is a somewhat challenging task on its own, rather, it is to obtain explicit and provably efficient techniques to achieve the complexity bounds claimed in Theorem~\ref{cubicthm}. For example, we make heavy use of Taylor expansions throughout, which simplifies some conceptual aspects of the techniques that we present, but it also often leads to large asymptotic constants and a loss of practicality. 

We expect the van der Corput iteration (\ref{eq:vanbook2}) to generalize even further. Under the same assumptions on $K$, $a$, $b$, and $c$ as before, and for $J$ of moderate size (say $J$ bounded by some absolute power of $\log K$), we expect a relation like

\begin{equation}\label{eq:vanbook3}
\begin{split}
\sum_{j=0}^J w_j\, H(K,j;a,b,c) = \sum_{j=0}^{J+\tilde{J}} \tilde{w}_j\, H(\tilde{K},j;\tilde{a},\tilde{b},\tilde{c})+\tilde{\mathcal{R}}_2 +\mathcal{\tilde{E}}_{K,\tilde{J}}\,,
\end{split}
\end{equation}

\noindent
where $\tilde{K}$, $\tilde{a}$, $\tilde{b}$, and $\tilde{c}$ are the same as in (\ref{eq:neww2}). This is because the additional term $k^j/K^j$ (on the l.h.s.) is not oscillatory, and so the form of the result is given by (\ref{eq:vanbook}), except for an extra factor of $(x_m)^j$ on the right side.

The utility of a transformation like (\ref{eq:vanbook3}) is it can be repeated easily, for we still obtain cubic exponential sums with each repetition. In other words, provided the cubic coefficient is small enough (specifically $cK^2 \le 0.01 b$), relation (\ref{eq:vanbook3}) enables us to circumvent one difficulty in the cubic case, which is the lack of self-similarity (recall the self-similarity of the Gaussian and the periodicity of the complex exponential are critical ingredients in our algorithm for computing quadratic sums). 

Importantly, if $a$ and $b$ continue to be normalized suitably at the beginning of each iteration of (\ref{eq:vanbook3}), and if $c$ continues to satisfy $|cK^2|\le 0.01b$ throughout, then the outcome after $m$ repetitions is a linear combination of $J+m\tilde{J}+1$ cubic sums, each of length $\le K/2^m+2$, where $K$ denotes the length of the original cubic sum (the sum we started with). It is straightforward to show the coefficients in the linear combination are bounded by an absolute power of $K$. And due to the lack of saddle-points the remainder term $\mathcal{R}_2$, we anticipate it will be computable accurately and efficiently enough for purposes of proving the complexity bounds of Theorem~\ref{cubicthm}. 

It remains to deal with the possibility that either (or both) of the conditions $1\le bK$ and $|cK^2| \le 0.01 b$ fails, or that $K$ gets too small in comparison with the given $\epsilon$, say $K\le \nu(K,\epsilon)^6$. The latter case is trivial though since the cubic sum can then be evaluated directly to within $\pm\,\epsilon$ in poly-log time in $K/\epsilon$. 

A failure of the condition $1\le bK$ implies $b<1/K$, which means $b$ is quite small. This is a boundary case of the algorithm. One can use the Euler-Maclaurin summation technique to compute the cubic sum to within $\pm\,\epsilon$ in poly-log time in $K/\epsilon$. This is possible to do because we will have $a+2bK+|3cK^2|=O(1)$, which means the derivatives of the summand $\exp(2\pi i a x +2\pi i b x^2+2\pi i c x^3)$, evaluated at $K$, will not grow too rapidly for purposes of applying the Euler-Maclaurin formula in such a way that only $O(\log (K/\epsilon))$ correction terms are needed. (Of course, depending on the exact size of $a+2bK+|3cK^2|$, one may first have to divide the cubic sum into $O(1)$ consecutive subsums, then apply the Euler-Maclaurin formula to each subsum.)  To see why the estimate $a+2bK+|3cK^2|=O(1)$ should hold, let $b'$, $c'$, and $K'$ denote the quadratic coefficient, cubic coefficient, and sum length from the previous iteration, respectively. Then $K\le 2b'K'+1$, $c=c'/(2b')^3$, and by hypothesis $|c'(K')^2|\le 0.01b'$, $1\le b'K'$, and $bK<1$. Combined, this implies $a+2bK+|3cK^2|< 3+|4c'(K')^2|/b'=O(1)$.  Notice this is the first time the full force of the condition $|cK^2|\le 0.01b$ has been used.

Put together, we may now assume the condition $|cK^2|\le 0.01b$ is the sole condition that fails.  Notice the reason the condition $|cK^2|\le 0.01b$ can eventually fail is that the new cubic coefficient is given by $c=c'/(8(b')^3)$, which for $b'\in (0,1/2)$ is greater than the previous cubic coefficient $c'$. So although the length of the cubic sum is cut by a factor of about $2b'$ with each application of the van der Corput iteration,  the size of the cubic coefficient grows by a factor of $1/((2b')^3$.  

Now, when the condition $|cK^2|\le 0.01b$ fails, we apply the van der Corput iteration exactly once more. Since $cK^2 \le c' (2b'K'+1)^2/(2b')^3\le 0.02$, the series expansion (\ref{eq:vancubbook1}) is still valid, but its convergence is quite slow.  In particular, the exponential sum resulting from this last application of the van der Corput iteration, which is a sum of length $\tilde{K} \le cK^3$ terms, is not necessarily cubic. In order to compute this (possibly high degree) sum efficiently, we ultimately rely on precomputed data. And as shown in \textsection{4} later, the cost of the precomputation is about $\tilde{K}^4 \le c^4K^{12}$ operations. 

In summary, starting with a cubic coefficient satisfying $c\in [0,K^{\mu-3}]$ say,  where $\mu\le 1$, we repeatedly apply the van der Corput iteration (\ref{eq:vanbook3}) until $K$ is not large enough, or the condition $1\le bK$ fails, or we encounter a sole failure of the condition $|cK^2|<0.01b$ ($b$ here refers to the normalized $b\in [0,1/4]$). The first two scenarios are relatively easy to handle, and represent boundary points of the algorithm. The third scenario is more complicated. There, we apply the van der Corput iteration exactly once more, which leads to an exponential sum of length roughly $\le K^{\mu}$ terms. This last sum is not necessarily cubic, and can be of high degree. Nevertheless, we show it can in fact be evaluated to within $\pm\,\epsilon$ in poly-log time in $K/\epsilon$ provided an FFT precomputation costing $K^{4\mu+o(1)}$ operations on numbers of $O(\log (K/\epsilon)^2)$ bits is performed. 

It might be helpful to keep the following prototypical example in mind. We start with $c=K^{\mu-3}$ and $b=K^{\mu-1}$, and apply the van der Corput iteration to the cubic sum. This produces a new sum of length $\tilde{K}\approx K^{\mu}$ terms. The new sum is still cubic because the contribution of the quartic and higher terms in the series (\ref{eq:vancubbook1}) is of size $O(1)$.  We observe that the new cubic coefficient $\tilde{c}=c/(2b)^3$ is  of size about $\tilde{K}^{-2}\approx K^{-2\mu}$, and that the van der Corput iteration must end since $\tilde{c}\tilde{K}^2\approx 1$. The last cubic sum is then evaluated in poly-log time using data that was precomputed via the FFT. The typical cost of the precomputation is about $\tilde{K}^4\approx K^{4\mu}$ steps since there are approximately $\tilde{K}$, $\tilde{K}^2$, and $\tilde{K}$, discretized values to consider for the linear, quadratic, and cubic, arguments in the last sum (see our description of Sch\"onhage's method earlier). 

We remark that many of steps of the previous outline still work even if we relax some of the conditions; e.g. it might be possible to relax the halting criteria $cK^2\le 0.01b$ to $cK^2\le 0.01\sqrt{b}$. So it might be possible to obtain some yet faster algorithms to compute cubic sums via this approach, which in turn lead to faster methods to compute $\zeta(1/2+it)$. In addition, the idea of repeated applications of the van der Corput iteration can be used to compute exponential sums of higher degree (as well as cubic sums with a larger cubic coefficient). However, the complexities of the resulting algorithms are not useful for computing zeta since the costs of their needed precomputations are too high. 

It is plain that the quadratic and cubic sums algorithms share many features, and it is desirable to take advantage of such similarities as directly and as fully as possible. So although the approach just outlined appears to be a natural way to generalize the quadratic sums algorithm of~\cite{Hi} to the case of cubic sums, a literal implementation of it does not make direct use of the techniques and methods already developed in~\cite{Hi}. We thus choose a different implementation that breaks up the cubic sum into simpler components most of which are already handled by~\cite{Hi}. This way, in the course of our derivation of the cubic sums algorithm, we avoid having to reconstruct the algorithm for quadratic sums from scratch.

Specifically, in the initial step of our implementation, we convert the cubic sum  to an integral involving a quadratic sum. This is followed by two phases (one of which is largely a direct application of the quadratic sums algorithm), then an elementary saddle-point calculation, and an FFT precomputation. We give a technical overview of this chain of steps in the next few paragraphs. 

Let $\epsilon\in (0,e^{-1})$, and $0\le \mu\le 1$, assume $K>1000$, and $c_0\in [0,K^{\mu-3}/100]$ say. It is straightforward to verify 
\begin{equation}\label{eq:stphase}
H(K,j;a_0,b_0,c_0)=\frac{1}{K^j}\int_{-1/2}^{1/2} F(K;a_0-x,b_0)  Q(K,j;x,c_0) \, dx \,,
\end{equation}
\noindent
where the quadratic sum $F(.)$ and the cubic sum $Q(.)$ are defined by 
\begin{equation}
\begin{split}
F(K;a_0,b_0)&:=\sum_{k=0}^K \exp(2\pi i a_0 k+2\pi i b_0 k^2)\,,\\
Q(K,j;x,c_0)&:= \frac{1}{K^j}\sum_{k=0}^{K} k^j \exp(2\pi i  x k +2 \pi i c_0 k^3)\,.
\end{split}
\end{equation}

In order to motivate the first phase of the algorithm, we observe that if $F(K;a_0-x,b_0)$ on the r.h.s. of (\ref{eq:stphase}) is replaced (completely unjustifiably) by $F(K;a_0,b_0)$, then the cubic and quadratic sums there become ``decoupled'', or independent of each other. Also, since $c$ is small, the cubic sum $Q(.)$ can essentially be converted, via the Euler-Maclaurin formula, to an integral that is not problematic to compute. So, under such a hypothetical replacement of $F(K;a_0-x,b_0)$ by $F(K;a_0,b_0)$, the bulk of the effort in computing the r.h.s. is essentially in computing a quadratic sum, which we already know how to do efficiently via Theorem~\ref{quadraticthm}. 

With this in mind, the purpose of the first phase is to ``decouple'' the quadratic and cubic sums in (\ref{eq:stphase}) by making the coefficient of $-x$ in $F(K;a_0-x,b_0)$, which starts at 1, as small as possible. We essentially regard the first phase as a useful technical device to allow us to apply the quadratic sums algorithm during the second phase. Slightly more explicitly, the first phase consists of $O(\log K)$ iterations. With each iteration, the size  of cubic coefficient $c_0$ in (\ref{eq:stphase}) grows, while the length of the cubic sum $Q(.)$ decreases. Also, the coefficient of $-x$ in $F(K;a_0-x,b_0)$, which starts at 1, decreases with each iteration. It is shown the number of operations cost of each iteration is polynomial in $\nu(K,j,\epsilon)$. 

The first phase ends when $c$, which denotes the current value of the cubic coefficient, starts becoming too large for the Euler-Maclaurin formula to accurately approximate the cubic sum $Q(.)$ in such a way that only a few correction terms are needed. This roughly occurs when $c N^2\approx 1$, where $N$ denotes the current length of the cubic sum. At that point, the cubic sum $Q(.)$ is converted to an integral, plus a few correction terms, via the Euler-Maclaurin formula, which yields a main expression of the form:
\begin{equation}\label{eq:sndphase} 
\frac{1}{N^j}\int_0^N \exp(2\pi i c y^3)\int_{-1/4}^{1/4} \exp(-2\pi i x y) F(K;a_0- \alpha_0\, x,b_0)  \, dx\,dy \,, 
\end{equation}

\noindent
where $0< \alpha_0 \lesssim K^{\mu-1}$, $N \lesssim K^{\mu}$, and $0< c < 1/N^2$ (see \textsection{4.1} for precise details).

In the second phase, we apply the quadratic sums algorithm to $F(K;a_0-\alpha_0\, x,b_0)$. It is straightforward to do so, despite the presence of the integral sign in (\ref{eq:sndphase}),  precisely because $\alpha_0$ is relatively small (recall $0< \alpha_0 \lesssim K^{\mu-1}$), and so the length of the quadratic sum, which is about $a_0-\alpha_0 x+2bK$, depends very weakly on $x$. With each iteration of the second phase, the size of $\alpha_0$ grows by a factor of $1/2b_0\ge 2$ while the length of the quadratic sum $F(.)$ is multiplied by a factor of about $2b_0\le 1/2$. It is shown the number of operations cost of each iteration is polynomial in $\nu(K,j,\epsilon)$. In general, the second phase is more time-consuming than the first one because it requires evaluating the remainder terms resulting from the quadratic sums algorithm. 

The second phase ends when the value of $\alpha_0$ nears 1, at which point further applications of the quadratic sums algorithm start becoming complicated. This is because the length of the quadratic sum, which is about $a_0-\alpha_0 x+2bK$, then depends measurably on $x$. At the end of the second phase, we are left with an expression of the form 
\begin{equation}\label{eq:nddphase}
\begin{split} 
\frac{1}{N^j}\int_0^N \exp(2\pi i c y^3)\int_{-1/4}^{1/4} & \exp(-2\pi i y x-2\pi i \alpha_1 x -2\pi i \alpha_2 x^2)\,\times\\
& F(M;a- \alpha x,b)  \, dx\,dy\,,
\end{split} 
\end{equation}

\noindent
where  $M\le N/\alpha$, $1/\Lambda(K,j,\epsilon)<\alpha<1$, and $\alpha$, $a$, and $b$, denote the values of $\alpha_0$, $a_0$, and $b_0$ at the end of the second phase. The numbers $\alpha_1$ and $\alpha_2$ in (\ref{eq:nddphase}) are certain real parameters that are related to $\alpha$ and satisfy $|\alpha_1|<4\alpha$ and $|\alpha_2|<\alpha$. 

To each term in the quadratic sum $F(M;a- \alpha x,b)$, there is an associated ``saddle-point with respect to $y$'' (see \textsection{4.3}). On extracting the saddle-point contributions (there are about $M$ of them), expression (\ref{eq:nddphase}) is reduced  to a short linear combination of $O(\nu(K,j,\epsilon))$ exponential sums of the form
\begin{equation} \label{eq:genericsums1}
\frac{1}{M^l} \sum_{k=0}^M k^l \exp(2\pi i \beta_1 k+2\pi i \beta_2 k^2+\ldots+2\pi i \beta_S k^S)\,, 
\end{equation}

\noindent
where $3\le S \le 3+\log N/\log M$, $0\le l =O(\nu(K,j,\epsilon))$, and the (real) coefficients $\beta_{s}$, $3\le s\le S$, will typically assume values in  restricted subintervals near zero of decreasing length with $s$; see \textsection{4.3}. The appearance of higher degree exponential sums in (\ref{eq:genericsums1}) is simply a reflection of the growth in the size of the cubic coefficient during the first phase, which implies that expansion (\ref{eq:vancubbook1}) from earlier converges more slowly, leading to higher degree exponential sums.

Next, we perform a ``dyadic subdivision'' of the sums (\ref{eq:genericsums1}) so that we may restrict our attention to lengths of the form $M=2^n$. For each relevant value of $n$, $l$, and $S$, we precompute the sums (\ref{eq:genericsums1}) on a dense enough grid of points, taking into account that the coefficients $\beta_s$, $3\le s\le S$, are generally small in size.  There are also some relations among the coefficients $\beta_4,\beta_5,\ldots,\beta_S$, which are useful during the FFT precomputation in the case $M$ much smaller than $N$. It is shown in \textsection{4.4}  that the overall cost of the FFT precomputation is about $N^{4+o(1)}=K^{4\mu+o(1)}$ operations. Once the precomputation is finished, the sums (\ref{eq:genericsums1}) can be evaluated quickly elsewhere via Taylor expansions, as claimed in Theorem~\ref{cubicthm}. 

\section{The algorithm for cubic exponential sums with a small cubic coefficient}

Let $\lfloor x \rfloor$  denote the largest integer less than or equal to $x$ , $\lceil x \rceil$ denote smallest integer greater than or equal to $x$, $\{x\}$ denote $x-\lfloor x \rfloor$, and $\log x$ denote $\log_e x$. Let  $\exp(x)$ and $e^x$ both stand for the usual exponential function (they are used interchangeably). We define $0^0:=1$ whenever it occurs. We measure complexity (or time) by the number of arithmetic operations on numbers of $O((\log t)^{\kappa_0})$ bits required, where $\kappa_0$ is an absolute constant (not necessarily the same for different methods to compute zeta). An {\em{arithmetic operation}} means an addition, a multiplication, an evaluation of the logarithm of a positive number, or an evaluation of the complex exponential. In what follows, asymptotic constants are absolute unless otherwise is indicated. 

Let $\mu\in [0,1]$, $\epsilon \in (0,e^{-1})$, $0\le j$, $0<K$, $a_0\in [0,1)$, $b_0\in [0,1)$, and $c_0\in [0,K^{\mu-3}]$. As before, $\nu(K,j,\epsilon):=(j+1)\log (K/\epsilon)$, and we define $\Lambda(K,j,\epsilon):=50^4 \nu(K,j,\epsilon)^{6}$ say. We also define $F(K;a,b):=F(K,0;a,b)$. 

In this section, different occurrences of $K$, $j$, and $\epsilon$  will denote the same values. For this reason, we drop the dependence of $\Lambda$ and $\nu$ on $K$, $j$, and $\epsilon$, throughout \textsection{4}. We use the same computational model as the one described at the beginning of \textsection{2}. Arithmetic is performed using $O(\nu(K,j,\epsilon)^2)$-bits. And any implicit asymptotic constants are absolute, unless otherwise is indicated.

In order to spare the reader some straightforward and repetitious details, we will often use informal phrases such as ``It is possible to \textit{reduce} (or \textit{simplify}) the problem of computing the function $X_{K,j}(.)$ to that of computing the function $Y_{K,j}(.)$,''  or ``In order to compute the function $X_{K,j}(.)$, it \textit{is enough} (or \textit{suffices}) to compute the function $Y_{K,j}(.)$." This will mean there are absolute constants $\tilde{\kappa}_3$, $\tilde{\kappa}_4$, $\tilde{B}_1$, $\tilde{B}_2$, and $\tilde{B}_3$ (not necessarily the same on different occasions) such that for any positive $\epsilon < e^{-1}$, if the function $Y_{K,j}(.)$ can computed for any of the permissible values of its arguments to within $\pm\,\epsilon$, then the function $X_{K,j}(.)$ can be computed for any of the permissible values of its arguments to within $\pm\, \tilde{B}_1\,\nu(K,j,\epsilon)^{\tilde{\kappa}_3}\epsilon$ using at most $\tilde{B}_2\,\nu(K,j,\epsilon)^{\tilde{\kappa}_4}$ arithmetic operations on numbers of $\tilde{B}_3\,\nu(K,j,\epsilon)^2$ bits. The meaning of the phrase ``permissible values of the arguments'' will be clear from the context. 

Similarly, we frequently say ``the function $X_{K,j}(.)$ can be \textit{computed} (or \textit{evaluated}) \textit{efficiently} (or \textit{quickly}).'' This means there are absolute constants $\tilde{\kappa}_5$, $\tilde{\kappa}_6$, $\tilde{B}_4$, $\tilde{B}_5$, and $\tilde{B}_6$ (not necessarily the same on different occasions) such that for any positive $\epsilon < e^{-1}$,  the function $X_{K,j}(.)$ can be computed for any of the permissible values of its arguments to within $\pm\,\tilde{B}_4\,\nu(K,j,\epsilon)^{\tilde{\kappa}_5}\epsilon$ using at most $\tilde{B}_5\,\nu(K,j,\epsilon)^{\tilde{\kappa}_6}$ arithmetic operations on numbers of $\tilde{B}_6\,\nu(K,j,\epsilon)^2$ bits. 

We may assume $K>\Lambda$, otherwise the cubic sum can be evaluated directly in $O(\Lambda)$ operations on numbers of $O(\nu)$ bits. Notice by the conventions just presented, we will often abbreviate this, and similar statements, by saying that the cubic sum can be computed \textit{efficiently} or \textit{quickly} if $K<\Lambda$. 

As stated at the beginning of the section, $c_0\in [0,K^{\mu-3}]$. We may assume $\mu \ge 0$, because if $\mu< 0$ we have $cK^3\le 1$, so by a routine application of Taylor expansions, the term $\exp(2\pi i c k^3)$ can be reduced to a polynomial in $k$ of degree $O(\nu)$, plus an error of size $O(\epsilon/K)$ say. As it is clear the coefficients of this polynomial are quickly computable and are of size $O(1)$ each, then the cubic sum can be expressed as a linear combination of $O(\nu)$ quadratic sum, plus an error of size $O(\epsilon)$. And since each such quadratic sum can be computed efficiently via Theorem~\ref{quadraticthm}, so can the cubic sum. 

We may also assume $c_0 K^2<1/\Lambda^4$ (this is convenient to assume during the first phase of the algorithm in \textsection{4.1}). Because if $K^{-2}/\Lambda^4\le c_0\le K^{-2}$, then by a procedure completely similar to that used in describing Sch\"onhage method earlier, the cubic sum can be computed quickly provided an FFT precomputation costing $\le \Lambda K^4$ operations on numbers of $O(\nu)$ bits, and requiring $\le \Lambda K^{4}$ bits of storage, is performed. In particular, since $K^{-2}/\Lambda^4\le c_0$, then $\mu$ in the statement of Theorem~\ref{cubicthm} satisfies $1-4(\log\Lambda)/(\log K)\le \mu $. Therefore, $K^4 \le \Lambda^{16} K^{4\mu}$, and the claim of Theorem~\ref{cubicthm} over $K^{-2}/\Lambda^4\le c_0 \le K^{-2}$ holds anyway. 

\subsection{The first phase: decoupling the cubic and quadratic sums}

As before, let $Q(K,j;x,c_0):=H(K,j;x,0,c_0)$, so

\begin{equation}
Q(K,j;x,c_0):= \frac{1}{K^j}\sum_{k=0}^{K} k^j \exp(2\pi i  x k +2 \pi i c_0 k^3)\,.
\end{equation}

\noindent
Also define

\begin{equation}
S_{p,K,j,a_0,b_0} (N_0,c_0,\alpha_{0,0}):= \frac{1}{(N_0)^j}\int_{I_p} F(K;a_0- \alpha_{0,0}\, x,b_0) Q(N_0,j;x,c_0)\, dx\,, 
\end{equation}

\noindent
where $I_1:=[-1/2,-1/4)$,  $I_2:=[-1/4,1/4]$, and $I_3:=(1/4,1/2]$. By a straightforward calculation, we have

\begin{equation}\label{eq:stphasehkj}
\begin{split}
H(K,j;a_0,b_0,c_0) &=S_{1,K,j,a_0,b_0}(K,c_0,1)+S_{2,K,j,a_0,b_0}(K,c_0,1)\\
&+S_{3,K,j,a_0,b_0}(K,c_0,1)\,.
\end{split} 
\end{equation}

We do not expect the terms $S_1(.)$ and $S_3(.)$ in (\ref{eq:stphasehkj}) to be computationally troublesome because if $K$ is large enough, which we are assuming, then the derivative with respect to $k$ of  $x k + c_0 k^3$ (this is the exponent of the summand in $Q(.)$) never vanishes over $1/4\le |x|$, $0\le k\le K$, and $0\le c_0 \le K^{\mu-3}$, with $c_0K^2< 1/\Lambda^4$. So once $Q(.)$ is converted to an integral via the Euler-Maclaurin summation formula, as we plan to do, then the resulting cubic exponential integral will not contain any saddle-points. Thus, ultimately, we expect $S_1(.)$ and $S_3(.)$ can be expressed as a linear combination of a few quadratic exponential sums.

Indeed, on applying the Euler-Maclaurin summation formula to the cubic sum $Q(K,j;x,c_0)$ in $S_{3,K,j,a_0,b_0}(K,c_0,1)$, for instance, we obtain, via auxiliary lemma~\ref{lem:la1} in \textsection{5}, an expression of the form 

\begin{equation}\label{eq:subsec1add}
\begin{split}
S_{3,K,j,a_0,b_0}(K,c_0,1)=&\frac{1}{K^j}\int_{1/4}^{1/2} \int_0^K y^j \exp(2\pi i c_0 y^3+2\pi i xy) \,\times \\
& F(K;a_0- x,b_0) \, dy\, dx + R_{3,K,j,a_0,b_0}(K,c_0,1)\,,
\end{split}
\end{equation}

\noindent 
where $R_{3,K,j,a_0,b_0}(K,c_0,1)$ is a remainder function arising from the correction terms in the Euler-Maclaurin formula. 

We claim the the remainder $R_{3,K,j,a_0,b_0}(K,c_0,1)$ can be computed efficiently (in poly-log time). For as an immediate consequence of auxiliary lemmas~\ref{lem:la3} and~\ref{lem:la4} in \textsection{5}, the remainder $R_{3,K,j,a_0,b_0}(K,c_0,1)$ can be written as a linear combination of $\le \tilde{B}_7 \nu$ quadratic exponential sums, plus an error of size $\le \tilde{B}_8\,\nu^{\tilde{\kappa}_7}\, \epsilon/K^2$, where $\tilde{B}_7$, $\tilde{B}_8$, and $\tilde{\kappa}_7$, are absolute constants (notice an error size of $\le \tilde{B}_8\,\nu^{\tilde{\kappa}_7}\, \epsilon/K^2$ is small enough for purposes of proving Theorem~\ref{cubicthm}). The coefficients of said linear combination can be computed to within $\pm\,\nu^{\tilde{\kappa}_8}\,\tilde{B}_9\,\epsilon$ using $\le \tilde{B}_{10} \nu^{\tilde{\kappa}_9}$ operations on numbers of $\le \tilde{B}_{11} \nu^2$ bits, and each coefficient is of size $\le \tilde{B}_{12}$, where $\tilde{B}_9$, $\tilde{B}_{10}$, $\tilde{B}_{11}$, $\tilde{B}_{12}$, $\tilde{\kappa}_8$, and $\tilde{\kappa_9}$, are absolute constants. For simplicity, we will often abbreviate the above technical details by saying ``$R_{3,K,j,a_0,b_0}(K,c_0,1)$ can be written as a linear combination, with quickly computable coefficients each of size $O(1)$, of  $O(\nu)$ quadratic sum, plus an error of size $O(\epsilon/K^2)$.'' Now, since each quadratic sum can be computed efficiently via Theorem~\ref{quadraticthm}, then so can the remainder $R_{3,K,j,a_0,b_0}(K,c_0,1)$. 

As for the main term in (\ref{eq:subsec1add}), which is a double integral, we have by auxiliary lemma~\ref{lem:la5} that it too can be written as a linear combination, with quickly computable coefficients each of size $O(1)$, of $O(\nu)$ quadratic exponential sums, plus an error of size $O(\epsilon/K^2)$. The treatment of the term $S_{1,K,j,a_0,b_0}(K,c_0,1)$ in (\ref{eq:stphasehkj}) is almost identical to that of the term $S_{3,K,j,a_0,b_0}(K,c_0,1)$.

So it remains to tackle the term $S_{2,K,j,a_0,b_0}(K,c_0,1)$ in (\ref{eq:stphasehkj}). This is the computationally demanding term because it is where the cubic exponential integral obtained from the cubic sum $Q(.)$ can contain saddle-points. For simplicity, assume $K$ is a power of 2. The argument to follow is easily modifiable to the case $K$ not a power of 2. We define 

\begin{equation}
N_m:=K/2^m\,,\qquad c_m:=2^{3m}c_0\,,\qquad  \alpha_{0,m}:=2^{-m}\,.
\end{equation}

\noindent
Notice $S_{2,K,j,a_0,b_0}(K,c_0,1)=S_{2,K,j,a_0,b_0}(N_0,c_0,\alpha_{0,0})$. By splitting $Q(.)$ into a sum over the evens and a sum over the odds we obtain

\begin{equation}\label{eq:cub111}
Q(N_m,j;x,c_m)=2Q(N_{m+1},j;2x,c_{m+1})-Q(N_m,j;x+1/2,c_m)\,. 
\end{equation}

\noindent
By the definitions of $N_m$, $c_m$, and $\alpha_{0,m}$, coupled with the transformation (\ref{eq:cub111}) and the change of variable $x\leftarrow 2x$ applied to the integral with respect to $x$ in $S_2(N_m,c_m,\alpha_{0,m})$, we obtain

\begin{equation} \label{eq:it1}
\begin{split}
S_{2,K,j,a_0,b_0} (N_m,c_m,\alpha_{0,m}) &=  S_{2,K,j,a_0,b_0}(N_{m+1},c_{m+1},\alpha_{0,m+1}) \\
&+ S_{1,K,j,a_0,b_0}(N_{m+1},c_{m+1},\alpha_{0,m+1}) \\
&+S_{3,K,j,a_0,b_0}(N_{m+1},c_{m+1},\alpha_{0,m+1}) \\
&-S_{4,K,j,a_0+\alpha_{0,m+1},b_0}(N_m,c_m,\alpha_{0,m})\,. 
\end{split}
\end{equation}

\noindent
where the integral with respect to $x$ in $S_4(.)$ is taken over the interval $I_4:=(1/4,3/4]$. 

Again, by the auxiliary lemmas~\ref{lem:la1} through~\ref{lem:la5} in \textsection{5}, the functions $S_1(.)$, $S_3(.)$, and $S_4(.)$, on the right side of (\ref{eq:it1})  can be computed efficiently provided $c_m N_m^2<1/\Lambda$ say (this condition ensures the Euler-Maclaurin formula can approximate the cubic sum $Q(.)$ by an integral to within $O(\epsilon/K^2)$ using only $O(\nu)$ correction terms, which, since $1/4\le |x|$, ensures said integral never contains a saddle-point). 

So by repeating the transformation (\ref{eq:it1}) at most $\lceil \log_2 K \rceil$ times, we reach either $1/\Lambda< c_mN_m^2$ or $N_m<\Lambda$. The latter is a boundary point of the algorithm since the problem simplifies to computing a total of $N_m+1 \le \Lambda$ functions of the form $\int_{-1/2}^{1/2} \exp(2\pi i x n) F(K;a_0-\alpha_{0,m}\, x, b_0) \, dx$, where $0\le n \le N_m$ is an integer. As an immediate consequence of the proof of lemma~\ref{lem:la4} (see the calculations following (\ref{eq:done11}) there), such functions can be evaluated efficiently since they reduce to quadratic sums. So we may assume $\Lambda\le N_m$, and that the first phase is ended due to a failure of the condition $0\le c_mN_m^2 \le 1/\Lambda$. By the definitions of $c_m$ and $N_m$, a failure of this condition implies $ 1/\Lambda < (2^{3m} c) (K/2^m)^2$, and hence $\alpha_{0,m}=2^{-m} < \Lambda c K^2$. Recalling that  \mbox{$0\le c \le K^{\mu-3}$} by hypothesis, it follows $\alpha_{0,m}\le \Lambda K^{\mu-1}$, which in turn implies $N_m=\alpha_{0,m} K\le \Lambda K^{\mu}$. Notice also, since $c_0N_0\le 1/\Lambda^4$ by hypothesis, then $\alpha_{0,m}\le 1/\Lambda$. 

Put together, letting $N$, $\alpha_0$, and $c$ denote the values of $N_m$, $\alpha_{0,m}$, and $c_m$, respectively, at the end of the first phase, our task has been reduced to numerically evaluating (to within $\pm\,\nu^{\kappa}\,\epsilon$, for any absolute $\kappa$)  the function

\begin{equation}\label{eq:emmcc}
\frac{1}{N^j}\int_{-1/4}^{1/4} \int_0^N  y^j \exp(2\pi i c y^3-2\pi i xy)\,F(K;a_0+\alpha_0\, x,b_0) \, dy\, dx\,,
\end{equation}

\noindent
where, for later convenience, we made the change of variable $x\leftarrow -x$ in (\ref{eq:emmcc}). Here, $a_0 \in [0,1)$, $b_0 \in [0,1)$, and $N$, $c$, and $\alpha_0$, satisfy the bounds

\begin{equation} \label{eq:assump1}
1/\Lambda \le  c N^2 \le 2/\Lambda\,, \,\,\,\, \Lambda \le N\le \Lambda K^{\mu}\,, \,\,\,\, \Lambda/K \le \alpha_0 \le \min\{\Lambda K^{\mu-1},1/\Lambda\} \,.
\end{equation}

\subsection{The second phase: the algorithm for quadratic sums}

Each term in the quadratic sum $F(K;a_0+\alpha_0\, x,b_0)$ in (\ref{eq:emmcc}) has a saddle-point associated with it; see \textsection{4.3} for a precise formulation of this. If we extract the contribution of each saddle-point at this point of the algorithm (like done in \textsection{4.3} later), we obtain a sum of length $K$ terms, which is of the same length as the original cubic sum $H(K,j;a_0,b_0,c_0)$. But if we are able to cut the length of $F(K;a_0+\alpha_0\, x,b_0)$, then there will be fewer saddle-points to consider. 

To cut the length, we employ the algorithm of~\cite{Hi}. By the periodicity of the complex exponential we have

\begin{equation}\label{eq:periodcubic}
\begin{split}
F(K; a_0+ \alpha_0 \,x, b_0) &= F(K; a_0+ \alpha_0 \,x\pm\, 1/2,b_0 \pm\, 1/2)\\
&= F(K; a_0+ \alpha_0 \, x\pm\, 1/2,b_0 \mp \, 1/2)\,.
\end{split}
\end{equation}

\noindent
Using (\ref{eq:periodcubic}), it is not too hard to see given any pair $(\tilde{a}_0,\tilde{b}_0) \in [0,1)\times [0,1)$, and any $0\le \tilde{\alpha}\le 1/\Lambda$ say, there is a quickly computable pair $(\tilde{a}_1, \tilde{b}_1)$, depending only on $\tilde{a}_0$ and $\tilde{b}_0$, that satisfies 

\vspace{6pt}
\begin{enumerate}\itemsep10pt
\item[\textbf{N1.}] $\tilde{a}\in[0,2]$,
\item[\textbf{N2.}] $\tilde{b}_1\in [0,1/4]$,
\item[\textbf{N3.}] $\tilde{a}_1+ \tilde{\alpha}\, x \in (0,2)$ for all $x\in [-1/4,1/4]$,
\end{enumerate}

\vspace{8pt}
\noindent
and such that either 

\begin{equation}\label{eq:stposs}
F(K;\tilde{a}_0+\alpha_0\, x,\tilde{b}_0)= F(K;\tilde{a}_1+ \alpha_0\, x,\tilde{b}_1)\,,
\end{equation}

\noindent
or 

\begin{equation}\label{eq:ndposs}
F(K; \tilde{a}_0+\alpha_0 \, x,\tilde{b}_0)= \overline{F(K;\tilde{a}_1- \alpha_0 \, x,\tilde{b}_1)}\,.
\end{equation}

\noindent
The pair $(\tilde{a}_1, \tilde{b}_1)$ is not necessarily unique.

Without loss of generality, we may assume the original pair $(a_0,b_0)$ in (\ref{eq:emmcc}) already satisfies the normalization conditions N1 and N3, and that $b_0\in [-1/4,1/4]$. Also, for now, let us assume $\lceil a_0+ \alpha_0\, x\rceil < \lfloor a+ \alpha_0\, x + 2 |b_0| K \rfloor$ holds for all $x\in [-1/4,1/4]$. Notice this immediately implies $|b_0|\ge 1/(2K)$, so $|b_0|$ cannot be too small. Under such circumstances, lemma 6.6 in ~\cite{Hi} applies. If $b_0\in (0,1/4]$, that lemma yields:

\begin{equation} \label{eq:van0011}
\begin{split}
F(K; a_0+ \alpha_0\, x, b_0)=&\frac{1}{\sqrt{2b_0}}\, e^{ \pi i /4-\pi i (a_0+\alpha_0\, x)^2/(2b_0)} \,  F\left(\lfloor 2 b_0 K \rfloor;\frac{ a_0+ \alpha_0 \, x}{2 b_0},-\frac{1}{4 b_0}\right)\\
& +R(K,a_0+ \alpha_0\, x, b_0)+O(K^2\epsilon +e^{-K})\,, 
\end{split}
\end{equation}

\noindent
which is valid for any $\epsilon \in (0,e^{-1})$, and any $x\in [-1/4,1/4]$. And if $b_0\in [-1/4,0)$, which implies conjugation is needed to ensure condition N2 holds, we obtain the same formula as (\ref{eq:van0011}) except the right side is replaced by its conjugate, and $\alpha_0$, $b_0$, and $a_0$, are replaced by $-\alpha_0$, $-b_0$, and $1-a_0$ or $2-a_0$, respectively. In either case, the resulting remainder term  $R(K,a_0+\alpha_0 \, x, b_0)$ in (\ref{eq:van0011}) is fully described by lemma 6.6 in~\cite{Hi}, as we discuss here later. (It might be helpful to consult lemmas 6.6 and 6.7 in~\cite{Hi} at this point.) 

We can repeatedly apply formula (\ref{eq:van0011}) for as long as the analogue of the condition $\lceil a_0+\alpha_0\, x\rceil < \lfloor a_0+\alpha_0\, x+2|b_0| K\rfloor$ holds for all $x\in [-1/4,1/4]$. After $m$ such applications say, we arrive at an expression of the form:

\begin{equation} \label{eq:pervan}
\begin{split}
F(K;a_0+\alpha_0\, x,b_0)&= D_m \, e^{-2 \pi i\, \alpha_{1,m}\, x- 2 \pi i\, \alpha_{2,m}\, x^2}\, F\left(K_m; a_m+ \alpha_m\, x, b_m\right)  \\
&\quad + R_m(K;a_0,\alpha_0, x,b_0) +O(K^2 \epsilon+ e^{-K_m})\,,
\end{split}
\end{equation}

\noindent
where 

\begin{equation} \label{eq:rrmm}
\begin{split}
K_l \le 2b_0\,2b_1\,\ldots 2b_{l-1} K\,,\\
R_m(K;a_0,\alpha_0,x,b_0):=&\sum_{l=0}^{m-1} D_l\, e^{-2 \pi i\, \alpha_{1,l}\, x- 2 \pi i \,\alpha_{2,l}\, x^2}\, R(K_l,a_l+ \alpha_l\, x,b_l)\,. 
\end{split}
\end{equation}

For example, if $b_0\in (0,1/4]$ and $m=1$, then $K_0:=K$, $K_1:=\lfloor 2b_0 K_0\rfloor$, $\alpha_1:=\alpha_0/(2b_0)$, $\alpha_{1,0}=0$, $\alpha_{1,1}:= \alpha_{1,0}+a_0\alpha_1$, $\alpha_{2,0}=0$, $\alpha_{2,1}:=\alpha_{2,0}+b_0 (\alpha_1)^2$, $D_0=1$, and $D_1:= D_0\,(2b_0)^{-1/2} e^{\pi i /4-\pi i (a_0)^2/(2b_0)}$. As for $a_1$ and $b_1$, they are defined according to whether the normalization procedure expresses $F(K_1; a_0/(2b_0)+\alpha_1 x, -1/(4b_0))$ as  $F(K_1; \tilde{a}+\alpha_1 x, \tilde{b})$ or as $\overline{F(K_1; \tilde{a}-\alpha_1 x, \tilde{b})}$, for some $\tilde{a}$ and $\tilde{b}$ satisfying conditions N1, N2, and N3. In the former case we define $a_1:=\tilde{a}$ and $b_1:=\tilde{b}$, and in the latter case we define $a_1:=-\tilde{a}$ and $b_1:=-\tilde{b}$. It is understood if $b_l$ in (\ref{eq:rrmm}) is negative, the remainder $R(K_l, a_l+\alpha_l x, b_l)$ stands for $\overline{R(K_l, -a_l-\alpha_l x, -b_l)}$.  

The numbers $K_l$, $a_l$, $b_l$, $\alpha_l$, $\alpha_{1,l}$, $\alpha_{2,l}$, and $D_l$ are quickly computable. We only need the following properties for them, which are valid for $1\le l\le m$, where $m$ denotes the number of repetitions of formula (\ref{eq:van0011}), or its conjugate analogue, so far: 

\begin{equation} \label{eq:obs001}
\begin{split}
& \alpha_{l-1}/\alpha_l = |2b_{l-1}|\le 1/2 \,, \qquad |D_l| = 1/\sqrt{|2b_0\,2b_1\,\ldots 2b_{l-1}|}\le \sqrt{K}\,,\\
& 0< \alpha_{1,l} = \sum_{r=1}^{l} |a_{r-1}| \alpha_r < 4\alpha_l\,,\qquad  |\alpha_{2,l}| = \sum_{r=1}^{l} |b_{r-1}| (\alpha_r)^2  < \min\{\alpha_l,(\alpha_l)^2\}\,. 
\end{split}
\end{equation}

Each application of formula (\ref{eq:van0011}) reduces the length of the quadratic sum by a factor of about $2b_0\le 1/2$, but it also multiplies the size of $\alpha_0$ by a factor of $1/(2b_0)\ge 2$. So during the second phase, the length of quadratic sum decreases, while the size of the parameter $\alpha_0$ grows. Eventually, formula (\ref{eq:van0011}) is no longer useful because, essentially, due to the growth in $\alpha_0$, the length of the new quadratic sum will start to depend strongly on $x$. The precise point at which we stop applying formula (\ref{eq:van0011}) is determined by the following criteria: let $m_0$ be the first non-negative integer for which at least one  of the following conditions fails (notice multiple conditions can fail at the same time):

\vspace{6pt}
\begin{enumerate}\itemsep10pt
\item[\textbf{C1.}] $\alpha_{m_0} \le 1/\Lambda\,,$
\item[\textbf{C2.}] $\Lambda \le  K_{m_0}\,,$
\item[\textbf{C3.}] $\left\lceil |a_{m_0}|+\alpha_{m_0} x \right\rceil < \left \lfloor |a_{m_0}|+ \alpha_{m_0} x +2|b_{m_0}| K_{m_0} \right \rfloor$ for some $x\in [-1/4,1/4]$.
\end{enumerate}

\vspace{8pt}
\noindent
Then the second phase is ended after exactly $m_0$ applications of formula (\ref{eq:van0011}), or its conjugate analogue. We observe since $2b_l\le 1/2$, then $K_{l+1} \le K_l/2$. So by construction, $m_0 \le \log_2 K$. 

A failure of condition C2 or condition C3 is not hard to handle, and in fact represents a boundary points of the algorithm (while a failure of condition C1 is substantially more difficult to deal with, and will occupy most of this remainder of this subsection). For the former means $K_{m_0}$ is not large enough, and the latter means $b_{m_0}$ is too small. If $K_{m_0}$ is not large enough, then, ignoring the remainder $R_{m_0}(.)$ for the moment, we need to deal with the sum

\begin{equation}\label{eq:ext2nds}
\begin{split}
\sum_{k=0}^{K_{m_0}} & \exp(2\pi i\, a_{m_0} k+2\pi i\, b_{m_0} k^2) \,\frac{D_{m_0}}{N^j}\int_{0}^{N} y^j \, \exp(2\pi i c y^3) \,\times\\
&\qquad \int_{-1/4}^{1/4} \exp\left(2\pi i yx+2\pi i (\alpha_{m_0}\, k - \alpha_{1,m_0})x -2\pi i \alpha_{2,m_0}\,x^2\right) \, dx\, dy\,.
\end{split}
\end{equation}

\noindent
Since $K_{m_0}\le \Lambda$, it suffices to deal with this sum term by term. Lemma~\ref{lem:la6} shows each term in (\ref{eq:ext2nds}) can indeed be computed efficiently (because, essentially, the integral over $y$ in each term contains at most one saddle-point, and there are only $O(\Lambda)$ terms). And if condition C3 is the one that fails (so $b$ is too small), then the Euler-Maclaurin formula can be applied to $F(.)$, which leads to a triple integral 

\begin{equation}\label{eq:emcinttt}
\begin{split}
\frac{1}{N^j}\int_{0}^{N} y^j \exp(2\pi i c y^3) \int_{-1/4}^{1/4} \exp\left(2\pi i yx-2\pi i \alpha_{1,m_0} x -2\pi i \alpha_{2,m_0}\,x^2\right)\, \times \\
\int_0^{K_{m_0}} \exp(2\pi i \alpha_{m_0}  x z+ 2\pi i a_{m_0} z+2\pi i b_{m_0} z^2) \, dz \, dy\, dx\,. 
\end{split}
\end{equation}

\noindent
plus a remainder term arising from the correction terms in the Euler-Maclaurin formula. Cauchy's Theorem as well as saddle-point techniques very similar to those carried out in \textsection{4.3} later allow us to reduce (\ref{eq:emcinttt}) to double integrals (with respect to $x$ and $y$) of the type handled by Lemma~\ref{lem:la6}. The calculations involved are tedious but elementary to do, and they involves considering several cases; see the discussion following (\ref{eq:genhif2}) for instance. 

Put together, we may assume conditions C2 and C3 still hold by the last iteration (that is, $K_{m_0}>\Lambda$ and $\alpha_{m_0} > 1/\Lambda$), and the algorithm halts due to a failure of condition C1. In other words, our task has been reduced to showing how to deal with a sole failure of condition C1, and also to dealing with the remainder functions

\begin{equation}\label{eq:rrmm12}
\begin{split}
\frac{D_l}{N^j} \int_{0}^{N} y^j \exp(2\pi i c y^3) \int_{-1/4}^{1/4} & \exp(-2\pi i x y -2 \pi i \alpha_{1,l}\, x- 2 \pi i \alpha_{2,l}\, x^2) \times  \\
&\quad R(K_l,a_l+ \alpha_l\,x,b_l)\, dx\, dy\,, 
\end{split}
\end{equation}

\noindent
for $0\le l < m_0$. 

Let us deal with the remainder functions first. We will show how to efficiently compute (\ref{eq:rrmm12}). To this end, suppose $b_l$ (hence $a_l$) is positive. Let $[w,z)$ be any subinterval of $[-1/4,1/4)$ such that $\lfloor a_l+ \alpha_l x+ 2 b_l K_l\rfloor$ and $\lceil a_l+ \alpha_l x \rceil$ are constant for all $x\in [w,z)$. Since $\alpha_l<1/\Lambda$, the interval $[-1/4,1/4)$ can be written as the union of at most 4 such subintervals, and these subintervals can be determined quickly. Similarly, if $b_l$ (hence $a_l$) is negative, we choose the subinterval $[w,z)\subset [-1/4,1/4)$ so that $\lfloor -a_l- \alpha_l x- 2 b_l K_l\rfloor$ and $\lceil -a_l- \alpha_l x \rceil$ are constant for all $x\in [w,z)$. Since the treatments of these possibilities are analogous, let us focus out attention on the case $b_l$ is positive. 

Let $ r,d \in [0, 1000\,\nu(K,\epsilon)]$ be integers, and let $\omega_{a_0+\alpha_0\, x}:=\omega_{a_0+\alpha_0\, x, b_0, K_0}=\{a_0+\alpha_0\, x+2b_0K_0\}$, where $\{y\}$ denotes the fractional part of $y$, and let $\omega_{1,a_0+\alpha_0\, x}:=\lceil a_0+\alpha_0\, x\rceil -(a_0+\alpha_0\, x)$. Then by lemma~6.7 in~\cite{Hi}, we have that over $x\in [w,z)$ the remainder $R(K_l,a_l+\alpha_l\,x,b_l)$ can be written as a linear combination of the functions

\begin{equation}\label{eq:remform1}
 x^r\,, \qquad  x^r \exp\left(2\pi i \alpha_l x K_l \right)\,, \qquad  \exp\left[2\pi i P\,\alpha_{l+1} x - 2\pi i b_l\,(\alpha_{l+1})^2 \, x^2 \right]\,,  \\
\end{equation}

\vspace{10pt}

\noindent
where $P\in \{-1,0,K_{l+1},K_{l+1}+1 \}$, and the functions

\begin{equation}\label{eq:remform2}
\begin{split}
&\exp\left[ 2\pi i \,\omega_{a_l +\alpha_l x}\,Q - 2\pi (1-i)r\, \frac{\omega_{a_l+ \alpha_l x}}{\sqrt{2b_l}}\right]\times \\
& \qquad\qquad\qquad\qquad \int_0^1 t^d \exp\left[ - 2\pi (1- i)\, \frac{\omega_{a_l+ \alpha_l x}}{\sqrt{2b_l}}t-2 \pi r t\right]\,dt\,, 
\end{split}
\end{equation}

\vspace{10pt}

\noindent
where $Q\in \{0,K_l\}$, and the functions

\begin{equation}\label{eq:remform3}
(\omega_{a_l+\alpha_l x})^r\, \exp\left[2\pi i \, \omega_{a_l+ \alpha_l x}\, L -  2\pi \,\omega_{a_l+ \alpha_l x}\,R \right]\,, 
\end{equation}

\vspace{10pt}

\noindent
where $ L, R \in [K_l,  K_l+ 1000 \,\nu(K_l,\epsilon)]$ say, as well as functions of the same form, but with $\omega_{a_l+\alpha_l x}$ possibly replaced  by $1-\omega_{a_l+\alpha_l x}$, or $\omega_{1,a_l+ \alpha_l x}$, or $1-\omega_{1,a_l+\alpha_l x}$, plus an error term bounded by $O(\Lambda K^{-2} \epsilon)$; the length of the linear combination is $O(\Lambda)$ terms, and the coefficients in the linear combination can all be computed efficiently. We remark that, using the notation and terminology of~\cite{Hi}, the functions (\ref{eq:remform1}) arise from bulk terms like $J(K,j;M,\omega_{a_l+\alpha_l x},b_l)$, while the functions (\ref{eq:remform2}) arise from boundary terms like $\tilde{I}_{C_7}(K,j;\omega_{a_l+\alpha_l x},b_l)$. 

By choice of $[w,z)$, it is straightforward to see there are two numbers $\lambda:=\lambda(a_l,\alpha_l)$ and $\lambda':=\lambda'(a_l,\alpha_l)$, which can be computed quickly, such that $\omega_{a_l+ \alpha_l x} =\lambda+ \alpha_l x$ and $\omega_{1,a_l+\alpha_l x}=\lambda'-\alpha x$, for all $x\in [w,z)$ (notice $0\le \lambda+\alpha_l x \le 1$ and $0\le \lambda'-\alpha_l x\le 1$ over $x\in [w,z)$). Substituting $\lambda+\alpha_l x$ for $\omega_{a_l+\alpha_l x}$ we see that the functions (\ref{eq:remform2}) and (\ref{eq:remform2}) can be expressed explicitly in terms of $x$. It is plain such substitutions extend completely similarly if, instead of $\omega_{a_l+\alpha_l x}$, the functions (\ref{eq:remform2}) and (\ref{eq:remform3}) involve  $\omega_{1,a_l+\alpha_l x}$, or $1-\omega_{a_l+\alpha_l x}$, or $1-\omega_{1,a_l+\alpha_l x}$.

Therefore, in order to enable an efficient computation of (\ref{eq:rrmm12}), it suffices to show how to efficiently compute the expressions arising from replacing $R(K_l,a_l+ \alpha_l\,x,b_l)$ in (\ref{eq:rrmm12}) by any of the functions in (\ref{eq:remform1}), (\ref{eq:remform2}), or (\ref{eq:remform3}). 

Substituting any of the functions (\ref{eq:remform1}) and (\ref{eq:remform3}) for $R(K_l,a_l+ \alpha_l\,x,b_l)$ leads to integrals that can be computed efficiently by a direct application of lemma~\ref{lem:la6}, provided one appeals to the set of observations (\ref{eq:obs001}), and the fact $0\le \omega_{a_l+ \alpha_l x} =\lambda_1+ \alpha_l x\le 1$ over $x\in [w,z)$. In fact, lemma~\ref{lem:la6} can handle substantially more general integrals than those arising from (\ref{eq:remform1}) and (\ref{eq:remform3}), and can be generalized yet more. 

As for the functions (\ref{eq:remform2}), they produce somewhat more complicated expressions, involving a triple integral:

\begin{equation} \label{eq:genhif2}
\begin{split}
\frac{D_l}{N^j}\int_{0}^{N}  y^j e^{2\pi i c y^3} \int_w^z & e^{2\pi i Q \lambda \alpha_l x-2\pi (1-i)r\frac{\lambda+\alpha_l x}{\sqrt{2b_l}} -2 \pi i y x-2\pi i \alpha_{1,l} x- 2 \pi i \alpha_{2,l} x^2} \, \times \\
&\int_0^1 t^d e^{- 2\pi (1-i)\frac{\lambda +\alpha_l x}{\sqrt{2b_l}}\,t-2\pi r t}\,dt \, dx\, dy\,,
\end{split}
\end{equation}

\noindent
which are not of the type immediately handled by lemma~\ref{lem:la6}. Nevertheless, expression (\ref{eq:genhif2}) can still be evaluated efficiently via that lemma. For one can first apply the change of variable $x\leftarrow \lambda+\alpha_l x$ to (\ref{eq:genhif2}), so the interval of integration with respect to $x$ is transformed to $[w_1,z_1]$, where $w_1:=w_{1,\lambda,\alpha_l,w}= \lambda+\alpha_l w$ and $z_1:=z_{1,\lambda,\alpha_l,w}=\lambda+\alpha_l z$. One then considers the following two cases. 

On the one hand, if $w_1>\Lambda\sqrt{2b_l}$ say, so $\sqrt{2b_l}/w_1< 1/\Lambda$, we evaluate the integral with respect to $t$ explicitly, which leads to a polynomial in $\sqrt{2b_l}/x$ of degree $d$. We then make the change of variable $x\leftarrow 2x/w_1$, which transforms the interval of integration with respect to $x$ to $[2,2z_1/w_1]$. Observing $z_1/w_1=O(\sqrt{K_l})$, one can divide the interval of integration with respect to $x$ into $O(\log K_l)$ consecutive subintervals of the form $[A_n,A_n+\Delta_n)$, where $2\le A_n<2z_1/w_1$ and $\Delta_n= \lfloor A_n/2\rfloor$, except in the final subinterval where $\Delta_n$ is possibly smaller. In any case, we always have $\Delta_n < A_n/2$. So now, the change of variable $x\leftarrow x-A_n$, followed by an application of Taylor expansions to the new term $\sqrt{2b_l}/(w_1(x+A_n))$, where now $0\le x\le \Delta_n\le A_n/2$, can be used to write $\sqrt{2b_l}/(A_nw_1(1+x/A_n))$ as a polynomial in $x/A_n$ of degree bounded by $O(\nu)$, plus an error of size $O(\epsilon/K^2)$ say. Together, this procedure yields a linear combination, with quickly computable coefficients each of size $O(1)$, of $O(\Lambda)$ integrals of the type directly handled by lemma~\ref{lem:la6}. 

On the other hand, if $w_1 \le \Lambda\sqrt{2b_l}$, then one separately deals with the integral over $x\in [\Lambda\sqrt{2b_l}, z_1]$ as was just described,  while over $x\in [w_1,\Lambda\sqrt{2b_l}]$  one expresses the cross-term $e^{-2\pi (1- i) x t/\sqrt{2b_l}}$, which was obtained after our very first change of variable $x\leftarrow \lambda+\alpha_l x$, as a polynomial in $x$ of degree $O(\nu)$, with coefficients depending on $t$, plus an error of size $O(\epsilon/K^2)$. Specifically, we apply the preliminary change of variable $t\leftarrow \lceil \Lambda^2\rceil t$ say, then divide the interval of integration with respect to $t$ into $\lceil \Lambda^2\rceil$ consecutive subintervals $[n,n+1)$. Over each such subinterval we apply the change of variable $t\leftarrow t-n$. Last, by a routine application of Taylor expansions, followed by integrating explicitly with respect to $t$, we are led to a linear combination, with quickly computable coefficients each of size $O(1)$, of $O(\Lambda^2)$ integrals of the type directly handled by lemma~\ref{lem:la6}. 

To summarize, let $M:=K_{m_0}$, $a:=a_{m_0}$, $D:=D_{m_0}$, $b:=b_{m_0}$, $\alpha:=\alpha_{m_0}$, $\alpha_1:=\alpha_{1,m_0}$, and $\alpha_2:=\alpha_{2,m_0}$. Notice $N:=\alpha_0 K$ and $\alpha:= \alpha_0/|2b_0\,2b_1\,\ldots \,2b_{m_0}|$ by definition, and $\Lambda \le M\le |2b_0\,2b_1\,\ldots\, 2b_{m_0}| K$ by construction. From this it follows $\Lambda \le M\le N/\alpha$. Also, we may assume $m_0> 0$, otherwise condition C3 fails before entering the second phase, in which case the question is reduced, via the Euler-Maclaurin summation formula, to computing an integral of the form (\ref{eq:emcinttt}), which can be done efficiently, as described earlier. Thus, our task has been reduced to evaluating the expression: 

\begin{equation} \label{eq:2e1}
\begin{split}
\frac{D}{N^j} \int_{0}^{N} y^j \exp(2\pi i c y^3) \int_{-1/4}^{1/4}  & \exp(-2\pi i y x -2 \pi i \alpha_1 x-2\pi i \alpha_2 x^2)\, \times  \\
& F(M; a+ \alpha  x, b) \, dx\, dy\,, 
\end{split}
\end{equation}

\noindent
where $a\in [0,1)$, $b\in [0,1)$, and by (\ref{eq:assump1}), (\ref{eq:obs001}), the remarks preceding (\ref{eq:2e1}), as well as the bound  $1/K^3 \le |\alpha_2|$, which is easy to show, we have:

\begin{equation} \label{eq:assump2}
\begin{split}
& \Lambda \le N \le \Lambda K^{\mu} \,, \qquad \Lambda \le M\le N/\alpha \,, \qquad 1/\Lambda\le c N^2 \le 2/\Lambda\,,\\ 
&  1/\Lambda \le \alpha \le N/M\,, \qquad 0\le  \alpha_1 \le 4\alpha \,, \qquad 1/K^3 \le |\alpha_2| \le \alpha \,.
\end{split}
\end{equation}

\subsection{Some saddle-point calculations}

In this subsection, we extract the saddle-point contribution associated with  the terms of $F(M; a+ \alpha  x, b)$ in the double integral (\ref{eq:2e1}). There are about $M$ saddle-points.

Since $\alpha_2$ can assume values in a symmetric interval about zero, then without loss of generality we may replace $\alpha_2$ by $-\alpha_2$. Also, let us drop the constant $D$ in front of  (\ref{eq:2e1}) since it is bounded by $\sqrt{K}$, and the methods we present permit the evaluation of (\ref{eq:2e1}) to within $\pm\, K^{-d}$ for any fixed $d>0$ anyway. So the expression we wish to compute can be written explicitly as:

\begin{equation} \label{eq:3e1}
\begin{split}
\frac{1}{N^j} \int_{0}^{N} y^j \exp(2\pi i c y^3) \int_{-1/4}^{1/4} & \exp(2 \pi i \alpha_{2} x^2-2 \pi i \alpha_{1} x-2\pi i y x) \times   \\
&\sum_{k=0}^M \exp(2\pi i (a+\alpha x)k+2\pi i b k^2) \, dx\, dy\,.
\end{split}
\end{equation}

We split the sum over $k$ in (\ref{eq:3e1}) into three subsums (some of which possibly empty): a bulk subsum consisting of the terms \mbox{$\lfloor \Lambda^2\rfloor< k< M-\lfloor \Lambda^2\rfloor$}, and two tail subsums consisting of the terms $0\le k\le \lfloor \Lambda^2\rfloor$ and \mbox{$M-\lfloor \Lambda^2\rfloor \le k\le M$}. By a direct application of lemma~\ref{lem:la6}, each term in the tail subsums can be computed efficiently (each such term contains at most one saddle-point with respect to $y$ and there are only $O(\Lambda^2)$ terms). We remark the reason we single out the tail subsums is technical and it is to simplify the proof of lemma~\ref{lem:la7} and the calculation of integrals (\ref{eq:ikjj10}) later.

Therefore, we only need to deal with the bulk subsum. The domain of integration with respect to $x$ for each term in the bulk can be extended to $(-\infty,\infty)$ because by a direct application of lemma~\ref{lem:la7} the sum of the integrals over the extra pieces $x\in (-\infty,-1/4)$ and $x\in (1/4,\infty)$ can be computed efficiently (because over these pieces, the integral with respect to $y$ contains no saddle-points, so the expression can be reduced to quadratic exponential sums). Each term in the bulk sum thus becomes

\begin{equation}  \label{eq:saddle1}
\begin{split}
&\frac{1}{N^j} \int_{0}^{N} y^j  \exp(2\pi i  c y^3)\int_{-\infty}^{\infty} \exp\left(2\pi i \alpha k x-2\pi i xy-2\pi  i \alpha_1 x+2\pi i \alpha_2 x^2\right)  \, dx\, dy  \\
&= \frac{\exp\left(\mathop{\mathrm{sign}}(\alpha_2) \frac{i\pi}{4}\right)}{\sqrt{2|\alpha_2|}N^j} \int_{0}^{N}y^j \exp\left(2\pi i f_k(y) \right)\, dy\,, 
\end{split}
\end{equation}

\noindent
where $\mathop{\mathrm{sign}}(x):=x/|x|$ for $x\ne 0$, and

\begin{equation}
f_k(y):=f_k(y,c,\alpha,\alpha_1,\alpha_2)= c y^3-\frac{y^2}{4\alpha_2}+\frac{\alpha k- \alpha_1}{2 \alpha_2}y - \frac{(\alpha k- \alpha_1)^2}{4 \alpha_2 }\,,
\end{equation}

\noindent
and we used the easily-provable formula

\begin{equation}
\int_{-\infty}^{\infty} e^{2\pi i x t+2\pi i y t^2} \,dt= \frac{e^{\mathop{\mathrm{sign}}(y)\pi i /4}}{\sqrt{2|y|}}e^{-2\pi i x^2/(4y)},\qquad x, y\in \mathbb{R}, y\ne 0\,, 
\end{equation}

\noindent
We want to extract the saddle point contribution from (\ref{eq:saddle1}). To this end, define the saddle-points

\begin{equation} 
y_k:=y_k(c,\alpha,\alpha_1,\alpha_2)=\frac{1}{12c\alpha_2} \left(1-\sqrt{1-24c\alpha_2(\alpha k- \alpha_1)}\right)\,. 
\end{equation}

\noindent
Notice by choice of $y_k$, we have $f'_k(y_k)=0$. In what follows, it might be helpful to keep in mind the bound

\begin{equation}
|24c\alpha_2(\alpha k- \alpha_1)| \le \frac{48}{M \Lambda}\,,
\end{equation}

\noindent
which follows from assumptions (\ref{eq:assump2}). The integral (\ref{eq:saddle1}) can be written as

\begin{equation}
\frac{\exp\left(\mathop{\mathrm{sign}}(\alpha_2) \frac{i\pi}{4}\right)}{\sqrt{2|\alpha_2|}N^j}  \exp\left(2\pi i f_k(y_k)\right) \int_{0}^{N} y^j \exp\left(2\pi i h_k(y-y_k)\right)\, dy \,, 
\end{equation}
\begin{equation} 
h_k(y):=h_k(y,c,\alpha_2)= c y^3 +\left(3cy_k-\frac{1}{4\alpha_2} \right)y^2 \,. 
\end{equation}

\noindent
By elementary algebraic manipulations and Taylor expansions, we obtain

\begin{equation} \label{eq:ssum}
\begin{split}
f_k(y_k)&=  \frac{1}{864 c^2\alpha_2^3}\left( \left(1- 24c\alpha_2(\alpha k-\alpha_1)\right)^{3/2}- 1+ 36c\alpha_2(\alpha k-\alpha_1)\right.  \\
&\left.- 216 c^2 \alpha_2^2 (\alpha k- \alpha_1)^2  \right)  \\
&=:\sum_{s=0}^{\infty}d_s k^{s}\,,
\end{split}
\end{equation}

\noindent
where

\begin{equation} \label{eq:altform}
\begin{split}
&d_s:=d_{s,c,\alpha,\alpha_1,\alpha_2}= (24)^s c^{s-2} \alpha_2^{s-3} \alpha^{s} q_s\,,\\
& q_s :=q_{s,c,\alpha_1,\alpha_2}= \sum_{l=0}^{\infty}  \binom{l+s}{s}g_{l+s} (-1)^l (24)^{l}  c^{l}\alpha_2^{l}\alpha_1^l  \,,
\end{split}
\end{equation}

\noindent
and where  $g_0=0$, $g_1 = 0$, $g_2 = 0$, $|g_l| \le 1$, and $g_l$ depends on $l$ only. Notice since $|q_s| \le 2^{s+1}$, $|\alpha_2|\le \alpha$,  $cN^2 \le 1/\Lambda$, and $\alpha k\le N$, it follows

\begin{equation} \label{eq:bdd}
\left|d_{s+3} M^{s+3}\right| \le 2(48)^3 c \alpha^3 k^3  (48)^s c^s \alpha^{2s} M^s < \frac{N}{4 M^s}\qquad  \textrm{for  $s\ge 0$}\,.
\end{equation}

\noindent
Also, each $q_s$, hence $d_s$, can be computed efficiently. Now define

\begin{equation}\label{eq:ikjsaddle}
I_{k,j}:=I_{k,j,N,c,\alpha,\alpha_1,\alpha_2}=\frac{1}{\sqrt{2|\alpha_2|}N^j}\int_{0}^{N} y^j \exp\left(2\pi i  h_k(y-y_k)\right)\, dy\,.  
\end{equation}

\noindent
Then the sum (\ref{eq:3e1}) has been reduced to  a sum of the form

\begin{equation} \label{eq:contr}
\sum_{k=\lfloor \Lambda^2\rfloor}^{M-\lfloor \Lambda^2\rfloor} \exp\left(2\pi i a k+2\pi i b k^2+2\pi i f_k(y_k)\right) I_{k,j}\,. 
\end{equation}

By our assumptions on $N$, $M$, $c$, $\alpha$, $\alpha_1$, and $\alpha_2$, and the bound $\lfloor \Lambda^2 \rfloor\le  k\le M-\lfloor \Lambda^2\rfloor$, we have $\alpha(\lfloor \Lambda^2\rfloor -6) \le  y_k \le N- \alpha(\lfloor \Lambda^2\rfloor+6)$, and so $0<y_k<N$. So let us consider the integral (\ref{eq:ikjsaddle}) over the subintervals $[0,y_k)$ and $[y_k,N]$ separately. We use Cauchy's theorem to replace the contour $\{y\,:\,0\le y<y_k\}$, for instance, with $\{ye^{\pi i/4}\,:\,0\le y< \sqrt{2}y_k\}$ and $\{y+iy_k\,:\,0\le y<y_k\}$ (or their conjugates, appropriately oriented, depending on whether $\alpha_2$ is negative or positive, respectively).  Taking into account the easily-deducible facts $h_k(-y_k)= -(\alpha k-\alpha_1)^2/(4\alpha_2)-f_k(y_k)$ and $h_k(N-y_k)=f_k(N)-f_k(y_k)$, combined with suitable applications of Taylor expansions, one finds

\begin{equation}\label{eq:ikjj10}
\begin{split}
I_{k,j}&=\sum_{l=0}^L z_l \frac{k^l}{M^l}+e^{2\pi i a_1 k+2\pi i b_1 k^2-2\pi i f_k(y_k)} \sum_{l=0}^L w_l \frac{k^l}{M^l}   \\
&+e^{2\pi i a_1 k+2\pi i b_1 k^2-2\pi i f_k(y_k)} \sum_{l=0}^L \tilde{w}_l \frac{1}{k^l} +e^{2\pi i a_2 k+2\pi i b_1 k^2-2\pi i f_k(y_k)} \sum_{l=0}^L v_l \frac{k^l}{M^l}  \\
&+e^{2\pi i a_2 k+2\pi i b_1 k^2-2\pi i f_k(y_k)} \sum_{l=0}^L \tilde{v}_l \frac{1}{k^l} +O(\Lambda K^{-2} \epsilon)\,, 
\end{split}
\end{equation}

\noindent
where $L=O(\nu)$, and $a_1:=a_{1,\alpha,\alpha_1,\alpha_2}$, $a_2:=a_{2,N,\alpha,\alpha_1,\alpha_2}$, and $b_1:=b_{1,\alpha,\alpha_2}$, are real numbers of size $O(K^3)$ that are quickly computable. The coefficients $z_l:=z_{l,N,c,\alpha,\alpha_1,\alpha_2}$, $w_l:=w_{l,N,c,\alpha,\alpha_1,\alpha_2}$, and $v_l:=v_{l,N,c,\alpha,\alpha_1,\alpha_2}$ are quickly computable and bounded by $O(1)$. And the coefficients $\tilde{w}_l:=\tilde{w}_{l,N,c,\alpha,\alpha_1,\alpha_2}$ and $\tilde{v}_l:=\tilde{v}_{l,N,c,\alpha,\alpha_1,\alpha_2}$ are of size $O(\Lambda^l)$ and are also quickly computable. In particular, on substituting expression (\ref{eq:ikjj10}) back into (\ref{eq:contr}), we see the sum (\ref{eq:contr}) is equal to

\begin{equation} \label{eq:finsum}
\sum_{l=0}^L \frac{z_l}{M^l}\sum_{k=\lfloor \Lambda^2\rfloor}^{M-\lfloor \Lambda^2\rfloor} k^l \exp(2\pi i a k+2\pi i b k^2+2\pi i f_k(y_k))\,, 
\end{equation}

\noindent
plus a linear combination, with quickly computable coefficients, of either $O(\nu)$ quadratic sums of length $\le M$ terms, or $O(\nu)$ sums of the type discussed in \textsection{5} of~\cite{Hi} of length $\le M$ terms (the latter also reduce to usual quadratic sums; see~\cite{Hi}). And all such sums can be computed efficiently via Theorem~\ref{quadraticthm}.  

Let $S:=3+ \left\lfloor \log N/\log M\right\rfloor$. Then by the Taylor expansion of $f_k(y_k)$ as a polynomial in $k$, given in (\ref{eq:ssum}) and (\ref{eq:altform}), we have

\begin{equation}
f_k(y_k)=\sum_{s=0}^S d_s k^s+ \sum_{s=S+1}^{\infty} d_s k^s\,.
\end{equation}

\noindent
Also, by estimate (\ref{eq:bdd}) we have $|d_s M^s|\le N/M^{s-3}$. Since $0\le k\le M$, the tail $\sum_{s=S+1}^{\infty} d_sk^s=O(N/M^{S-2})=O(1)$. So the tail can be routinely eliminated from the exponent in (\ref{eq:finsum}) via Taylor expansions; see \textsection{2} for a similar calculation. This transforms (\ref{eq:finsum}) into a linear combination, with quickly computable coefficients each of size $O(1)$, of $O(\nu^2)$ sums of the form

\begin{equation} \label{eq:finsum1}
\frac{1}{M^l}\sum_{k=0}^{M} k^l \exp(2\pi i \beta_1 k+2\pi i \beta_2 k^2+\ldots+2\pi i \beta_S k^S)\,, 
\end{equation}

\noindent
where by (\ref{eq:ssum}), (\ref{eq:altform}), (\ref{eq:bdd}), and the periodicity of the complex exponential, we have

\begin{equation} \label{eq:fincond1}
\begin{split}
&0\le l =O(\nu)\,, \quad \Lambda\le M \le \Lambda N \le \Lambda^2 K^{\mu}\,, \\
&S:=3+ \left\lfloor \log N/\log M\right\rfloor\,,\quad |\beta_s| \le \min\left\{ \frac{1}{2}, \frac{N}{4M^{2s-3}}\right\}\,.
\end{split}
\end{equation} 

\noindent
Notice that, for simplicity, we extended the range of summation in (\ref{eq:finsum1}) to include the tails $0\le k<\lfloor \Lambda^2\rfloor$ and $M-\lfloor \Lambda^2\rfloor<k\le M$ since these extra subsums can be computed efficiently (they involve $O(\Lambda^2)$ terms only). 

\subsection{The FFT precomputation}

We will show any sum of the form  (\ref{eq:finsum1}) satisfying conditions (\ref{eq:fincond1}) can be computed efficiently provided we perform a precomputation costing $\le 16 \Lambda^5 N^4$ operations on numbers of $O(\nu^2)$ bits, and requiring $\le 16 \Lambda^5 N^4$ bits of storage. (Notice by conditions (\ref{eq:fincond1}) we have $N\le \Lambda K^{\mu}$, and so $N^4\le \Lambda^4 K^{4\mu}$.) 

More specifically, our plan is to precompute the sum (\ref{eq:finsum1}) for values of its arguments specified by conditions (\ref{eq:fincond1}) on a dense enough grid of points so its evaluation elsewhere can be done quickly. To this end, let 

\begin{equation}\label{eq:rndef}
n:=n(M)=\lfloor \log_2 M\rfloor\,,\quad R:=R(N,\Lambda)=\lceil \log_2 (\Lambda N+1)\rceil\,.
\end{equation}

\noindent
Notice by the bounds on $N$ and $M$ in (\ref{eq:fincond1}), we have $0<n<R$. Rather than dealing with (\ref{eq:finsum1}), we deal with the following more general sum:

\begin{equation}\label{eq:finnsumm} 
F_{M,l,q}(\tilde{M}; \tilde{\beta}_1,\ldots,\tilde{\beta}_{\tilde{S}}):=\frac{1 }{M^l}\sum_{k=0}^{\tilde{M}} (q+k)^l \exp(2\pi i\tilde{\beta}_1 k+\ldots+2\pi i \tilde{\beta}_{\tilde{S}} k^{\tilde{S}})\,,  
\end{equation}

\noindent
where $q$ and $\tilde{M}$ are non-negative integers, the $\tilde{\beta}$'s are real numbers, and

\begin{equation}\label{eq:finncon} 
q+\tilde{M}\le M\,,\qquad \tilde{S}= 3+ \left\lfloor R/n \right\rfloor, \qquad |\tilde{\beta}_s|\le \min\left\{1/2, 2^{R-2} 2^{(3-2s)n}\right\}\,, 
\end{equation}

\noindent
Since $S\le \tilde{S}$, and $N/(4M^{2s-3})\le 2^{R-2}2^{(3-2s)n}$, the sum (\ref{eq:finnsumm}), with conditions (\ref{eq:finncon}), is indeed more general than the sum (\ref{eq:finsum1}), with conditions (\ref{eq:fincond1}).  For example, (\ref{eq:finsum1}) can be written as 

\begin{equation}\label{eq:exstart}
\frac{1}{M^l}\sum_{k=0}^{M} k^l \exp(2\pi i \beta_1 k+\ldots+2\pi i \beta_S k^S) =F_{M,l,0}(M;\beta_1,\ldots,\beta_S,0,\ldots,0)\,,
\end{equation}

\noindent
where we padded $\tilde{S}-S$ zeros at the end. We now carry out a ``dyadic approximation'' of the sum (\ref{eq:exstart}). With this in mind, let $q_0 := 0$, $M_0:=M$, $n_0:=n$, and $\tilde{\beta}^{(0)}_s:=\tilde{\beta}_s$, for $1\le s\le \tilde{S}$.  Then for integers $d\ge 0$, and for as long as $M_{d+1}>  1$, sequentially define

\begin{equation}\label{eq:finncon111}
\begin{split}
M_{d+1}&:=M_d-2^{n_d}\,,\qquad  n_{d+1}:=\lfloor \log_2 M_{d+1} \rfloor\,,\\
q_{d+1}&:=q_d+2^{n_d}\,, \qquad  \tilde{\beta}_s^{(d+1)}  = \sum_{p=s}^{\tilde{S}} \binom{p}{s} 2^{n_d(p-s)} \tilde{\beta}_p^{(d)}\,. 
\end{split}
\end{equation}

\noindent
Notice $M_d+q_d=M_0+q_0\le M$, $M_{d+1}<M_d/2$, $d< n$,  and $n_d\le n-d$. And we have

\begin{equation} \label{eq:setupit}
\begin{split}
 F_{M,l,q_d}(M_d; \tilde{\beta}_1^{(d)},\ldots,\tilde{\beta}_{\tilde{S}}^{(d)}) &= F_{M,l,q_d}(2^{n_d}-1; \tilde{\beta}^{(d)}_1,\ldots,\tilde{\beta}^{(d)}_{\tilde{S}})\\
&+ c_d \,F_{M,l,q_{d+1}}(M_{d+1};\tilde{\beta}_1^{(d+1)},\ldots,\tilde{\beta}_{\tilde{S}}^{(d+1)})\,,
\end{split}
\end{equation}

\noindent
where $c_d$ satisfies $|c_d|=1$, and $M_d$, $q_d$, $c_d$, $n_d$, $\tilde{\beta}^{(d)}$, and $\tilde{\beta}^{(d+1)}$, can all be computed efficiently. By iterating (\ref{eq:setupit}) at most $n$ times, the evaluation of (\ref{eq:exstart}) can be reduced  to numerically evaluating at most $n$ functions of the form 

\begin{equation}\label{eq:fmlqd}
F_{M,l,q_d}(2^{n_d}-1; \tilde{\beta}^{(d)}_1,\ldots,\tilde{\beta}^{(d)}_{\tilde{S}})\,.
\end{equation}

\noindent
Since $n=O(\nu)$, it suffices to show how to deal with each such function. To do so, we will need an upper bound on the size of the coefficients $\beta_s^{(d)}$. By induction on $d$, suppose the inequality

\begin{equation} \label{eq:indcon1}
|\tilde{\beta}_s^{(d)}|\le 2^{R-2} 2^{(3-2s)n}2^{sd} (1+1/(2R))^d \,,
\end{equation}

\noindent
holds for $3\le s\le \tilde{S}$ (notice by the third condition in (\ref{eq:finncon}) that  (\ref{eq:indcon1}) is satisfied for $d=0$ and all $3\le s\le \tilde{S}$).  Then by the recurrence for $\tilde{\beta}^{(d+1)}$ in (\ref{eq:finncon111}), and the estimate $\binom{p+s}{p}\le 2^{p+s}$, we obtain

\begin{equation}\label{eq:betaest}
|\beta_s^{(d+1)}| \le    2^{R-2} 2^{(3-2s)n}2^{s(d+1)} (1+1/(2R))^d \sum_{p=0}^{\tilde{S}-s} 2^p 2^{pn_d} 2^{-2np}2^{dp} \,.
\end{equation}

\noindent
As remarked earlier, $n_d\le n-d$, so if $n\ge \log_2 R+ 4$ say, then

\begin{equation}\label{eq:auxestr2}
\sum_{p=0}^{\tilde{S}-s} 2^p 2^{pn_d} 2^{-2np}2^{dp} \le \sum_{p=0}^{\tilde{S}} 2^{-(n-1)p} \le 1+1/(2R)\,.
\end{equation}

\noindent
Substituting (\ref{eq:auxestr2}) back into (\ref{eq:betaest}) shows estimate (\ref{eq:indcon1}) holds for $\beta_s^{(d+1)}$, as claimed. Moreover, if $n< \log_2 R+4$, then $M\le 32R=O(\nu)$, so should this happen the sum (\ref{eq:finnsumm}) can be evaluated directly anyway. By (\ref{eq:indcon1}), the fact $(1+1/(2R))^d\le 2$, and using similar calculations to those in \textsection{2} (while describing Sch\"onhage's method), one can employ Taylor expansions to reduce the evaluation of  (\ref{eq:fmlqd}) to that of precomputing the sums:

\begin{equation}\label{eq:precomsum}
\frac{1}{2^{(n-d)l}}\sum_{k=0}^{2^{(n-d)}-1} k^l \exp\left(2\pi i\frac{\tilde{\sigma}_{n,1}^{(d)}}{2^{(n-d)}} k+2\pi i \frac{\tilde{\sigma}_{n,2}^{(d)}}{2^{2(n-d)}} k^2+\ldots+2\pi i \frac{\tilde{\sigma}_{n,\tilde{S}}^{(d)}}{2^{\tilde{S} (n-d)}} k^{\tilde{S}}\right)\,, 
\end{equation}

\noindent
for all integers $1\le R <\lceil \log_2(\Lambda^2 K^{\mu}+1)\rceil$, $1\le n<R$, $0\le d<n$, $0\le l=O(\tilde{S}\,\nu)$, and all integers $\tilde{\sigma}_{n,s}^{(d)}$ satisfying

\begin{equation}\label{eq:precomcon}
|\tilde{\sigma}_{n,s}^{(d)}|\le \min\left\{2^{s(n-d)-1},2^{R-1} 2^{(3-s)n}\right\} \,, 
\end{equation}

\noindent
Once the sums (\ref{eq:precomsum}) are precomputed for all such values, the sum (\ref{eq:finsum1}), with conditions (\ref{eq:fincond1}), can be evaluated efficiently using Taylor expansions. Alternatively, one can use band-limited interpolation techniques, which is probably more practical; see~\cite{Od}. 

Now, the sum (\ref{eq:precomsum}) is the discrete Fourier transform, evaluated at $-\tilde{\sigma}_{n,1}^{(d)}$, of the sequence of points 

\begin{equation}
\frac{k^l}{2^{(n-d)l}}\,\exp\left(2\pi i \frac{\tilde{\sigma}_{n,2}^{(d)}}{2^{2(n-d)}} k^2+\ldots+2\pi i \frac{\tilde{\sigma}_{n,\tilde{S}}^{(d)}}{2^{\tilde{S} (n-d)}} k^{\tilde{S}}\right)\,,\quad 0\le k<2^{(n-d)}\,.
\end{equation}

\noindent
So, given $l$, $n$, and $d$, (\ref{eq:precomsum}) can be computed at all the (integer) values of $\tilde{\sigma}^{(d)}$ specified in (\ref{eq:precomcon}) using the FFT in at most 

\begin{equation} \label{eq:fft001}
\Lambda\,2^{3(n-d)}\,\prod_{s=3}^{3+\lfloor R/n\rfloor} \min\left\{2^{s(n-d)}, 2^R 2^{(3-s)n} \right\}
\end{equation}

\noindent
operations on numbers of $O(\nu^2)$ bits, and requiring at most as many bits of storage. Since $\min\left\{2^{s(n-d)-1}, 2^{R-1}2^{(3-s)n} \right\} \le \min\left\{ 2^{sn},2^{R+(3-s)n}\right\}$, then (\ref{eq:fft001}) is bounded by 

\begin{equation} \label{eq:fft0001}
\Lambda \prod_{s=0}^{3+\lfloor R/n\rfloor} \min\left\{ 2^{sn},2^{R+(3-s)n}\right\}\,,
\end{equation}

\noindent
We observe $\min\left\{2^{sn}, 2^{R+(3-s)n}\right\}= 2^{sn}$ for\mbox{ $s<\lceil R/2n+3/2 \rceil$}, from which it follows by a fairly straightforward calculation that (\ref{eq:fft0001}) is bounded by 

\begin{equation}
\Lambda \,2^{n(R/(2n)+3/2)^2}\,.
\end{equation}

The function $n(R/(2n)+3/2)^2$ is of size $\le 4R$ exactly when $R/9\le  n\le R$. So, given $n$, $l$, and $d$ such that  $R/9\le  n\le R$, the cost of the FFT precomputation is $\le \Lambda\,2^{4R}$ operations on numbers of $O(\nu^2)$ bits. Since $0\le d< n-1$ and $0\le l=O(\tilde{S}\nu)$, and since by conditions (\ref{eq:fincond1}) and definitions (\ref{eq:rndef}) we have $2^R\le 2\Lambda^2 K^{\mu}$, then the total cost of the precomputation (for all possible values of $l$, $d$, and $R/9\le n<R$) is at most $16\ \Lambda^9 K^{4\mu}$ operations on numbers of $O(\nu^2)$ bits. Notice for $R/9\le n\le R$, we have $\tilde{S}\le 12$. So the exponential sum (\ref{eq:precomsum}) will have have degree $\le 12$ over that range of $n$. 

It remains to consider the case $n<R/9$. This implies $M \le \nu N^{1/9}$. So $M$ is small compared to $N$, and the convergence of the Taylor series (\ref{eq:ssum}) from \textsection{2} is slower, which leads to higher degree exponential sums (\ref{eq:finsum1}). 

By the definitions (\ref{eq:rndef}) of $n$ and $R$, and the condition $2^R \le 2\Lambda^2 K^{\mu}$, we have if $n< R/9$, the length of the sum (\ref{eq:finsum1}) is at most $2^{n+1}\le 2^{R/9+1}\le  4\Lambda K^{\mu/9}$ terms. Since this is a relatively short length, one option is to directly evaluate (\ref{eq:finsum1}) in such cases. If we do so, however,  a simple optimization procedure reveals the resulting algorithm to compute $\zeta(1/2+it)$ has complexity $t^{37/117+o(1)}$ only (notice $37/117\approx 0.316\ldots$). We would like to avoid direct computation of these sums in order to achieve the $t^{4/13+o(1)}$ complexity. 

To this end, observe the sum (\ref{eq:finsum1}) can be viewed in the following alternative light.  Recall by Taylor series (\ref{eq:ssum}) we have  $\beta_s=d_s$ for $s \ge 3$, where $d_s$ is defined as in (\ref{eq:altform}). So the coefficients $\beta_{s+3}$, for $s\ge 0$, can be expressed in the form

\begin{equation}\label{eq:relations222}
\beta_{s+3}= \tau \eta_s\,,\qquad \eta_s:=\eta_{s,\rho,\gamma}= \rho^s \sum_{l=0}^{\infty} z_{l,s} \gamma^l\,,
\end{equation}

\noindent
where the numbers $z_{l,s}$ are quickly computable, depend only on $l$ and $s$, satisfy $|z_{l,s}| \le 1/2$, and 

\begin{equation} \label{eq:relations111} 
\tau :=\tau_{c,\alpha}=  2 (48)^3 c \alpha^3\,, \quad \rho:=\rho_{c,\alpha,\alpha_2}= 2 c \alpha \alpha_2\,, \quad \gamma:=\gamma_{c,\alpha_1,\alpha_2}=48 c\alpha_1 \alpha_2\,.
\end{equation}

\noindent
(Notice $\tau$, $\rho$, and $\gamma$ are real numbers.) Therefore, by the bounds on $c$, $\alpha$, $\alpha_1$, and $\alpha_2$, specified in (\ref{eq:assump2}), as well as the bound $|z_{l,s}|\le 1/2$, we have

\begin{equation} \label{eq:fincond2}
|\tau| \le \frac{N}{M^3}\,, \qquad |\rho| \le  \frac{1}{M^2}\,, \qquad |\gamma|\le \frac{1}{M^2}\,. 
\end{equation}

\noindent
The infinite series defining $\eta_s$ in (\ref{eq:relations222}) converges rapidly, and only $O(\nu)$ terms are needed to ensure its calculation to within $O(\epsilon/K^2)$ say. Since each of $z_{l,s}$, $\rho$, and $\gamma$ can be computed quickly, so can $\eta_s$.  So the sum (\ref{eq:finsum1}) may now be formulated as a function $W(M,l,S; \beta_1,\beta_2,\beta_3, \tau,\rho,\gamma)$ given by 

\begin{equation}\label{eq:finsum2}
\begin{split}
W(M,l,S; \beta_1,\beta_2,\beta_3, \tau,\rho,\gamma):= \frac{1}{M^l}\sum_{k=0}^{M} k^l \exp(& 2\pi i \beta_1 k+2\pi i \beta_2 k^2+ 2\pi i \beta_3 k^3\\
&+2\pi i \tau\eta_1 k^4+\ldots + 2\pi i \tau \eta_{S-3} k^S ) \,,
\end{split} 
\end{equation}

\noindent
In particular, once $\beta_1$, $\beta_2$, $\beta_3$, $\tau$, $\rho$, and $\gamma$, are determined, so is the value of the original sum (\ref{eq:finsum1}). 

Presenting the sum (\ref{eq:finsum1}) in the form $W(M,l,S; \beta_1,\beta_2,\beta_3, \tau,\rho,\gamma)$  is useful because it considers that the coefficients $\beta_4, \beta_5,\ldots,\beta_S$ in (\ref{eq:finsum1}) are not independent of each other. Thus, the grid points where it is necessary to precompute the sum (\ref{eq:finsum1}) is much sparser than is required in formulation (\ref{eq:precomsum}). This becomes especially important when $M$ is small in comparison to $N$ because this is when $S=3+\lfloor \log N/\log M\rfloor$ is of noticeable size, and so simplifying matters by treating the variables $\beta_4,\ldots,\beta_S$ independently, like we did to arrive at (\ref{eq:precomsum}), becomes quite costly as there are many variables $\beta$. 

Given $\beta_1,\beta_2,\beta_3\in[0,1)$, and $\tau$, $\rho$, and $\gamma$, conforming to conditions (\ref{eq:fincond2}), we obtain by Cauchy's estimate  
applied with circles $\tilde{C}_1$, $\tilde{C}_2$, $\tilde{C}_3$, $\tilde{C}_4$, $\tilde{C}_5$, and $\tilde{C}_6$, going about the origin once with radii $1/(8\pi M)$, $1/(8\pi M^2)$, $1/(8\pi M^3)$, $1/(8\pi M^2)$, $1/(8\pi NM)$, and $M/(8\pi N)$, respectively, that

\begin{equation}\label{eq:wder1}
\begin{split}
\frac{1}{r_1!\,\ldots\,r_6!}\,&\left|\frac{\partial^{r_1}\ldots \partial^{r_6}}{\partial z_1^{r_1}\ldots\partial z_6^{r_6}}\, W(M,l,S;\beta_1+z_1,\ldots,\gamma+z_6)\right|_{z_1=0\,,\,\ldots\,,\,z_6=0} \le\\
&N\,M\,(8\pi M)^{r_1}\, (8\pi M^2)^{r_2}\, (8\pi M^3)^{r_3}\,(8\pi M^2)^{r_4}\, (8\pi MN)^{r_5}\, (8\pi N/M)^{r_6}\,.
\end{split}
\end{equation} 

\noindent
Also, by Taylor expansions we have 

\begin{equation}\label{eq:wexep1}
\begin{split}
W(M,& l,S;\beta_1+z_1,\,\ldots\,,\gamma+z_6)= \sum_{r_1=0}^{\infty}\,\cdots\,\sum_{r_6=0}^{\infty}  \frac{z_1^{r_1}\,\ldots \, z_6^{r_6}}{r_1!\,\ldots\,r_6!}\,\times \\
&\left[\frac{\partial^{r_1}\ldots \partial^{r_6}}{\partial z_1^{r_1}\ldots\partial z_6^{r_6}}\, W(M,l,S;\beta_1+z_1,\ldots,\gamma+z_6)\right]_{z_1=0\,,\,\ldots\,,\,z_6=0}\,. 
\end{split}
\end{equation}

\noindent
Therefore, if we ensure $|z_1|\le 1/(16\pi M)$, $|z_2|\le 1/(16\pi M^2)$, $|z_3|\le 1/(16\pi M^3)$, $|z_4|\le 1/(16\pi M^2)$, $|z_5|\le 1/(16\pi MN)$, and $|z_6|\le M/(16\pi N)$, then by bound (\ref{eq:wder1}) each of the series over $r_1,\ldots,r_6$ in (\ref{eq:wexep1}) can be truncated after $O(\nu)$ terms, which results in a truncation error of size $O(\epsilon/K^2)$ say. So for $z_1,\ldots,z_6$ of such sizes, the value of the perturbed function $W(M,l,S;\beta_1+z_1,\ldots,\gamma+z_6)$ can be recovered efficiently, using expansion (\ref{eq:wexep1}), from the values 

\begin{equation}\label{eq:wpart}
\left[\frac{\partial^{r_1}\,\ldots\,\partial^{r_6}}{\partial z_1^{r_1}\,\ldots\,z_6^{r_6}}\,W(M,l,S;\beta_1+z_1,\,\ldots\,,\gamma+z_6)\right]_{z_1=0\,,\,\ldots\,,\,z_6=0}
\end{equation}

\noindent
for $0\le r_1\,,\, \ldots\,,\, r_6=O(\nu)$, and $0\le l=O(S\nu)$, assuming each such value is known to within $\pm\,\epsilon/K^2$ say. 

We discretize the interval $[0,1)$, which is where $\beta_1$ resides, in step sizes of $1/(16\pi M)$. This ensures any $\beta_1\in[0,1)$ can be expressed as $\beta_1=p_1/(16\pi M)+z_1$, where $0\le p_1 \le 16\pi M$ an integer, and $|z_1|\le 1/(32\pi M)$. This suffices to kill the growth in the derivatives of $W$ with respect to $z_1$, as can be seen from bound (\ref{eq:wder1}).  

As for $\beta_2$ and $\beta_3$, which both also reside in $[0,1)$, we use step sizes of $1/(16\pi M^2)$ and $1/(16\pi M^3)$ respectively. This ensures they can be written as $\beta_2=p_2/(16\pi M^2)+z_2$ and $\beta_3=p_3/(16\pi M^3)+z_3$, where $|z_2|\le 1/(32\pi M^2)$ and $|z_3|\le 1/(32\pi M^3)$. This again suffices to kill the growth in the derivatives of $W$ with respect to $z_2$ and $z_3$. 

Similarly, we discretize the interval $[-N/M^3, N/M^3]$, which is where $\tau$ resides, in steps of $1/(16\pi M^2)$. This gives $\le 32\pi N/M$ relevant discretizations for $\tau$. As for $\rho$, which resides in $[-1/M^2,1/M^2]$, we use a step size of  $1/(16\pi NM)$, which yields $32\pi N/M$ relevant discretizations. Last, in the case of $\gamma$, which is restricted to $[-1/M^2,1/M^2]$, we use a step size of $M/(16\pi N)$, yielding $32\pi N/M^3$ relevant discretizations for it.

It is clear once we obtain the values of the partial derivatives of $W$ in (\ref{eq:wpart}) at the discretized values of $\beta_1$, $\beta_2$, $\beta_3$, $\tau$, $\rho$, and $\gamma$, and for integers $0\le l=O(S\nu)$, and integers $0\le r_1\,,\,\ldots\,,\,r_6=O(\nu)$, then the values of $W$ elsewhere in the region $\beta_1,\beta_2,\beta_3\in[0,1)$, and $\tau$, $\rho$, and $\gamma$, conforming to conditions (\ref{eq:fincond2}), and with $0\le l=O(\nu)$, can be recovered quickly, in $O(\nu^6)$ steps, via (the truncated version of) expansion (\ref{eq:wexep1}). 

Suppose the values of $W(M,l,S;\beta_1,\beta_2,\beta_3,\tau,\rho,\gamma)$ have been precomputed at all the discretized values of $\beta_1$, $\beta_2$, $\beta_3$, $\tau$, $\rho$, and $\gamma$, and for integers $0\le l=O(S\nu)$, to within $\pm\,\epsilon/K^2$ each say. Since the partial derivatives of $W$ in (\ref{eq:wder1}), evaluated at any of these discretizations, are needed up to $r_1=O(\nu)\,,\ldots\,,\, r_6=O(\nu)$  only, they can be calculated via (quite laborious) recursions using $O(\Lambda^3)$ operations on numbers of $O(\nu^2)$ bits (see the proof of lemma~\ref{lem:la3} for an example of such a recursion). 

To conclude, for each $l$, $S$, and $M$, we employ the FFT to precompute the values of $W(M,l,S;\beta_1,\beta_2,\beta_3,\tau,\rho,\gamma)$ at all the discretizations of $\beta_1$, $\beta_2$, $\beta_3$, $\tau$, $\rho$, and $\gamma$, to within $\pm\,\epsilon/K^2$ each say. By the discussion following (\ref{eq:wpart}), there are $\le (16\pi M)^6\,(32\pi N/M)^2\,(32\pi N/M^3)\le(32\pi)^9 N^3 M$ discretizations to consider. So the cost of the FFT precomputation is bounded by $O(\nu^2 M N^3)$ operations on numbers of $O(\nu^2)$ bits. Finally, since $l=O(S\nu)$, $S=O(\nu)$, and, by hypothesis, $M\le \nu N^{1/9}$, there are only $O(\nu^4 N^{1/9})$ permissible tuples $(l,S,M)$. Therefore, the total cost of the FFT precomputation is bounded by $O(\Lambda N^{29/9})$ operations on numbers of $O(\nu^2)$ bits. Since $N\le \Lambda K^{\mu}$, this cost is certainly bounded by $O(\Lambda^5 K^{4\mu})$ operations on numbers of $O(\nu^2)$, which completes our proof of Theorem~\ref{cubicthm}.

\section{Auxiliary results}

\noindent
\textit{Remark.} The conventions stated at the start of \textsection{4} regarding the presentation of certain frequently occurring details apply here as well.

\begin{lem} \label{lem:la1}
Let $B_2=1/6$, $B_4=-1/30$,$\ldots\,$, denote the even Bernoulli numbers. For any $j\ge 0$, any integer $K>\Lambda(K,j,\epsilon)$, any real $c$,  and with $f_{K,j,x,c}(y):=\frac{y^j}{K^{j}}e^{2\pi i c y^3+2\pi i x y}$, we have

\begin{equation}
\begin{split}
\sum_{n=0}^{K} f_{K,j,x,c}(n)&= \int_0^{K} f_{K,j,x,c}(y)\, dy +\frac{1}{2}\left(f_{K,j,x,c}(K)+f_{K,j,x,c}(0)\right) \\
&+ \sum_{m=1}^M \frac{B_{2m}}{(2m)!}\left(f_{K,j,x,c}^{(2m-1)}(K)-f_{K,j,x,c}^{(2m-1)}(0)\right)+E_{M,K,j,x,c}\,. 
\end{split}
\end{equation}

\noindent
where $f_{K,j,x,c}^{(m)}(y)$ denotes the $m^{th}$ derivative of $f_{K,j,x,c}(y)$ with respect to $y$, and 

\begin{equation}\label{eq:emest1}
|E_{M,K,j,x,c}|\le \displaystyle \frac{10}{(2\pi )^{2M}}\displaystyle \int_0^{K} \left| f_{K,j,x,c}^{(2M)}(y)\right| \, dy\,.
\end{equation}
\end{lem}

\begin{proof}
This is a direct application of the well-known Euler-Maclaurin summation formula, and the estimate $|B_{2m}| \le 10\, (2m)!/(2\pi)^{2m}$ for $m\ge 1$, say; see~\cite{Ru} for instance. 
\end{proof}

\begin{lem} \label{lem:la3}
For any $\epsilon \in (0,e^{-1})$ , any integer $j\ge 0$, any integer $K>\Lambda(K,j,\epsilon)$, any $x\in [-3/4,3/4]$ say, any real $c$ satisfying $|cK^2| \le 1/48$ say, and with $f_{K,j,x,c}(y)$ and $E_{M,K,j,x,c}$ defined as in lemma~\ref{lem:la1}, we have 

\begin{equation}
\max_{\substack{|y| \le K\\|x|\le 1}} \left|f_{K,j,x,c}^{(m)}(y)\right|\le \left(6\pi c K^2+3\pi/2+\frac{2m+j}{K}\right)^m\,,
\end{equation}

\noindent
where $f_{K,j,x,c}^{(m)}(y)$ denotes the $m^{th}$ derivative of $f_{K,j,x,c}(y)$ with respect to $y$. If $M=\lceil 8\log (K/\epsilon) \rceil$, the remainder $E_{M,K,j,x,c}$ from lemma~\ref{lem:la1} satisfies $|E_{M,K,j,x,c}|<\epsilon$. Also, the value of the derivative $f_{K,j,x,c}^{(m)}(y)$ for any $y$, $x$, $K$, $j$, and $c$, falling within the ranges specified in the lemma, can be computed to within $\pm\,K^{-2} \epsilon$, say, using $O((m+j+1)^2)$ operations on numbers of $O(\nu(K,j,\epsilon)^2)$ bits.  
\end{lem}

\begin{proof}
It is not hard to see $f_{K,j,x,c}^{(m)}(y)=P_{m,K,j,x,c} (y) e^{2\pi i c y^3+2\pi i x y}$, $P_{m,K,j,x,c}(y)$ is a polynomial in $y$ of degree $2m+j$ (it is also a polynomial in $x$ of degree $m$). Notice $|f_{K,j,x,c}^{(m)}(y)|=|P_{m,K,j,x,c}(y)|$, and the polynomials $P_{m,K,j,x,c}(y)$ are determined by the following recursion on $m$: 

\begin{equation}\label{eq:pkxxrec}
P_{m+1,K,j,x,c}(y)= 2\pi i (x+3cy^2) P_{m,K,j,x,c}(y)+ \frac{d}{dy} P_{m,K,j,x,c}(y)\,,
\end{equation}

\noindent
where $P_{0,K,j,x,c}(y):=y^j/K^j$.  So  $P_{K,j,x,c,m}(y)=\sum_{l=0}^{2m+j} d_{l,m,K,j,c}(x) y^l$, where the coefficients $d_{l,m,K,j,c}(x)=:\sum_{r=0}^m z_{r,l,m,K,j,c}\, x^r$. It is convenient to define the norm $|P_{m,K,j,x,c}(y)|_1:= \sum_{l=0}^{2m+j} \sum_{r=0}^m |z_{r,l,m,K,j,c}\,x^r y^l|$.  Notice $|P_{m,K,j,x,c}(y)|\le |P_{m,K,j,x,c}(y)|_1$. By induction on $m$, suppose 

\begin{equation}\label{eq:pkxx1}
\max_{\substack{|y|\le K\\ |x|\le 1}} \left|P_{K,j,x,c,m}(y)\right|_1 \le \left(6\pi c K^2+3\pi/2 +(2m+j)/K\right)^m\,. 
\end{equation}

\noindent
One easily deduces from (\ref{eq:pkxx1}) that

\begin{equation}\label{eq:pkxx2}
\max_{\substack{|y|\le K\\ |x|\le 1}} \left|\frac{d}{dy}\,P_{m,K,j,x,c}(y)\right|_1 \le \frac{2m+j}{K}\,\max_{\substack{|y|\le K\\ |x|\le 1}} \left|P_{m,K,j,x,c}(y)\right|_1\,.
\end{equation}

\noindent
On combining (\ref{eq:pkxxrec}), (\ref{eq:pkxx1}), and (\ref{eq:pkxx2}), the first part of the lemma follows. The second part of the lemma follows by using the recursion (\ref{eq:pkxxrec}).


\end{proof} 

\begin{lem} \label{lem:la4}
Let $B_2=1/6$, $B_4=-1/30$,$\ldots\,$, denote the even Bernoulli numbers. There are absolute constants $\kappa_{11}$, $\kappa_{12}$, $A_{15}$, $A_{16}$, and $A_{17}$, such that for any \mbox{$\epsilon \in (0,e^{-1})$}, any integer $j\ge 0$,  any positive integers $K$, $K_1$ satisfying $\Lambda (K,j,\epsilon)\le K_1 \le K$, any $a \in [0,1)$, any $b\in [0,1)$, any real $c$ satisfying $|c K_1^2|<1/48$, any $1\le m \le 100 \nu(K,j,\epsilon)$ say,  any  $\alpha\in [-1,1]$, any interval $[w,z] \subset [-1,1]$, and with $f_{K,j,x,c}(y):=\frac{y^j}{K^{j}}e^{2\pi i c y^3+2\pi i x y}$, the sum

\begin{equation} \label{eq:done1}
\frac{B_{2m}}{(2m)!}\int_{w}^{z} \left(f_{K,j,x,c_1}^{(2m-1)}(K_1)-f_{K,j,x,c_1}^{(2m-1)}(0)\right)  F(K;a+\alpha x,b)\, dx\,, 
\end{equation}

\noindent
where $f_{K,j,x,c}^{(m)}(y)$ denotes the $m^{th}$ derivative of $f_{K,j,x,c}(y)$ with respect to $y$, can be computed to within $\pm \,A_{15}\,\nu(K,j,\epsilon)^{\kappa_{11}} K^{-2} \epsilon$ using $\le A_{16}\,\nu(K,j,\epsilon)^{\kappa_{12}}$ arithmetic operations on numbers of $\le A_{17}\,\nu(K,j,\epsilon)^2$ bits. 
\end{lem}

\begin{proof}
Write $f_{K,j,x,c}^{(m)}(y)=P_{K,j,y,c,m} (x)e^{2\pi i c y^3+2\pi i x y}$, where, as can be seen from the proof of lemma~\ref{lem:la3}, $P_{K,j,y,c,m}(x)=\sum_{l=0}^m v_{l,K,j,c,m}(y) x^l$, and $v_{l,K,j,c,m}(y)$ are polynomials in $y$ of degree $\le 2m+j$.  Now the bound $|v_{l,K,j,c,m}|\le (2\pi)^m$, afforded by lemma~\ref{lem:la3}, together with $|B_{2m}/(2m)!|\le 10/(2\pi)^{2m}$, yields the bound $(B_{2m}/(2m)!) |v_{l,K,j,c,2m-1}(y)|_1\le 10$. So in order to be able to compute (\ref{eq:done1}) with the claimed accuracy, it is enough to be able to deal with the integrals 

\begin{equation} \label{eq:done11}
\int_w^z  x^l e^{2\pi i K_1 x} F(K; a+\alpha x,b) \, dx\,,\qquad  \int_{w}^{z}  x^l F(K; a+\alpha x,b)\, dx\,,
\end{equation}

\noindent
where $0\le l\le m$.  More generally, we show how to deal with integrals of the form 

\begin{equation}\label{eq:done11gen}
\int_w^z x^l e^{2\pi i \tau x} F(K;a+\alpha x,b)\,,
\end{equation}

\noindent
where $-K\le \tau \le K$ say. We split the quadratic sum $F(K; a+\alpha x,b)$ into two subsums, one over $\{0\le k\le K\,:\, |\tau+\alpha k|\le 2l+2\}$, and another over $\{0\le k\le K\,:\, |\tau+\alpha k|> 2l+2\}$. For the first subsum, we use a change of variable to expand the interval of integration, then we divide the expanded interval into a few consecutive subintervals, and over each subinterval we show the integral can be computed efficiently. Specifically, we start by applying the change of variable $x\leftarrow (2l+2)x$ to the integral (\ref{eq:done11gen}). Then we divide the now expanded interval of integration into $\le 4l+4=O(\nu)$ consecutive subintervals $[n,n+1)$, where $0\le n<4l+4$. This leads to $O(\nu)$ integrals of the form

\begin{equation}\label{eq:stanarg1}
\sum_{\substack{0\le k\le K\\ |\tau+\alpha k|\le 2l+2}} e^{2\pi i a k +2\pi i b k^2} \frac{1}{(2l+2)^{l+1}}\int_n^{n+1} x^l e^{2\pi i (\tau+\alpha k) x/(2l+2)}\,dx,
\end{equation} 

\noindent
For each integral (\ref{eq:stanarg1}), we apply the change of variable $x\leftarrow x-n$. Then we use Taylor expansions to reduce the integrand to a polynomial in $x$ of degree bounded by $O(\nu)$, plus an error of size $O(\epsilon/K^2)$ say. Explicitly, we obtain

\begin{equation}
\begin{split}
\frac{1}{(2l+2)^{l+1}}&\int_n^{n+1} x^l e^{2\pi i (\tau+\alpha k) x/(2l+2)}\,dx=\\
&\frac{e^{2\pi i (\tau +\alpha k)n}}{(2l+2)^{l+1}}\sum_{r=0}^l \binom{l}{r} n^{l-r} \sum_{s=0}^{\lceil 2\nu\rceil} \frac{(2\pi i )^s}{s!}\,\frac{(\tau+\alpha k)^s}{(2l+2)^s}\int_0^1 x^{r+s} + O(\epsilon/K^2)\,.
\end{split}
\end{equation} 

\noindent
The coefficients in said polynomial are quickly computable and are each bounded by $O(1)$. We then integrate (the polynomial in $x$) explicitly. On substituting back into (\ref{eq:stanarg1}), we obtain a linear combination, with quickly computable coefficients each of size $O(1)$, of quadratic exponential sums. And these sums are handled by Theorem~\ref{quadraticthm} of~\cite{Hi}. 

We remark it is not desirable to immediately apply a binomial expansion to the powers $(\tau +\alpha k)^s$ resulting from the above procedure  because the terms of such an expansion might have significant cancellations among them, depending on the signs and sizes of $\alpha$ and $\tau$. Instead, one can first change the index of summation by $k\leftarrow k+\lfloor\tau/\alpha\rfloor$ (if $|\alpha| > 1/K^2$ say), then apply a binomial expansion. This way, the amount of cancellation is minimal, which is useful in practice (in theory this does not matter because $|\tau|\le K$, $s=O(\nu)$, and we are using $O(\nu^2)$ bit arithmetic, so the amount cancellation is manageable either way). 

For the second subsum, we integrate explicitly with respect to $x$. Specifically, 

\begin{equation}\label{eq:frange2}
\int_w^z x^l e^{2\pi i (\tau+\alpha k)x}\,dx = \sum_{v=0}^l \frac{(-1)^v\,l!}{(l-v)!}\,\frac{z^{l-v} e^{2\pi i (\tau+\alpha k)z}-w^{l-v}e^{2\pi i (\tau +\alpha k)w}}{(2\pi i \tau+2\pi i \alpha k)^{v+1}}\,.
\end{equation}

\noindent
Substituting (\ref{eq:frange2}) back into the second subsum produces a linear combination, with quickly computable coefficients, of $2l+2$ exponential sums; namely,

\begin{equation}\label{eq:ndsubsum}
\sum_{v=0}^l \frac{(-1)^v\,l!\, z^{l-v}e^{2\pi i \tau z}}{(l-v)!\, (2\pi i )^{v+1}} \sum_{\substack{0\le k\le K\\2l+2<|\tau+\alpha k|}} \frac{e^{2\pi i (a+\alpha z) k+2\pi i b k^2}}{(\tau+\alpha k)^{v+1}}\,,
\end{equation}

\noindent
as well as another identical sum but with $z$ replaced by $w$. We may assume $|\alpha|\ge  1/K^2$ say, otherwise we can apply a Taylor expansion to the term $e^{2\pi i \alpha k x}$ in (\ref{eq:done11gen}) from the beginning, which immediately reduces it to a linear combination, with quickly computable coefficients each of size $O(1)$,  of $O(\nu(K,j,\epsilon))$ quadratic exponential sums, and such sums are handled by Theorem~\ref{quadraticthm}. To deal with subsum in (\ref{eq:ndsubsum}) with $\tau+\alpha k>2l+2$, for example, we define $k_0=\lfloor (2l+2 -\tau)/\alpha\rfloor$, and so

\begin{equation}
\sum_{\substack{0\le k\le K\\2l+2<\tau+\alpha k}} \frac{e^{2\pi i (a+\alpha z) k+2\pi i b k^2}}{(\tau+\alpha k)^{v+1}}= e^{2\pi i (a+\alpha z)k_0+2\pi i b k_0^2} \sum_{k=1}^{K-k_0} \frac{e^{2\pi i (a+\alpha z+2b k_0) k+2\pi i b k^2}}{(\tau+\alpha k_0+\alpha k)^{v+1}}\,.
\end{equation}
 
\noindent
Observing $\tau+\alpha k_0 \ge 2l+2$, and $l!/(l-v)!\le l^v$, we see 

\begin{equation}
\left|\frac{(-1)^v\,l!\, z^{l-v}e^{2\pi i \tau z}}{(l-v)!\, (2\pi i )^{v+1}\,(\tau+\alpha k_0+\alpha k)^{v+1}}\right|<1\,,
\end{equation}

\noindent
for all $1\le k\le K-k_0$. In particular, the subsum in (\ref{eq:ndsubsum}) with $\tau+\alpha k>2l+2$ is of the type discussed in \textsection{5} of~\cite{Hi}, which is handled by Theorem~\ref{quadraticthm} since, as shown in~\cite{Hi}, such sums can be reduced to a linear combination, with quickly computable coefficients each of size $O(1)$, of $O(\nu^2)$ quadratic sums. Last, the treatment of the subsum in (\ref{eq:ndsubsum}) with $\tau+\alpha k<-2l-2$ is identical. 
\end{proof}

\begin{lem} \label{lem:la5}
There are absolute constants $\kappa_{13}$, $\kappa_{14}$, $A_{18}$, $A_{19}$, and $A_{20}$, such that for any \mbox{$\epsilon \in (0,e^{-1})$}, any integer $j \ge 0$,  any positive integers $K$, $K_1$ satisfying $\Lambda(K,j,\epsilon)\le K_1\le K$, any $a\in [0,1)$, any $b \in [0,1)$, any real $c$ satisfying $|cK_1^2|<1/48$,  any  $\alpha\in [-1, 1]$, and any interval $(w,z)\subset (-1,1)$ such that $|w|\ge 1/4$, the function 

\begin{equation} \label{eq:done2}
\frac{1}{(K_1)^j} \int_{w}^{z} \int_0^{K_1} y^j e^{2\pi i c y^3-2\pi i x y} F(K;a+\alpha x,b)\, dy\, dx\,, 
\end{equation}

\noindent
can be computed to within $\pm\,A_{18}\,\nu(K,j,\epsilon)^{\kappa_{13}} K^{-2} \epsilon$ using $\le A_{19}\,\nu(K,j,\epsilon)^{\kappa_{14}}$ arithmetic operations on numbers of $\le A_{20}\,\nu(K,j,\epsilon)^2$ bits.
\end{lem}

\begin{proof}
We assume $w>0$, since if $w<0$ the treatment is completely analogous.  Define the contours $C_1:= \{te^{-i\pi/6}\,|\,0\le t \le 2K_1/\sqrt{3}\}$, $C_2:= \{K_1-it\,|\,0\le t \le K_1/\sqrt{3}\}$, and $C_0:= \{t\,|\,0\le t \le K_1\}$. Also define

\begin{equation}
I_x(C) :=I_{x,K_1,j,c}(C)= \frac{1}{(K_1)^j}\int_C y^j e^{2\pi i c y^3-2\pi i x y}\, dy\,. 
\end{equation}

\noindent
With this notation, the integral (\ref{eq:done2}) can be expressed as 

\begin{equation}\label{eq:ixc00}
\int_w^z I_x(C_0)F(K;a+\alpha x, b)\, dx\,. 
\end{equation}

\noindent
By Cauchy's Theorem $I_x(C_0)=I_x(C_1)-I_x(C_2)$. And by a routine calculation,

\begin{equation}\label{eq:ixc12}
\begin{split}
&I_x(C_1) = \frac{d_{1,K_1,j,c}}{(K_1)^j} \int_0^{\frac{2K_1}{\sqrt{3}}} y^j e^{2\pi c y^3-\sqrt{3} \pi i x y-\pi  x y}\, dy\,,  \\
&I_x(C_2) =   \frac{d_{2,K_1,j,c} e^{-2\pi i K_1 x}}{(K_1)^j} \int_0^{\frac{K_1}{\sqrt{3}}} (K_1-iy)^j e^{6\pi  c K_1^2 y -6\pi i c K_1 y^2-2\pi  c y^3-2\pi x y}\, dy\,, 
\end{split}
\end{equation}

\noindent
where $d_{1,K_1,j,c}$ and $d_{2,K_1,j,c}$ each has modulus 1, and both can be computed quickly. 

Since $1/4\le w\le x$ and $c K_1^2 \le 1/48$, the absolute values of the integrands in $I_x(C_1)$ and $I_x(C_2)$ decline faster than $e^{-y/6}$ throughout $y\in[0,\sqrt{2}K_1]$. So, in both cases we can truncate the interval of integration with respect to $y$ at $L:=L(K,j,\epsilon)=\lceil 6\nu(K,j,\epsilon)\rceil$ say. 

Once truncated, the interval of integration (in both cases) is divided into $L$ consecutive intervals $[n,n+1)$. As explained in detail following (\ref{eq:done11gen}) earlier, over each subinterval $[n,n+1)$, we apply the change of variable $y\leftarrow y-n$. Then we employ Taylor expansions to reduce the integrand to a polynomial in $y$, with coefficients depending on $x$, of degree $O(\nu(K,j,\epsilon))$, plus an error of size $O(\epsilon/K^2)$ say. On integrating said polynomial directly with respect to $y$, we see that (\ref{eq:ixc00}), hence (\ref{eq:done2}), is equal to a linear combination of $O(\nu(K,j,\epsilon)^2)$ integrals of the form 

\begin{equation}\label{eq:ixc212}
\begin{split}
\int_w^z x^l e^{-\sqrt{3} \pi i n x -\pi  n x}  F(K;a+\alpha x, b) \, dx\,,\quad \int_w^z  x^l e^{-2\pi i K_1 x  -2\pi  n x } F(K;a+\alpha x, b) \, dx \,,
\end{split}
\end{equation}

\noindent
plus an error of size $O(\epsilon/K^2)$ say, where $0\le n<L$, and $0\le l=O(\nu(K,j,\epsilon))$, and where the coefficients of the linear combination can be computed quickly and are each of size $O(1)$. Finally, these integrals are treated similarly to (\ref{eq:done11gen}) earlier.
\end{proof}

\begin{lem} \label{lem:la6}
There are absolute constants $\kappa_{15},\kappa_{16}$, $A_{21}$, $A_{22}$, and $A_{23}$, such that for any $\epsilon \in (0,e^{-1})$, any integer $j\ge 0$,  any positive integers $K$ and $K_1$ satisfying $\Lambda(K,j,\epsilon)< K_1< K$,  any real $\beta$ satisfying $|\beta|\le 100 K_1^3$ say, any real $\alpha$ satisfying $|\alpha| \le 100 K_1^3$ say, any real $\eta$ satisfying  $|\eta|\le 100 K_1^3$ say, any real $w\in [0,1]$ say, any $\theta \in \{-1,1\}$, and any real $c$ satisfying $|c K_1^2|<1/\Lambda(K,j,\epsilon)$, the integral

\begin{equation} \label{eq:airy00}
\frac{1}{(K_1)^j} \int_{0}^{K_1} y^j e^{2\pi i  c y^3+2\pi i \eta y} \int_{0}^{w} e^{2\pi  i (\alpha-\theta y) x- 2\pi i \beta x^2} \, dx\, dy\,,
\end{equation}

\noindent
can be computed to within $\pm\,A_{21}\, \nu(K,j,\epsilon)^{\kappa_{15}} K^{-2} \epsilon$ using $\le A_{22}\,\nu(K,j,\epsilon)^{\kappa_{16}}$ arithmetic operations on numbers of $\le A_{23}\,\nu(K,j,\epsilon)^2$ bits.
\end{lem}

\begin{proof}
Conjugating if necessary, we may also assume $\beta\ge 0$. Let us first deal with the case $\beta \le L:=L(K,j,\epsilon)= \lceil \nu(K,j,\epsilon)\rceil$. We make the change of variable $x\leftarrow Lx$ in (\ref{eq:airy00}). Then we divide the resulting interval of integration into $\lfloor L\rfloor$ consecutive subintervals $[n,n+1)$, where $0\le  n<L$ an integer, as well as a final subinterval over $[\lfloor wL \rfloor, wL)$. It suffices to show how to deal with the integral over each such subinterval since there are $\le L+1=O(\nu)$ of them. 

Following a similar procedure to that following (\ref{eq:done11gen}) earlier, over the subinterval $[n,n+1)$, we employ Taylor expansions (preceded, as usual, by the change of variable $y\leftarrow y-n$) to reduce the term $e^{-2\pi i \beta x^2/L^2}$ in the integrand to a polynomial in $x$ of degree $O(\nu(K,j,\epsilon))$, plus an error of size $O(\epsilon/K^2)$. Notice in doing so, we appeal to the bound $\beta\le L$. At this point, we reach a linear combination of $O(\nu(K,j,\epsilon))$ integrals of the form

\begin{equation} \label{eq:airy1}
\frac{1}{(K_1)^j} \int_{0}^{K_1} y^j  e^{2\pi i  c y^3+2\pi i \eta_1 y}\int_{0}^1  x^se^{2\pi  i \frac{\alpha_1-\theta y}{ L} x} \, dx\, dy\,, 
\end{equation}

\noindent
where the coefficients of the linear combination can be computed quickly, are of size $O(1)$ each, and where  $0\le s=O(\nu(K,j,\epsilon))$ an integer, $\eta_1:=\eta_{1,\beta,L,n}$ a real number satisfying $\eta_1=O(K_1^3)$, and $\alpha_1:=\alpha_{1,\beta,L,n}$ a real number satisfying $\alpha_1 =O(K_1^3)$. We divide the interval of integration with respect to $y$ in (\ref{eq:airy1}) into two sets: 

\begin{equation}
\begin{split}
&I_1:=\{y\in [0,K_1]\, :\, |\alpha_1-\theta y|\ge 2(s+1)L\}\,,\\
&I_2:=\{y\in [0,K_1]\,:\, |\alpha_1-\theta y|<2(s+1)L \}\,. 
\end{split}
\end{equation}

\noindent
So $I_1\cup I_2=[0,K_1]$. Notice each of $I_1$ and $I_2$, as are all other such sets that occur in this proof, is the union of $O(1)$ many intervals of the form $[s,t]$ where $s,t\in [0,K_1]$. 

To deal with the integral (\ref{eq:airy1}), with $y$ restricted to $I_2$, we start by making the change of variable $y\leftarrow \alpha_1-\theta y$. This leads to

\begin{equation}\label{eq:airy1x}
\frac{1}{(K_1)^j} \int_{I_3} y^j  e^{2\pi i  c \theta (\alpha_1-y)^3+2\pi i \eta_1 \theta(\alpha_1-y)}\int_{0}^1  x^s e^{2\pi  i y x/L} \, dx\, dy\,, 
\end{equation}

\noindent
where  $I_3:=\{y\in [\alpha_1,\alpha_1-\theta K_1]\,:\, |y|<2(s+1)L\}$. By another change of variable, $x\leftarrow 2(s+1)x$, applied to the integral with respect to $x$ in (\ref{eq:airy1x}), followed by dividing the resulting interval of integration into $2(s+1)$ consecutive intervals $[n,n+1)$, we can reduce (\ref{eq:airy1x}), via a standard application of Taylor expansions, to a linear combination, with quickly computable coefficients each of size $O(1)$, of $O(\nu(K,j,\epsilon))$ integrals of the form

\begin{equation}\label{eq:airy1xx}
\frac{1}{(K_1)^j\,(2s+2)^s} \int_{I_3} y^j  e^{2\pi i  c \theta (\alpha_1-y)^3+2\pi i \eta_1 \theta(\alpha_1-y)}\int_n^{n+1}  x^s e^{2\pi  i y x/(2(s+1)L)} \, dx\, dy\,, 
\end{equation}
 
\noindent
plus an error of size $O(\epsilon/K^2)$ say, where $0\le n<2(s+1)$ an integer. 

Since $|y/(2(s+1)L)|<1$ over $I_3$, then by the change of variable $x\leftarrow x-n$, followed by yet another application of Taylor expansions, we can eliminate the cross term $e^{2\pi i x y/(2(s+1)L)}$ in (\ref{eq:airy1xx}) as a polynomial in $xy$ of degree $O(\nu(K,j,\epsilon))$, plus an error of size $O(\epsilon/K^2)$ say. On integrating directly  with respect to $x$, we arrive at a linear combination, with quickly computable coefficients each of size $O(1)$, of $O(\nu(K,j,\epsilon))$ integrals of the form (\ref{eq:airy34}) below.  As we will soon explain, such integrals can be computed efficiently, 

As for the set $y\in I_1$, we integrate directly with respect to $x$ in (\ref{eq:airy1}) to also obtain a linear combination, with quickly computable coefficients each of size $O(1)$, of $s+1$ integrals

\begin{equation} \label{eq:airy2}
\frac{s!\, L^r}{(s+1-r)!\, (K_1)^j\,(2\pi i)^r} \int_{I_1} y^{j}(\alpha_1-\theta y)^{-r} e^{2\pi i  cy^3+ 2\pi i \eta_2 y} \, dy\,, 
\end{equation}

\noindent
where $0\le r \le s+1$ an integer and $\eta_2:=\eta_{2,\alpha_1,\eta_1}$ is a real number satisfying $\eta_2=O(K_1^3)$ that can be computed quickly. Since $|\alpha_1-\theta y|\ge 2(s+1)L$ for all $y\in I_1$, it follows with careful use of Taylor expansions (see the treatment of (\ref{eq:remform2}) in \textsection{4.2}) that the evaluation of (\ref{eq:airy2}) can be reduced to computing a linear combination, with quickly computable coefficients each of size $O(1)$, of $O(\nu(K,j,\epsilon)^3)$ integrals of the form

\begin{equation} \label{eq:airy34}
\frac{1}{\Delta^u} \int_{0}^{\Delta}  y^{u} e^{2\pi i  c y^3+2\pi i \eta_4 y^2+2\pi i \eta_3 y} \, dy\,, 
\end{equation}

\noindent
where $0\le u=O(\nu(K,j,\epsilon))$ an integer, $\Delta=O(K_1)$ a real number, $|c \Delta^2| <1/\Lambda(K,j,\epsilon)$ say, where $c$ is a real number,  $\eta_3=O(K_1^3)$ a real number, and $\eta_4=O(1)$ a real number. The integral (\ref{eq:airy34}) is a variation on the Airy integral; see~\cite{GK} for example. We implicitly showed how to compute it efficiently in \textsection{4.3}. Briefly though, one considers two cases: either the integrand has a saddle-point or it does not; that is, either \mbox{$3cy^2+2\eta_4 y+\eta_3$} has a zero in $[0,\Delta]$ or it does not. In the former case, we extract the saddle-point contribution like is done in \textsection{4.3}. This reduces the problem to evaluating an integral of the form (\ref{eq:airy34}), but with no saddle-point. And in the latter case (the no saddle-point case), we use Cauchy's theorem to suitably shift the contour of integration (to the stationary phase) so the modulus of the integrand is rapidly decaying, and the interval of integration can be truncated quickly, after distance about $O(\nu(K,j,\epsilon))$. We then divide the truncated interval of integration into $O(\nu(K,j,\epsilon))$ consecutive subintervals $[n,n+1)$. Over each subinterval, we show the integral can be computed efficiently. We mention the procedure for extracting the saddle point is that followed in evaluating the integrals (\ref{eq:ikjj10}) in \textsection{4.3}. We remark the evaluation of integral (\ref{eq:airy34}) essentially reduces to evaluating incomplete Gamma functions like (\ref{eq:igf}). 

Having disposed of the case $\beta<L$ in the integral (\ref{eq:airy00}), we now consider the case $\beta \ge L$ (recall $L:=L(K,j,\epsilon)=\lceil \nu(K,j,\epsilon)\rceil$).  So define

\begin{equation}\label{eq:i4i5}
\begin{split}
&I_4:=\{y\in [0,K_1]\,:\, (\alpha -\theta y)/(2\beta) \in [0,w]\}\,,\\
&I_5:=\{y\in [0,K_1]\,:\, (\alpha -\theta y)/(2\beta) \notin [0,w]\}\,,\\
\end{split}
\end{equation}

\noindent
Let us compute the integral (\ref{eq:airy00}) over the region $(x,y)\in [0,w]\times I_5$ first. We write \mbox{$I_5=I_6\cup I_7$}, where $I_6$ is the subset of $I_5$ where $\alpha-\theta y>2\beta w$ and $I_7$ is the subset where \mbox{$\alpha-\theta y<0$}. We deal with $I_6$ and $I_7$ separately.

Over $(x,y)\in [0,w]\times I_6$, we apply Cauchy's theorem to the integral with respect to $x$ in (\ref{eq:airy00}) to replace the contour $\{x\,:\,0\le x\le w\}$ there with the contours $\{ix\,:\, 0\le x<\infty\}$ and $\{w+ix\,:\,0\le x<\infty\}$, appropriately oriented. On following this by the change of variable $x\leftarrow \sqrt{2\beta} x$, we see the integral (\ref{eq:airy00}), restricted to $(x,y)\in [0,w]\times I_6$, is equal to a linear combination, with quickly computable coefficients each of size $O(1)$,  of the two integrals

\begin{equation} \label{eq:airy4}
\begin{split}
&\frac{1}{(K_1)^j} \int_{I_6} y^j e^{2\pi i  c y^3+2\pi i \eta y}\frac{1}{\sqrt{2\beta}} \int_{0}^{\infty} e^{-2\pi  \tau_1(y) x+\pi i x^2}\, dx\, dy\,,   \\
&\frac{1}{(K_1)^j} \int_{I_6} y^j e^{2\pi i  c y^3+2\pi i \tilde{\eta} y}\frac{1}{\sqrt{2\beta}} \int_{0}^{\infty} e^{-2\pi \tau_2(y) x+\pi i x^2 } \, dx\, dy \,,
\end{split}
\end{equation}

\noindent
where $\tilde{\eta}:=\tilde{\eta}_{\theta,w,\eta}$ is a real number satisfying $\tilde{\eta}=O(\eta+1)$, and

\begin{equation}\label{eq:tau1212}
\tau_1(y):=\tau_{1,\alpha,\theta,\beta}(y)=\frac{\alpha-\theta y}{\sqrt{2\beta}},\qquad \tau_2(y):=\tau_{2,\alpha,\theta,\beta,w}(y)=\frac{\alpha-\theta y-2 \beta w}{\sqrt{2\beta}}\,. 
\end{equation}

We start by showing how to compute the second integral in (\ref{eq:airy4}) efficiently. By the definitions of $\tau_2(y)$ and $I_6$, we have $\tau_2(y)\ge 0$ over $I_6$. So we can use Cauchy's theorem to shift the contour of integration in the inner integral by an angle of $\pi/4$. This transforms the second integral in (\ref{eq:airy4}) to:

\begin{equation} \label{eq:airy5}
\frac{1}{(K_1)^j} \int_{I_6} y^j e^{2\pi i  c y^3+2\pi i \tilde{\eta} y} \frac{e^{\pi i /4}}{\sqrt{2\beta}}\int_{0}^{\infty} e^{-2\pi e^{\pi i /4}\tau_2(y) x-\pi  x^2} \, dx\,. 
\end{equation}

\noindent
We invoke our standard argument where we exploit the exponential decay in the modulus of the integrand to truncate the interval of integration with respect to $x$ in (\ref{eq:airy5}) after distance about $O(\nu(K,j,\epsilon)$ (which results in small enough truncation error of size $O(\epsilon/K^2)$), then divide the truncated interval into $O(\nu(K,j,\epsilon))$ consecutive subintervals $[n,n+1)$, and deal with one subinterval at a time. Over each subinterval, one considers the regions determined by $\tau_2(y)\le L$ and $\tau_2(y)\ge L$ separately. One obtains, with some labor, that the evaluation of the integral (\ref{eq:airy5}) can be reduced to evaluating a linear combination of $O(\nu(K,j,\epsilon)^3)$ of Airy integrals of the form (\ref{eq:airy34}). (Notice over the region determined by $\tau_2(y)>L$, we encounter integrals like (\ref{eq:airy2}).) 

As for the first integral in (\ref{eq:airy4}), its computation is even easier because $\tau_1(y)\ge \sqrt{2\beta}w$ for $y\in I_6$ which ensures faster decay. And the evaluation of the integral (\ref{eq:airy00}), restricted to $(x,y)\in [0,w]\times I_7$, is similar to the case $(x,y)\in [0,w]\times I_6$ already considered. 

So it remains to deal with (\ref{eq:airy00}) when restricted to the region $(x,y)\in [0,w]\times I_4$.  There, we use Cauchy's theorem to replace the contour $\{x\,:\,0\le x\le w\}$ in the integral with respect to $x$ in (\ref{eq:airy00}) with the contours $\{w-ix\,:\, 0\le x\le w\}$ and $\{xe^{-\pi i/4}\,:\, 0\le x\le \sqrt{2}w\}$, appropriately oriented. On following this by the change of variable $x\leftarrow \sqrt{2\beta} x$, we see the integral  (\ref{eq:airy00}), restricted to $(x,y)\in [0,w]\times I_4$,  is equal to a linear combination of the two integrals

\begin{equation} \label{eq:airy10}
\begin{split}
&\frac{1}{(K_1)^j} \int_{I_4}y^j e^{2\pi i  c y^3+2\pi i \eta y} \frac{1}{\sqrt{2\beta}} \int_{0}^{2\sqrt{\beta}w} e^{2\pi i e^{-\pi i/4} \tau_1(y) x-\pi  x^2} \, dx\, dy\,,\\
&\frac{1}{(K_1)^j}  \int_{I_4}y^j e^{2\pi i  c y^3+2\pi i \eta y}  \frac{1}{\sqrt{2\beta}} \int_{0}^{\sqrt{2\beta}w}  e^{2\pi  \tau_2(y) x+\pi i x^2} \, dx\, dy\,.
\end{split}
\end{equation}

\noindent
where the coefficients of the linear combination are quickly computable, and are of size $O(1)$ each. By the definition of $\tau_2(y)$ in (\ref{eq:tau1212}), and the definition of $I_4$ in (\ref{eq:i4i5}), we have $\tau_2(y)\le 0$ for $y\in I_4$. So the second integral in (\ref{eq:airy10}) can be evaluated efficiently in an essentially similar way to the second integral in (\ref{eq:airy4}); that is, we use Cauchy's theorem to shift the contour of integration by an angle of $\pi/4$, then we exploit the guaranteed exponential decay thus obtained. 

As for the first integral in (\ref{eq:airy10}), we have $\tau_1(y)\ge 0$, which means there is possibly some exponential growth with $x$ in the linear factor $e^{2\pi  i e^{-\pi i /4}\tau_1(y) x}$ there. To deal with this, we consider two cases: $\sqrt{\beta}w \le 1$ and $1<\sqrt{\beta}w$. The treatment of the case $\sqrt{\beta}w \le 1$ is particularly simple since by a direct use of Taylor expansions, the integrand is reduced to a polynomial in $\tau_1(y)$ and $x$ of degree bounded by $O(\nu(K,j,\epsilon))$ in each, which, on integrating with respect to $x$, leads to integrals of the type (\ref{eq:airy34}). So suppose $\sqrt{\beta}w> 1$. In this case, we consider the ``complementary'' regions $(x,y)\in (-\infty,0]\times I_4$ and $(x,y)\in [2\sqrt{\beta}w,\infty)\times I_4$. Computing the first integral in (\ref{eq:airy10}), but with the region of integration switched to $(x,y)\in (-\infty,0]\times I_4$, is not problematic because the linear factor $e^{2\pi  i e^{-\pi i /4}\tau_1(y) x}$ now provides exponential decay with $x$, and so the treatment coincides with that of integral (\ref{eq:airy5}) earlier.  As for the region $(x,y)\in [2\sqrt{\beta}w,\infty)\times I_4$, we have $\tau_1(y)\le \sqrt{2\beta} w$ there. Combined with $x\ge 2\sqrt{\beta}w$, this yields $\tau_1(y)x/\sqrt{2}-x^2/2\le \sqrt{\beta}w (1-x)$. Since $\sqrt{\beta}w> 1$ by hypothesis, the modulus of the integrand over $(x,y)\in [2\sqrt{\beta}w,\infty)\times I_4$ declines faster than $e^{-x}$ with $x$. So once again our standard argument exploiting such exponential decay applies, and leads to integrals of the type (\ref{eq:airy34}). It only remains to calculate 

\begin{equation} \label{eq:airy12}
\begin{split}
&\frac{1}{\sqrt{2\beta}\, (K_1)^j} \int_{I_4}y^j e^{2\pi i  c y^3+2\pi i \eta y} \int_{-\infty}^{\infty}  e^{2\pi e^{\pi i/4} \tau_1(y) x-\pi  x^2} \, dx\, dy   \\
&=\frac{1}{\sqrt{2\beta}\,(K_1)^j} \int_{I_4}y^j e^{2\pi i  c y^3+2\pi i \eta y+\pi i \tau_1^2(y)} \, dy\,,
\end{split}
\end{equation}

\noindent
But this is also of the form (\ref{eq:airy34}), which we know how to handle efficiently.
\end{proof}

\begin{lem} \label{lem:la7}
There are absolute constants $\kappa_{17}$, $\kappa_{18}$, $A_{24}$, $A_{25}$, and $A_{26}$, such that for any $\epsilon \in (0,e^{-1})$, any integer $j\ge 0$, any positive integers $K$ and $K_1$ satisfying $\Lambda(K,j,\epsilon)< K_1< K$, any  $\alpha \in [1/\Lambda(K,j,\epsilon), K_1]$, any $\alpha_1\in [0, 4 \alpha]$, any real number $\alpha_2$ satisfying $1/K^2\le |\alpha_2| \le \alpha$ and $\alpha_2/\alpha^2\le 10$ say, any $a \in [0,1)$, any $b\in [0,1)$, any real number $c_1$ satisfying \mbox{$|c_1K_1^2|<1/\Lambda(K,j,\epsilon)$}, any positive integer $K_2$ satisfying, if possible, $K_2\le  K_1/\alpha$, and with $C_0=\{x\,:\, 1/4<x<\infty\}$, $\tilde{C_0}=\{x\,:\,-\infty<x<-1/4\}$, the sum

\begin{equation} \label{eq:comp1}
\frac{1}{(K_1)^j} \sum_{k=\lfloor \Lambda^2 \rfloor}^{K_2-\lfloor\Lambda^2\rfloor} e^{2\pi i a k+2\pi i b k^2}\int_{0}^{K_1} \int_{C}  y^j e^{2\pi i  c_1 y^3-2\pi i x y} e^{2\pi  i (\alpha k-\alpha_1) x-2\pi i \alpha_2 x^2}  \, dx\, dy \,, 
\end{equation}

\noindent
where  $C\in \{C_0,\tilde{C_0}\}$, can be computed to within $\pm\,A_{24} \,\nu(K,j,\epsilon)^{\kappa_{17}} K^{-2} \epsilon$ using $\le A_{25}\,\nu(K,j,\epsilon)^{\kappa_{18}}$ arithmetic operations on numbers of $\le A_{26}\,\nu(K,j,\epsilon)^2$ bits. 
\end{lem}

\begin{proof}
It suffices to efficiently compute the sum (\ref{eq:comp1}) with $C=C_0$ as the case $C=\tilde{C_0}$ is simply a conjugate case (since $c_1$ and $\alpha_2$ are allowed to assume values in a symmetric interval about 0, and $a$ and $b$ are allowed to be any numbers in $[0,1)$). By the change of variable $x\leftarrow x+1/4$, (\ref{eq:comp1}) is transformed to

\begin{equation} \label{eq:temp012}
\begin{split}
\frac{v_{\alpha_1,\alpha_2}}{(K_1)^j} \sum_{k=\lfloor \Lambda^2\rfloor}^{K_2-\lfloor \Lambda^2\rfloor} & e^{2\pi i a_1 k+2\pi i b k^2}\int_{0}^{\infty} e^{2\pi i \tau_k x-2\pi i \alpha_2 x^2}\,\times \\
& \int_{0}^{K_1} y^j e^{2\pi i  c_1 y^3-\pi i y/2-2\pi i xy} \, dy\, dx\,,
\end{split}
\end{equation}

\noindent
where $a_1:=a_{1,a,\alpha}=a+\alpha/4$, $\tau_k:=\tau_{k,\alpha,\alpha_1,\alpha_2}=\alpha k-\alpha_1- \alpha_2/2$, and $v_{\alpha_1,\alpha_2}$ is quickly computable coefficient of modulus 1. Define the contours

\begin{equation} 
C_1:= \{ye^{-\pi i/6}\,:\,0\le y \le 2K_1/\sqrt{3}\}, \quad C_2:= \{K_1-iy\,:\,0\le y \le K_1/\sqrt{3}\}\,. 
\end{equation}

\noindent
By an application of Cauchy's Theorem, the contour $\{y\,:\,0\le y\le K_1\}$ in the integral with respect to $y$ in (\ref{eq:temp012}), can be replaced with the contours $C_1$ and $C_2$, appropriately oriented. Over $C_1$, we obtain

\begin{equation}\label{eq:lastlemadd}
\begin{split}
\frac{\tilde{v}_{\alpha_1,\alpha_2,j}}{(K_1)^j} \sum_{k=\lfloor \Lambda^2\rfloor}^{K_2-\lfloor \Lambda^2\rfloor} & e^{2\pi i a_1 k+2\pi i b k^2}\int_{0}^{\infty} e^{2\pi i \tau_k x-2\pi i \alpha_2 x^2} \,\times \\
& \int_{0}^{\frac{2K_1}{\sqrt{3}}} y^j e^{2\pi  c_1 y^3-\pi e^{\pi i /3}  y/2-2\pi e^{\pi i /3} xy}\, dy\, dx \,. 
\end{split}
\end{equation}

\noindent
where $\tilde{v}_{\alpha_1,\alpha_2,j}$ is a quickly computable coefficient of modulus 1. Since $|c_1 K_1^2|\le 1/\Lambda(K,j,\epsilon)$ by hypothesis, the integrand in (\ref{eq:lastlemadd}) declines faster than $e^{-y/16}$ in absolute value throughout $0 \le y\le 2K_1/\sqrt{3}$. Therefore, by our standard argument of truncating the interval of integration with respect to $y$ at say $L:=L(K,j,\epsilon)=\lceil 16\nu(K,j,\epsilon) \rceil$, subdividing it into $L$ consecutive subintervals $[n,n+1)$, applying the change of variable $y\leftarrow y-n$ in each subinterval, followed by a routine application of Taylor expansions, the expression (\ref{eq:lastlemadd}) is  reduced to a linear combination, with quickly computable coefficients each of size $O(1)$, of $O(\nu(K,j,\epsilon)^2)$ sums of the form

\begin{equation} \label{eq:comp3}
\sum_{k=\lfloor \Lambda^2\rfloor}^{K_2-\lfloor \Lambda^2\rfloor} e^{2\pi i a_1 k+2\pi i b k^2} \int_0^1 y^l \int_{0}^{\infty} e^{ -2\pi e^{\pi i /3} (y+n) x} e^{2\pi  i \tau_k x-2\pi i \alpha_{2} x^2} \, dx\, dy \,,
\end{equation}

\noindent
where $n$ and $l$ are integers satisfying $0\le n <L$ and $0\le l=O(\nu(K,j,\epsilon))$, plus an error of size $O(\epsilon/K^2)$. Now define

\begin{equation}
C_3:= \{e^{- \pi i /4} x| 0 <x < \infty \}, \qquad C_4:= \{e^{  \pi i /4} x| 0 <x < \infty \} \,. 
\end{equation}

\noindent
If $\alpha_2$ is positive, then 

\begin{equation} 
\left| \int_{0}^{T} e^{ -2\pi e^{\pi i /3} (y+n) (T-ix)} e^{2\pi  i \tau_k (T-ix)-2\pi i \alpha_{2} (T-ix)^2}\, dx\right|\to_{T\to \infty} 0\,.
\end{equation}

\noindent
So via Cauchy's theorem, we can exchange the contour $\{x\,:\, 0\le x<\infty\}$ in the integral with respect to $x$ in (\ref{eq:comp3}) with the contour $C_3$. Similarly, if $\alpha_2$ is negative, we can exchange $\{x\,:\, 0\le x<\infty\}$ for $C_4$. Dealing with the case $\alpha_2$ is positive first, we obtain, after a few rearrangements, that (\ref{eq:comp3}) is equal to

\begin{equation}  \label{eq:comp5}
\sum_{k=\lfloor\Lambda^2\rfloor}^{K_2-\lfloor\Lambda^2\rfloor} e^{2\pi i a_1 k+2\pi i b k^2}\int_0^1 y^l \int_0^{\infty}  e^{2\pi  (e^{ \pi i /4}  \tau_k-e^{\pi i /12} (y+n)) x-2\pi  |\alpha_2| x^2} \, dy\, dx\,. 
\end{equation}

\noindent
Since by hypothesis  $0\le \alpha_1\le 4\alpha$ and $|\alpha_2|\le \alpha$, it follows from the definition $\tau_k:=\alpha k-\alpha_1- \alpha_2/2$ that for $k\ge \lfloor \Lambda(K,j,\epsilon)^2\rfloor$ we have 

\begin{equation}\label{eq:bound0122}
\tau_k \ge \alpha(\Lambda(K,j,\epsilon)^2-6)\,.
\end{equation}

\noindent
And since $\alpha\ge 1/\Lambda(K,j,\epsilon)$ by hypothesis, then (\ref{eq:bound0122}) implies $\tau_k \ge \Lambda(K,j,\epsilon)-1$, say. Hence, for $0\le y\le 1$, $\lfloor \Lambda(K,j,\epsilon)^2\rfloor \le k$, and $0\le n<L$,  we have

\begin{equation} \label{eq:bound012}
\Re \{ e^{\pi i /4} \tau_k -e^{\pi i /12} (y+n) \}\ge \tau_k/2-n-1\ge \Lambda(K,j,\epsilon)/3\,,
\end{equation}

\noindent
which is large. So initially there is some exponential growth with $x$ in the size of the integrand in (\ref{eq:comp5}), which is problematic. As usual though, we handle the situation by writing the integral with respect to $x$ in (\ref{eq:comp5}) as the difference of two integrals, one over $-\infty<x<\infty$ and another over $-\infty<x\le 0$. Starting with the latter, we have by (\ref{eq:bound012}) that the integrand declines faster than $e^{-\Lambda(K,j,\epsilon) |x|/3}$ in absolute value with $x$, hence, it can be truncated at $x=-1$. The resulting truncation error is bounded by $O(e^{-\Lambda(K,j,\epsilon)/3})=O(\epsilon/K^2)$, which is small enough for purposes of the lemma. Since now $|xy|\le 1$, then by a routine application of Taylor expansions, the cross-term $e^{-2\pi e^{\pi i /12} y x}$ in the integrand in (\ref{eq:comp5}) can be expressed as a polynomial in $yx$ of degree bounded by $O(\nu(K,j,\epsilon))$, plus an error of size $O(\epsilon/K^2)$ say. So integrating the resulting integral explicitly with respect to $y$, this procedure  yields a linear combination, with quickly computable coefficients each of size $O(1)$, of $O(\nu(K,j,\epsilon))$ integrals   

\begin{equation}  \label{eq:temp022}
\sum_{k=\lfloor\Lambda^2\rfloor}^{K_2-\lfloor\Lambda^2\rfloor} e^{2\pi i a_1 k+2\pi i b k^2} \int_{0}^{1} x^p  e^{-2\pi ( e^{\pi i /4}\tau_k-e^{\pi i /12} n) x-2\pi  |\alpha_2| x^2} \, dx\,, 
\end{equation}

\noindent
where $p$ is an integer satisfying $0\le p=O(\nu(K,j,\epsilon))$, plus an error of size $O(\epsilon/K^2)$. 

With the aid of the bound (\ref{eq:bound0122}) and $\alpha\ge 1/\Lambda(K,j,\epsilon)$ (by hypothesis),  the integral (\ref{eq:temp022}) can be truncated at $\min\{1,\alpha^{-1}\}$ with a truncation error bounded by $O(\epsilon/K^2)$ say. Once truncated, the quadratic factor $e^{-2\pi |\alpha_2|x^2}$ can be reduced, via Taylor expansions and  by appealing to the bound $\alpha_2/\alpha^2=O(1)$, to a polynomial in $x$ of degree at most $O(\nu(K,j,\epsilon))$, plus an error of size $O(\epsilon/K^2)$. Last, we evaluate the resulting integral (an incomplete Gamma function) explicitly. This, combined with a few further algebraic manipulations, yields quadratic exponential sums of the type discussed in \textsection{5} of~\cite{Hi}, and these sums can be computed efficiently via Theorem~\ref{quadraticthm}. 

So it only remains to calculate

\begin{equation}  \label{eq:comp7}
\sum_{k=\lfloor\Lambda^2\rfloor}^{K_2-\lfloor\Lambda^2\rfloor} e^{2\pi i a_1 k+2\pi i b k^2}\int_0^1 y^l \int_{-\infty}^{\infty} e^{2\pi  (e^{\pi i /4} \tau_k-e^{\pi i /12} (y+n)) x-2\pi  \alpha_2 x^2} \, dy\, dx\,, 
\end{equation}

\noindent
Integrating explicitly with respect to $x$ produces 

\begin{equation}  \label{eq:comp8}
\frac{1}{\sqrt{2\alpha_2}}\sum_{k=\lfloor\Lambda^2\rfloor}^{K_2-\lfloor\Lambda^2\rfloor} e^{2\pi i a_2 k+2\pi i b_1 k^2}\int_{0}^{1} y^l  e^{ -2 \pi  \frac{ e^{\pi i /3} (y+n) \tau_k}{2\alpha_2}   + 2\pi   \frac{ e^{\pi i /6} (y+n)^2}{4\alpha_2} } \, dy\,,
\end{equation}

\noindent
where $a_2:=a_{2,a_1,\alpha,\alpha_1,\alpha_2}$ and $b_1:=b_{1,b,\alpha,\alpha_1,\alpha_2}$ are real numbers satisfying $a_2=O(\alpha^2/\alpha_2)=O(K^4)$ and $b_1=O(\alpha^2/\alpha_2)=O(K^4)$. Certainly, $a_2$ and $b_1$ can be reduced modulo 1,  but we point out  their magnitude pre-reduction modulo 1 to show they can be expressed using $O(\nu(K,j,\epsilon)^2)$ bits throughout, as do all other numbers in this article. The integrand  in (\ref{eq:comp8})  declines very rapidly with $y$. In fact, it is of size $O(\epsilon/K^2)$ say if $n>0$, which is negligible for our purposes. So we may assume $n=0$. 

We truncate the integral with respect to $y$ at $\alpha_2/\alpha$ (notice $\alpha_2/\alpha \le 1$ by hypothesis). Since $y^2/\alpha_2=O(\alpha_2/\alpha^2)$ for $0\le y\le \alpha_2/\alpha$, and since $\alpha_2/\alpha^2=O(1)$ by hypothesis, we can use Taylor expansions to express the term $\exp(\pi  e^{\pi i /6} y^2/(2\alpha_2))$ as a polynomial in $y$ of degree bounded by $O(\nu(K,j,\epsilon))$, plus an error of size $O(\epsilon/K^2)$. Integrating explicitly with respect to $y$, along with a few further manipulations, reduces the problem once again to computing a linear combination, with quickly computable coefficients each of size $O(1)$, of $O(\nu(K,j,\epsilon))$ quadratic exponential sums of the type discussed in \textsection{5} in~\cite{Hi}.

This concludes our computation of (\ref{eq:temp012}) when the contour $\{y\,:\,0\le y\le K_1\}$ there is replaced by $C_1$, and assuming $\alpha_2$ is positive. The situation when $\alpha_2$ is negative, where $C_4$ is used instead of $C_3$, is even easier because we essentially obtain an integral of the form (\ref{eq:comp5}), but with contour $\{x\,:\,-\infty<x\le 0\}$, over which there is (extremely) rapid decay with $x$.

Having shown how to deal with (\ref{eq:temp012}) when the contour of integration with respect to $y$ is $C_1$, we move on to $C_2$. In this case, the integral is

\begin{equation}  \label{eq:comp10}
\begin{split}
\frac{\tilde{v}_{K_1,c_1}}{(K_1)^j} \sum_{k=\lfloor\Lambda^2\rfloor}^{K_2-\lfloor\Lambda^2\rfloor} & e^{2\pi i a_1 k+2\pi i b k^2} \int_{0}^{\infty} e^{-2\pi i \tilde{\tau}_k  x-2\pi i \alpha_{2} x^2} \,\times\\
& \int_{0}^{\frac{K_1}{\sqrt{3}}} (K_1-iy)^j  e^{-2\pi c_1 y^3-6\pi i c_1K_1 y^2-2\pi (1/4+x-3c_1K_1^2)y} \, dy\, dx \,, 
\end{split}
\end{equation}

\noindent
where $\tilde{\tau}_k:=\tau_{K_1,\tau_k}= K_1 -\tau_k$, and $\tilde{v}_{K_1,c_1}$ is a quickly computable constant of modulus 1. Since $|3c_1K_1^2|\le 3/\Lambda(K,j,\epsilon)$, the integral with respect to $y$ declines faster than $e^{-y/16}$ throughout $0\le y\le K_1/\sqrt{3}$.  Employing our standard argument of exploiting the exponential decay, the expression (\ref{eq:comp10}) can be reduced to a linear combination, with quickly computable coefficients each of size $O(1)$, of $O(\nu(K,j,\epsilon)^2)$ sums 

\begin{equation}  \label{eq:comp11}
\sum_{k=\lfloor\Lambda^2\rfloor}^{K_2-\lfloor\Lambda^2\rfloor} e^{\pi i a_1 k+2\pi i b k^2} \int_0^1 y^l \int_0^{\infty} e^{-2\pi (y+n)x -2\pi i \tilde{\tau}_k  x-2\pi i \alpha_{2} x^2} \, dy\, dx\,, 
\end{equation}

\noindent
where the integers $n$ and $l$ satisfy $0\le n <L$ and $0\le l =O(\nu(K,j,\epsilon))$. Finally, (\ref{eq:comp11}) is treated in a completely analogous way to (\ref{eq:comp3}) except now, instead of the bound (\ref{eq:bound0122}),  one appeals to the bound $\tilde{\tau}_k \ge \alpha (\Lambda(K,j,\epsilon)^2-6)$, which holds for $0\le k\le K_2-\lfloor \Lambda(K,j,\epsilon)^2\rfloor$, hence, holds over the range of summation in (\ref{eq:comp11}).  
\end{proof}

\textbf{Acknowledgment.} I would like to thank my PhD thesis advisor Andrew Odlyzko. Without his help and comments this paper would not have been possible.

\end{document}